\documentclass[a4paper,11pt]{article}
\usepackage{graphics,color}
\usepackage{amsmath}
\usepackage{bbm}
\usepackage{amssymb}
\usepackage{mathrsfs}
\usepackage{esint}
\usepackage{cancel}
\usepackage{cite}
\usepackage{verbatim}
\usepackage{float}
\usepackage{graphicx}
\usepackage{amsthm}
\usepackage{enumitem}
\usepackage{textcomp}
\usepackage{subfig}
\usepackage{wrapfig}
\usepackage{hyperref}
\newcommand{\black}{\color{black}}

\newcommand{\dx}{\, {\rm d}x}

\newcommand{\dz}{\, {\rm d}z}
\newcommand{\ds}{\, {\rm d}s}
\newcommand{\dt}{\, {\rm d}t}

\newcommand{\di}{\, {\rm d}}
\newcommand{\strokedint}{\fint}
\newcommand{\Rin}{R_{\rm in}}
\newcommand{\Rout}{R_{\rm out}}
\newcommand{\Tin}{T_{\rm in}}
\newcommand{\Tout}{T_{\rm out}}
\newif\ifdraft
\draftfalse 

\drafttrue

\def\eps{\varepsilon}

\newcommand{\F}{\mathcal{F}}

\def\L{\mathcal L}

\newcommand{\M}{\mathcal{M}}

\newcommand{\nada}[1]{}

\newcommand{\R}{\mathbb{R}}

\newcommand{\res}{\mathop{\hbox{\vrule height 7pt width 0.5pt depth 0pt
\vrule height 0.5pt width 6pt depth 0pt}}\nolimits}

\numberwithin{equation}{section}
\mathchardef\emptyset="001F

\newtheorem{theorem}{Theorem}[section]
\newtheorem{definition}[theorem]{Definition}
\newtheorem{prop}[theorem]{Proposition}
\newtheorem{cor}[theorem]{Corollary}
\newtheorem{lemma}[theorem]{Lemma}

\theoremstyle{definition}
\newtheorem{remark}[theorem]{Remark}

\title{Wrinkling in the Lamé problem: a $\Gamma$-convergence approach
	}
\author{Roberta Marziani\footnote{ Dipartimento di Ingegneria dell'Informazione e Scienze Matematiche, 53100 Siena, Italy. E-mail: roberta.marziani@unisi.it}
}

\begin{document}

\maketitle
\vspace{1em}

\begin{center}
	{\large
		\emph{
			Dedicated to the memory of Robert V. Kohn,\\
			who suggested this problem and generously shared his time and insight.
	}}
\end{center}
\vspace{2em}
\begin{abstract}
	We study wrinkling patterns in a thin elastic annulus subjected to radial stretching within the framework of the Föppl--von Kármán theory. Building on the analysis of the Lam\'e problem in Bella and Kohn \cite{BeKo16Lame} (see also \cite{BeKo14+}), we investigate the asymptotic regime $h\to0$ and establish a $\Gamma$-convergence result for suitably rescaled energies after subtraction of the relaxed membrane energy. The limiting functional is a scalar convex measure-valued energy coupled with a constraint on the marginal of the limiting measure, describing the distribution of wrinkle frequencies. We also prove existence and qualitative properties of minimizers of the limiting functional.
\end{abstract}

\noindent\textbf{Keywords:}
Wrinkling, thin elastic sheets, $\Gamma$-convergence, pattern formation, Föppl--von Kármán theory.

\medskip

\noindent\textbf{MSC (2020):}
35B27, 49J45, 74K25, 49S05.
\section{Introduction}
Thin elastic sheets subjected to compression tend to buckle rather than compress, thereby giving rise to wrinkle patterns. Understanding the formation and organization of such patterns is a fascinating problem that has attracted considerable attention in both the physics and mathematics communities over the last decades. Beyond providing insight into morphogenetic processes in plant leaves and animal epithelia, wrinkling phenomena also play an important role in the development of small-scale technologies.

Wrinkle patterns are typically highly complex and may develop across multiple length scales. One source of this complexity is the competition between the preferred wavelength $\lambda$, namely the distance between two adjacent peaks, and the preferred director $\vec n$, that is, the direction along which the sheet oscillates. The wavelength $\lambda$ is a microscopic quantity determined by the local bending rigidity and by the stiffness of a possible substrate attached to the film, whereas $\vec n$ is a macroscopic object governed by the lateral forces acting on the sheet.

From a mathematical perspective, wrinkle patterns may be interpreted as minimizers, or approximate minimizers, of suitable elastic energies. 

 In this work we adopt this variational viewpoint and study a model describing the deformation of a thin annular sheet subjected to radial stretching and allowing for the formation of wrinkles, motivated by physical experiments such as those in \cite{DaShCe12,GeBeMe04,Huangetal}.

More precisely, given parameters $0<R_{\rm in}<R_{\rm out}$ and $0<h\ll R_{\rm out}$, we consider the annular domain
\[
\Omega:=\{x\in\R^2\colon \Rin<|x|<\Rout\},
\]
representing the mid-surface of a thin elastic sheet of thickness $h$. We assume radial dead loads of magnitudes $T_{\rm in}>T_{\rm out}$, satisfying \eqref{eq:hyp}, acting on the inner and outer boundaries, respectively (see Figure \ref{fig1}). 
\begin{figure}
	\centering
	\def\svgwidth{0.35\textwidth}
\begingroup%
  \makeatletter%
  \providecommand\color[2][]{%
    \errmessage{(Inkscape) Color is used for the text in Inkscape, but the package 'color.sty' is not loaded}%
    \renewcommand\color[2][]{}%
  }%
  \providecommand\transparent[1]{%
    \errmessage{(Inkscape) Transparency is used (non-zero) for the text in Inkscape, but the package 'transparent.sty' is not loaded}%
    \renewcommand\transparent[1]{}%
  }%
  \providecommand\rotatebox[2]{#2}%
  \newcommand*\fsize{\dimexpr\f@size pt\relax}%
  \newcommand*\lineheight[1]{\fontsize{\fsize}{#1\fsize}\selectfont}%
  \ifx\svgwidth\undefined%
    \setlength{\unitlength}{768.02669015bp}%
    \ifx\svgscale\undefined%
      \relax%
    \else%
      \setlength{\unitlength}{\unitlength * \real{\svgscale}}%
    \fi%
  \else%
    \setlength{\unitlength}{\svgwidth}%
  \fi%
  \global\let\svgwidth\undefined%
  \global\let\svgscale\undefined%
  \makeatother%
  \begin{picture}(1,0.9805268)%
    \lineheight{1}%
    \setlength\tabcolsep{0pt}%
    \put(0,0){\includegraphics[width=\unitlength,page=1]{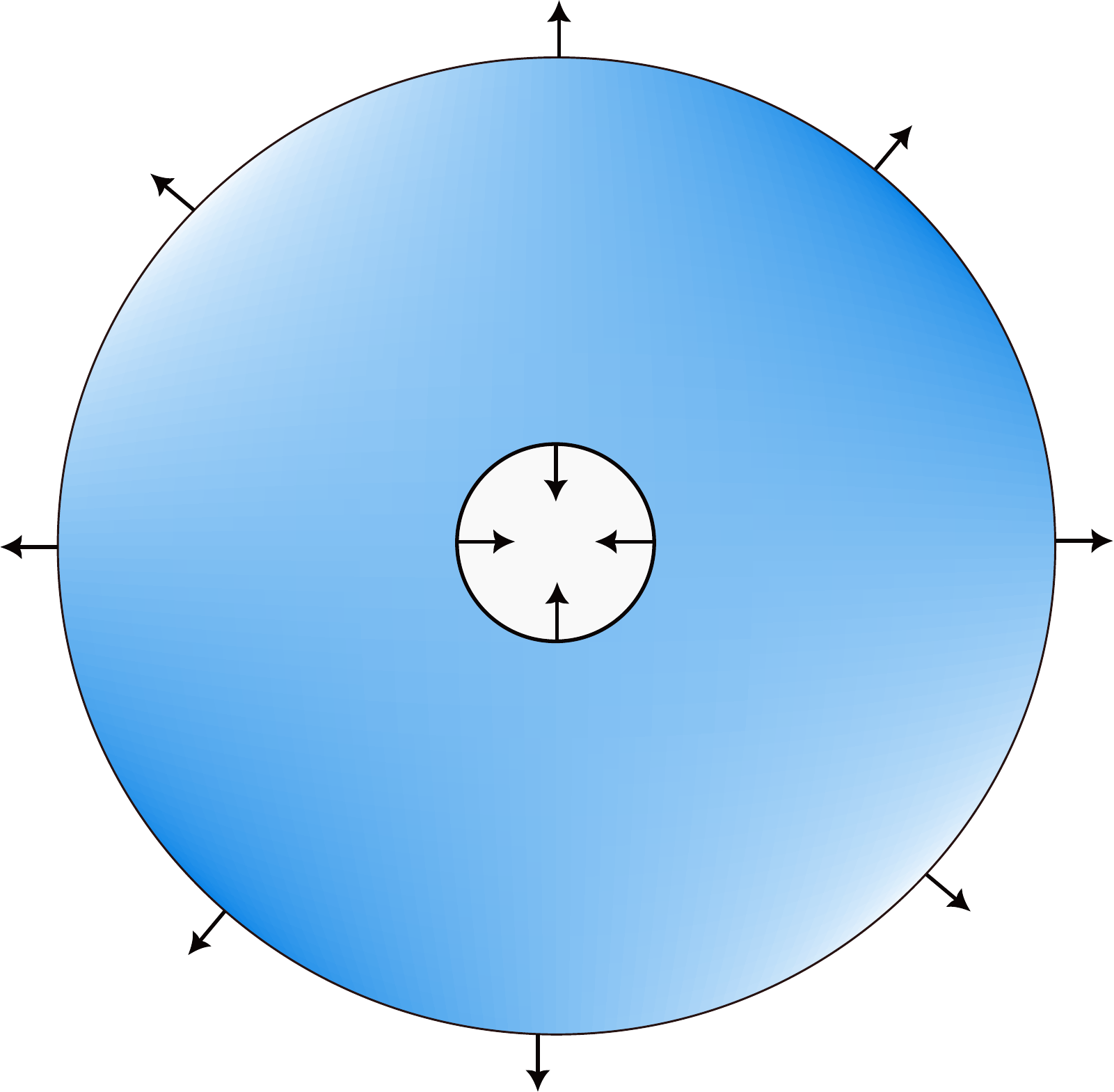}}%
    \put(0,0){\includegraphics[width=\unitlength,page=2]{anello.pdf}}%
      \put(0.46005078,0.47597175){\color[rgb]{0,0,0}\makebox(0,0)[lt]{\lineheight{1.25}\smash{\begin{tabular}[t]{l}\small \text{$T_{\rm in}$}\end{tabular}}}}%
        \put(0.86005078,0.82597175){\color[rgb]{0,0,0}\makebox(0,0)[lt]{\lineheight{1.25}\smash{\begin{tabular}[t]{l}\small \text{$T_{\rm out}$}\end{tabular}}}}%
        \put(0.65005078,0.65597175){\color[rgb]{0,0,0}\makebox(0,0)[lt]{\lineheight{1.25}\smash{\begin{tabular}[t]{l}\small \text{$\Omega$}\end{tabular}}}}%
  \end{picture}%
\endgroup%

\caption{Elastic annular membrane subjected to radial stretching by dead loads of magnitude $\Tin>\Tout$ applied on the inner and outer boundaries, respectively.}\label{fig1}
\end{figure}
To describe the deformation we introduce the in-plane and out-of-plane displacements
\[
u=(u_1,u_2)\colon\Omega\to\R^2,
\qquad
\xi\colon\Omega\to\R.
\]

To each pair $(u,\xi)$ we associate the Föppl--von Kármán energy with boundary loads
\begin{equation}\label{energy1}
	\begin{split}
		E_h(u,\xi):=
		&\frac12\int_\Omega
		\Big(
		|e(u)+\tfrac12\nabla\xi\otimes\nabla\xi|^2
		+
		h^2|\nabla^2\xi|^2
		\Big)\dx
		\\
		&+
		T_{\rm in}\int_{\Gamma^{\rm in}}
		u(x)\cdot\frac{x}{|x|}\di\sigma
		-
		T_{\rm out}\int_{\Gamma^{\rm out}}
		u(x)\cdot\frac{x}{|x|}\di\sigma,
	\end{split}
\end{equation}
where
\[
e(u):=\frac{\nabla u+\nabla^Tu}{2}
\]
denotes the symmetric part of the gradient, and
\[
\Gamma^{\rm in}:=\{|x|=R_{\rm in}\},
\qquad
\Gamma^{\rm out}:=\{|x|=R_{\rm out}\}.
\]
The first term in \eqref{energy1} corresponds to the membrane energy, accounting for stretching and compression, while the second term represents the bending energy and penalizes out-of-plane curvature. The boundary terms determine the radial stresses and therefore influence the onset and orientation of wrinkling patterns.

The above model may be viewed as a linearized version of the fully nonlinear elastic model studied in \cite{BeKo14+}, although a complete rigorous derivation of \eqref{energy1} from three-dimensional elasticity remains open (see, for instance, \cite{Muller}).

A key feature of the Lam\'e problem is that purely radial loading generates azimuthal compression. Indeed, when stronger radial loads are applied at the inner boundary, concentric material circles are forced to move closer to the center. This creates an excess of arclength in the azimuthal direction, which must either be stored through compression or released through the formation of wrinkles. Near the outer boundary, on the other hand, the material remains under tension and no wrinkling occurs. Equivalently, the amount of ``excess arclength'' is positive close to the inner boundary and negative near the outer boundary, and therefore necessarily vanishes along an intermediate free boundary separating the wrinkled and tensile regions. In \cite{BeKo14+,BeKo16Lame}, this excess arclength is characterized as the minimizer of a one-dimensional variational problem and is, in particular, nondegenerate near its zero set.

A central question is to understand the asymptotic behavior of minimizers of $E_h$ as $h\to0$: what is the minimum energy, how is compression released, and what geometric features characterize optimal wrinkle patterns?

Experiments suggest that minimizers of \eqref{energy1} exhibit pure stretching in the outer region of $\Omega$ and wrinkling in the inner region. More precisely, there exists a free boundary of radius $\Rin<R_0<\Rout$ separating the tensile region from the wrinkled one (see Figure \ref{fig2}). 
\begin{figure}
	\centering
	\def\svgwidth{0.35\textwidth}
\begingroup%
  \makeatletter%
  \providecommand\color[2][]{%
    \errmessage{(Inkscape) Color is used for the text in Inkscape, but the package 'color.sty' is not loaded}%
    \renewcommand\color[2][]{}%
  }%
  \providecommand\transparent[1]{%
    \errmessage{(Inkscape) Transparency is used (non-zero) for the text in Inkscape, but the package 'transparent.sty' is not loaded}%
    \renewcommand\transparent[1]{}%
  }%
  \providecommand\rotatebox[2]{#2}%
  \newcommand*\fsize{\dimexpr\f@size pt\relax}%
  \newcommand*\lineheight[1]{\fontsize{\fsize}{#1\fsize}\selectfont}%
  \ifx\svgwidth\undefined%
    \setlength{\unitlength}{689.49126957bp}%
    \ifx\svgscale\undefined%
      \relax%
    \else%
      \setlength{\unitlength}{\unitlength * \real{\svgscale}}%
    \fi%
  \else%
    \setlength{\unitlength}{\svgwidth}%
  \fi%
  \global\let\svgwidth\undefined%
  \global\let\svgscale\undefined%
  \makeatother%
  \begin{picture}(1,0.97934626)%
    \lineheight{1}%
    \setlength\tabcolsep{0pt}%
    \put(0,0){\includegraphics[width=\unitlength,page=1]{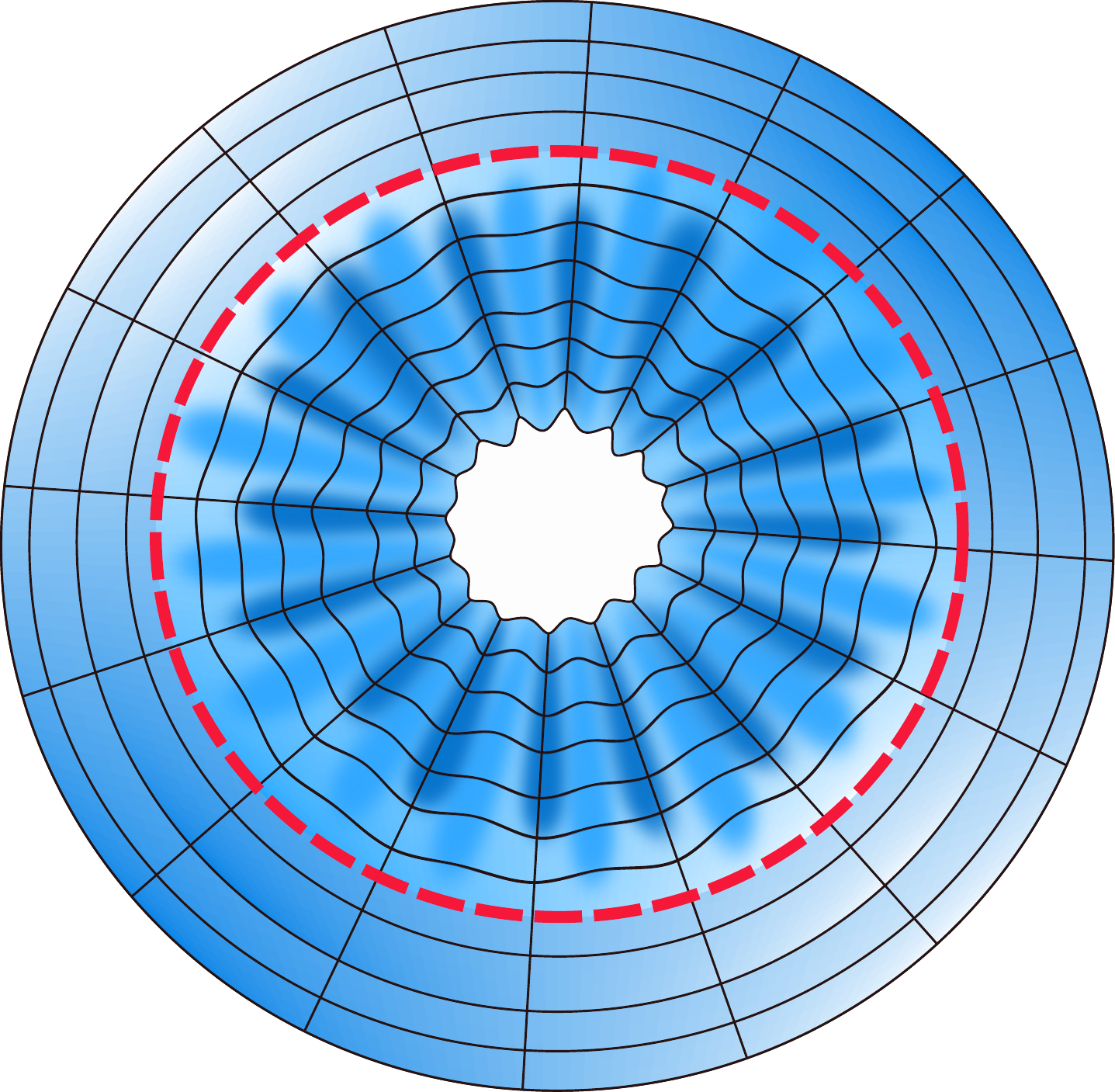}}%
     \put(0.88005078,0.494597175){\color[rgb]{0,0,0}\makebox(0,0)[lt]{\lineheight{1.25}\smash{\begin{tabular}[t]{l}\small \text{$\color{red}R_0$}\end{tabular}}}}%
  \end{picture}%
\endgroup%

	\caption{The red dotted curve of radius $R_0$ marks the transition between the inner wrinkled (relaxed) region and the outer stretched (unrelaxed) region.}\label{fig2}
\end{figure}
In connection with this picture, it was shown in \cite{BeKo16Lame} (and in \cite{BeKo14+} for the nonlinear analogue) that
\begin{equation}\label{energy-scaling}
	\mathcal{E}_0+C_0h
	\le
	\inf_{(u,\xi)}E_h(u,\xi)
	\le
	\mathcal{E}_0+C_1h,
\end{equation}
for some constant $\mathcal{E}_0\in\R$ and $0<C_0\le C_1$.

The quantity $\mathcal E_0$ represents the minimum value of the relaxed elastic energy in the limit $h\to0$, where compressive strains may be released at no energetic cost and only stretching contributes to the energy.  The correction of order $h$ in \eqref{energy-scaling} represents the leading energetic cost associated with the formation of wrinkles, resulting from the competition between bending and stretching effects.

The proof of the upper bound in \eqref{energy-scaling} relies on a construction in which a cascade of wrinkles is superimposed onto the relaxed solution through a branching mechanism in order to release compression in the inner region $\Rin<r<R_0$. However, such branching patterns are not typically observed experimentally. By contrast, the ansatz proposed in the physics literature consists of a single family of wrinkles with slowly varying wavelength and yields a higher energetic cost, namely a correction of order $h(|\log h|+1)$ (cf.~\cite{DaSh++11}).

This discrepancy reveals that scaling laws alone are insufficient to identify the effective wrinkling mechanism. In particular, they do not distinguish between branching constructions and experimentally observed configurations, nor do they describe how wrinkle frequencies are distributed in the asymptotic regime.

A central goal of this paper is therefore to determine the exact prefactor in the linear correction term in \eqref{energy-scaling} and to characterize the limiting variational problem governing the optimal wrinkling patterns.

To this end, we perform the $\Gamma$-convergence analysis, as $h\to0$, of the rescaled energies
\begin{equation}\label{scaled}
	\frac{E_h(u,\xi)-\mathcal E_0}{h}.
\end{equation}
We also prove existence of minimizers for the limiting functional and establish equipartition properties.\\

 Our analysis builds upon the measure-theoretic approach introduced in \cite{BeMa25} (see also \cite{BellaARMA}) for a related toy model. Although considerably simpler than the original annular setting, the model introduced in \cite{BeMa25} retains its essential geometric features. The idea is to focus on a neighborhood of the free boundary separating the stretched and wrinkled regions and replace it with a thin elastic sheet of thickness $h$ occupying the rectangular domain $[-1,1]^2\subset\R^2$, with the transition region corresponding to $\{0\}\times[-1,1]$. The sheet is stretched in the horizontal direction and subjected to a suitable vertical prestrain, chosen so that the left half of the rectangle is stretched whereas the right half is compressed. In this way, the model captures the geometric need to ``waste arclength'' responsible for wrinkle formation in the original Lam\'e problem.

In \cite{BeMa25} it was proved that the corresponding rescaled energies, of the form \eqref{scaled}, $\Gamma$-converge to a scalar convex functional defined on measures satisfying a suitable constraint. At a heuristic level, the limiting functional emerges by rewriting the energy in terms of the Fourier coefficients of the out-of-plane displacement. As the thickness tends to zero, increasingly many frequencies become energetically relevant and, in the limit, generate a diffuse measure describing the distribution of wrinkle frequencies. The associated constraint encodes the amount of arclength that must be wasted in order to accommodate the wrinkles.

A key feature of the problem is that the excess arclength vanishes linearly near the free boundary between the wrinkled and planar regions. In the toy model this profile is therefore replaced by its first-order Taylor approximation. As will become apparent in the analysis, this linear behaviour is precisely what leads to the appearance of infinitely many frequencies in the limit. Indeed, if the excess arclength vanished more slowly, for instance quadratically, one would expect the optimal configuration to involve only a single frequency. By contrast, in the present regime higher and higher frequencies become energetically relevant as one approaches the transition region.\\

Our model belongs to the class of variational problems exhibiting energy-driven pattern formation, namely the emergence of complex microstructures resulting from the competition between different energetic mechanisms acting at different scales \cite{KohnICM}. Typically, one energetic contribution favors smoothness or uniformity, while another favors oscillations, localization, or phase separation. The interplay between these competing effects gives rise to highly nontrivial patterns characterized by multiple length scales. In many situations, scaling laws determine the dominant energetic regime but do not characterize the geometry of minimizing configurations. A finer asymptotic analysis, often based on $\Gamma$-convergence methods, is therefore required to understand the effective pattern-selection mechanism. In this perspective, the limiting variational problem obtained through $\Gamma$-convergence describes the selection of wrinkle patterns beyond the information provided by scaling laws alone.

A first landmark result in this direction is the seminal work of Kohn and M\"uller \cite{KoMu94} on domain branching in martensitic phase transformations, later extended by Conti \cite{Co00}; see also e.g., \cite{capellaotto1,capellaotto2,RuTr23,RuTr23b,KnuepferKohnOtto} for further developments. Similar multiscale mechanisms arise in several other variational models, including magnetic domain patterns in micromagnetics \cite{Ot02,BrKn23,RiRo23}, superconducting vortex patterns \cite{Se15}, block copolymer microstructures \cite{Ch01},  dislocation networks in crystal plasticity \cite{GMS,CGM}, phase separation \cite{Marz} and gradient damage models \cite{BMZ,BEMZ1,BEMZ2,MS}.

Within the context of thin elastic sheets, wrinkling problems may be broadly divided into two classes, depending on whether compression is present in one or in two directions. In compressive wrinkling problems, such as blistering and delamination \cite{KoNg13,BeKo15,BoCoMu16,PG20,JiSt01,BBCoDeMu00,BBCoDeMu02}, crumpling \cite{CoMa08,Ve04}, and conical singularities \cite{Ol18,Ol19}, the relaxed energy is often highly degenerate: its minimum is zero and is attained by many different configurations. This degeneracy makes the analysis of the next-order expansion of the energy particularly delicate.

 This strong degeneracy sharply contrasts with the structure of tensile wrinkling problems, where the relaxed energy typically admits a distinguished minimizer. In such problems, compression may arise from boundary conditions, as in the raft problem \cite{BrKoNg13,benny-raft}, twisted ribbons \cite{KoOB18}, hanging drapes \cite{romandrapes,BeKo17drapes}, or compressed cylinders \cite{To18}; from prescribed incompatible strains \cite{BeKo14,MaSh19}; or from curvature effects \cite{BeKo17,BeRo23,To21}.

The Lam\'e problem considered in this paper belongs to the latter class. Indeed, stronger radial loads applied at the inner boundary force concentric material circles to move closer to the center, thereby inducing azimuthal compression in an inner region of the annulus, while the outer region remains tensile. Equivalently, the geometry of the deformation creates an excess of arclength that must either be stored through compression or released through wrinkling. The corresponding relaxed problem is governed by a one-dimensional radial minimizer, and the goal of the present work is to identify the next-order energy and the limiting variational problem describing the distribution of wrinkle frequencies.\\


We now describe the structure of the limiting problem and the main ideas underlying the analysis. In order to study the asymptotic regime $h\to0$, we first rewrite the energy \eqref{energy1} in polar coordinates
\[
(u_r,u_\theta,\xi)(r,\theta)\colon (\Rin,\Rout)\times[-\pi,\pi)\to\R^3,
\]
where $u_r$ and $u_\theta$ denote the radial and tangential components of the in-plane displacement. Owing to the radial geometry, it is natural to assume that the displacement is $2\pi$-periodic in the angular variable $\theta$.

The next step consists in rescaling the angular variable according to the characteristic wavelength of the wrinkles. Balancing the stretching and bending contributions suggests that the optimal wavelength is of order $\sqrt h$ (cf.~\cite{BeKo14+,BellaARMA}). We therefore introduce the scaling parameter
\[
L:=h^{-1/2},
\]
and perform the asymptotic analysis in the limit $L\to\infty$. After rescaling, the rescaled displacement becomes $2\pi L$-periodic in $\theta$, and the rescaled energy takes the form
\begin{equation*}\label{scaled1}
	L^2\bigl(\mathcal F_L(u_r,u_\theta,\xi)-\mathcal E_0\bigr),
\end{equation*}
defined on the enlarged domain
\[
(\Rin,R_0)\times[-\pi L,\pi L),
\]
see \eqref{def:F_L} for the precise definition of $\mathcal F_L$.

As $L\to\infty$, the leading contribution to the energy is given by
\[
\pi\fint_{-\pi L}^{\pi L}\int_{\Rin}^{R_0}
\left(\dot u^\star(r)
(\partial_r\xi)^2
+
\frac{(\partial_{\theta\theta}\xi)^2}{r^4}
\right)
r\di r\di\theta,
\]
where $u^\star$ denotes the one-dimensional relaxed solution and the dot denotes differentiation with respect to $r$. The above expression reflects the competition between radial stretching and azimuthal bending of the out-of-plane displacement $\xi$. At the same time, the excess arclength that must be released through wrinkling gives rise to the constraint
\begin{equation*}\label{arclength}
	\fint_{-\pi L}^{\pi L}
	\frac{(\partial_\theta\xi)^2}{2}
	\di\theta
	\sim
	-r u^\star(r)
	\qquad\text{in }(\Rin,R_0).
\end{equation*}

Since the angular domain becomes larger and larger as $L\to\infty$, it is convenient to exploit periodicity and rewrite the problem in terms of Fourier coefficients of $\xi$. Denoting these coefficients by $a_k(r)$, with frequencies $k\in\mathbb Z/L$, the leading-order energy and the constraint formally become
\[
\sum_{k\in\mathbb Z/L}
\int_{\Rin}^{R_0}
\left(
(\dot a_k(r))^2
+
a_k(r)^2k^4
\right)
r\di r,
\]
and
\[
\sum_{k\in\mathbb Z/L}
a_k(r)^2k^2
\sim
-r u^\star(r)
\qquad\text{in }(\Rin,R_0).
\]
As the grid $\mathbb Z/L$ becomes finer and finer, the wrinkle frequencies become asymptotically dense, and the natural limiting framework is therefore measure-theoretic. In order to linearize the constraint, we introduce the measure 
\[
\mu^L(\xi):=
\sum_{k\in\mathbb Z/L}
a_k(r)^2k^2\mathcal{L}^1\res(\Rin,\Rout)\times\delta_k,
\]
which are naturally interpreted as measures describing the distribution of wrinkle frequencies. The limiting functional $\mathcal F_\infty$ is then defined on weak limits of the measures $\mu^L$. Correspondingly, the arclength constraint is encoded in a condition on the marginal of the limiting measure; see Section~\ref{sec:main-result} for the precise formulation.

Our $\Gamma$-convergence result is complemented by a coercivity statement ensuring compactness for sequences with uniformly bounded energy. We next outline the proof strategy and discuss the main analytical challenges.
On the one hand, the $\Gamma$-liminf inequality is relatively direct, as it follows from Reshetnyak's lower semicontinuity theorem for measures. On the other hand, the proof of the $\Gamma$-limsup inequality is considerably more involved, and we summarize its main steps below. The strategy is inspired by the construction developed for the toy model in \cite{BeMa25}, although the radial geometry introduces several additional technical difficulties.

 Given a limiting measure, the basic idea for the upper bound is to discretize it in the $k$-variable (which denotes the Fourier counterpart of $\theta$). While the Fourier variable $k$ is discretized, the radial variable $r$ remains continuous and enters the construction through the weights appearing in the energy and in the arclength constraint. A key role is played by the introduction of an intermediate scale $L_0$, satisfying
 \[
 1\ll L_0\ll L,
 \]
 which determines the period of the building blocks used in the construction. This intermediate scale separates the microscopic wrinkle oscillations from the macroscopic radial variation and allows one to localize the construction while still capturing the limiting measure-valued distribution of frequencies.
 
 One then constructs the Fourier coefficients of the out-of-plane displacement and subsequently recovers the in-plane displacement so that all terms in the energy, except for the leading contribution, vanish in the limit.

Since the energy \eqref{energy1} contains higher-order derivatives, the Fourier coefficients must have sufficient regularity. For this reason, after discretizing the limiting measure, we regularize the coefficients by mollification at a suitably chosen small scale $\varepsilon=\varepsilon(L)$. Since this operation modifies the arclength constraint, it is more convenient to mollify the quadratic coefficients $a_k^2(r)$ rather than the coefficients $a_k(r)$ themselves.

The radial setting introduces additional difficulties. In the toy model, the amount of arclength to be wasted was replaced by its first-order Taylor expansion near the free boundary, which considerably simplifies the analysis. In the present annular problem, however, this quantity is given by
\[
-ru^\star(r)=-T_{\rm in}R_{\rm in}\, r\log\frac r{R_0},
\]
and its nonlinear dependence on $r$ creates further complications. One issue is that, in order to mollify the coefficients, they must first be extended outside the interval $(R_{\rm in},R_0)$, and this extension has to be chosen so as not to significantly alter the prescribed arclength profile.

A second difficulty arises in the construction of the in-plane displacement. Unlike in the toy model, the radial energy depends explicitly on $u_r$ and $u_\theta$, and not only on their derivatives. As a consequence, the two components of the in-plane displacement cannot be chosen independently; rather, they must be correlated in order to make the non-leading terms in the energy vanish in the limit.

Our second main result concerns the structure of minimizers of the limiting functional. In particular, we prove existence of minimizers, establish regularity properties through a disintegration of measures with respect to the frequency variable, and derive an equipartition identity between the two terms appearing in the limiting energy. We also show that minimizers cannot concentrate on arbitrarily small frequencies.

The existence result is obtained via the direct method of the calculus of variations, exploiting the fact that the stretching contribution to the limiting energy can be interpreted as a particular instance of the Benamou--Brenier functional arising in optimal transport theory, and is therefore convex. The equipartition identity is instead proved through a contradiction argument, which in turn implies that minimizers do not have small frequencies. Since these arguments closely follow those developed in \cite{BeMa25}, we omit the proofs here.\\

The paper is organized as follows. In Section~\ref{sec:preliminaries} we introduce the notation and collect several preliminary results. In particular, Subsection~\ref{sec:relaxed-pb} is devoted to the relaxed problem, while in Subsection~\ref{sec:scaling} we introduce the rescaled energy and discuss the main heuristics underlying the analysis.

In Section~\ref{sec:main-result} we introduce the measure-theoretic framework and state the main $\Gamma$-convergence result, Theorem~\ref{theo:main_theo}. Section~\ref{sec:compactness_lwb} contains the compactness result for sequences with uniformly bounded rescaled energy (Proposition~\ref{prop:compactness}) together with the $\Gamma$-liminf inequality (Proposition~\ref{prop:lower_bound}). The construction for the $\Gamma$-limsup inequality is carried out in Section~\ref{sec:upb}, where Proposition~\ref{prop:upb} is proved. Finally, in Section~\ref{sec:regularity} we study existence and qualitative properties of minimizers of the limiting functional (Theorem~\ref{thm:minimizers}).

\section{Preliminaries}\label{sec:preliminaries}

In this section we introduce the notation and preliminary results needed throughout the paper.
\subsection{Notation} \label{sec:notation}
\begin{enumerate}[label=$(\alph*)$]
	\item $a\lesssim b$ denotes $a\le Cb$ for some constant $C>0$;
	\item $\chi_A$ denotes the characteristic function of the set $A$;
	\item $\mathcal{L}^1$ denotes the 1-dimensional Lebesgue measure;
	\item $\delta_k$ denotes the Dirac measure on $k\in\R$;
\item $\mathcal{M}_b(A)$ denotes the space of bounded Radon measures on $A$ with $A\subset\R^2$ Borel measurable;
\item $\mathcal{M}^+_b(A)$ denotes the subspace of  $\mathcal{M}_b(A)$ of positive bounded Radon measures;
\item For a function $f\colon A\subset\R\to\R$ we denote by $\dot f(r)$  and $\ddot{f}(r)$ the first and the second derivative, respectively;
\item For a function $u\colon A\subset\R^2\to\R$ we denote by $\partial_{\small\underbrace{r\dots r}_{i \text{ times}}\small\underbrace{\theta\dots \theta}_{j \text{ times}}}f$ its partial derivative 
$$D^{i+j}f(r,\theta)= \frac{\di^j}{\di\theta^j}\frac{\di^i}{\di r^i}f(r,\theta)\,,\quad i,j\in\mathbb{N},\, 1\le i+j\le3\,;$$
\item For a measure $\mu\in \mathcal{M}_b(A)$ we denote by $\mu_{,r}$ its distributional derivative with respect to the first variable;
\item For a measure $\mu\in \mathcal{M}_b(A)$ we denote by $|\mu|\in \mathcal{M}^+_b(A)$ its total variation;
\item For $\tilde\mu=(\mu_1,\mu_2)\in ( \mathcal{M}_b(A))^2$ we analogously denote by $|\tilde\mu|\in \mathcal{M}^+_b(A)$ its total variation;
\item For $\mu_1\in \mathcal{M}_b(A)$, $\mu_2\in \mathcal{M}^+_b(A)$ we write $\mu_1\ll\mu_2$ if $\mu_1$ is absolute continuous with respect to $\mu_2$ and we indicate by $\frac{\di\mu_1}{\di\mu_2}\in L^1(A,\mu_2)$ the associated density {(Radon-Nikod\'ym derivative)};
\item  $f*g(r)$ denotes the convolution between two functions $f$ and $g$.\black
\end{enumerate}
\subsection{The relaxed problem}\label{sec:relaxed-pb}

We begin by recalling the characterization of the relaxed energy in the limit $h\to0$. In \cite{CoMaMu06} (see also \cite{pipkin1}) it was shown that the functionals in \eqref{energy1} $\Gamma$-converge, as $h\to0$, to the relaxed functional
\begin{equation*}
	\begin{split}
		E_{\rm rel}(u,\xi):=
		&\frac12\int_\Omega
		W_{\rm rel}\Big(
		e(u)+\frac12\nabla\xi\otimes\nabla\xi
		\Big)\dx
		\\
		&+
		T_{\rm in}\int_{\Gamma^{\rm in}}
		u(x)\cdot\frac{x}{|x|}\di\sigma
		-
		T_{\rm out}\int_{\Gamma^{\rm out}}
		u(x)\cdot\frac{x}{|x|}\di\sigma,
	\end{split}
\end{equation*}
where $W_{\rm rel}(F)$ denotes the square of the Frobenius norm of the positive part of $F$, namely
\[
W_{\rm rel}(F)
:=
\inf\{
|F+A|^2
:
A=A^T\ge0
\}
=
(\lambda_1(F))_+^2+(\lambda_2(F))_+^2.
\]
Here $\lambda_1(F),\lambda_2(F)$ denote the eigenvalues of the symmetric matrix $F$, and $\lambda_+:=\max\{\lambda,0\}$.

The minimization over positive semidefinite matrices in the definition of $W_{\rm rel}$ corresponds to the relaxation of compressive strains through arbitrarily fine oscillations. In particular, $W_{\rm rel}$ vanishes on all contractions, reflecting the fact that, as $h\to0$, the bending contribution disappears and compression can be released through increasingly fine wrinkling patterns at negligible energetic cost. Moreover, the density $W_{\rm rel}$ is convex.

\medskip

\noindent\textbf{Passage to polar coordinates.}
Since both the domain and the applied loads are radially symmetric, it is convenient to rewrite the energy in polar coordinates $(r,\theta)\in (R_{\rm in}, R_{\rm out})\times [-\pi,\pi)$. We introduce the orthonormal basis $(e_r,e_\theta,e_\xi)$ given by
\[
e_r:=(\cos\theta,\sin\theta,0),
\qquad
e_\theta:=(-\sin\theta,\cos\theta,0),
\qquad
e_\xi:=(0,0,1).
\]
Under the change of variables
\[
x=x(r,\theta):=(r\cos\theta,r\sin\theta),
\]
we define
\[
(\hat u_1(r,\theta),\hat u_2(r,\theta),\hat\xi(r,\theta))
:=
(u_1(x(r,\theta)),u_2(x(r,\theta)),\xi(x(r,\theta)))
\]
and write
\[
(\hat u_1,\hat u_2,\hat\xi)
=
u_r(r,\theta)e_r
+
u_\theta(r,\theta)e_\theta
+
\xi(r,\theta)e_\xi,
\]
where $u_r$ and $u_\theta$ denote the radial and tangential components of the in-plane displacement.

In these coordinates, the functional $E_{\rm rel}$ becomes
\begin{equation*}
	\begin{split}
		E_{\rm rel}(u_r,u_\theta,\xi)
		=
		&\frac12
		\int_{-\pi}^{\pi}
		\int_{R_{\rm in}}^{R_{\rm out}}
		W_{\rm rel}(F(u_r,u_\theta,\xi))
		\,r\di r\di\theta
		\\
		&+
		T_{\rm in}R_{\rm in}
		\int_{-\pi}^{\pi}
		u_r(R_{\rm in},\theta)\di\theta
		-
		T_{\rm out}R_{\rm out}
		\int_{-\pi}^{\pi}
		u_r(R_{\rm out},\theta)\di\theta,
	\end{split}
\end{equation*}
with
\[
F(u_r,u_\theta,\xi):=
\begin{pmatrix}
	\partial_r u_r+\frac12(\partial_r\xi)^2
	&
	\frac1{2r}\partial_\theta u_r
	+\frac12\partial_r u_\theta
	-\frac1{2r}u_\theta
	+\frac1{2r}\partial_r\xi\partial_\theta\xi
	\\[1em]
	\frac1{2r}\partial_\theta u_r
	+\frac12\partial_r u_\theta
	-\frac1{2r}u_\theta
	+\frac1{2r}\partial_r\xi\partial_\theta\xi
	&
	\frac{u_r}{r}
	+\frac1r\partial_\theta u_\theta
	+\frac1{2r^2}(\partial_\theta\xi)^2
\end{pmatrix}.
\]

Observe that, in the relaxed problem, it is energetically favorable to choose the out-of-plane displacement $\xi=0$. Indeed, the nonlinear contribution
\[
\frac12\nabla\xi\otimes\nabla\xi
\]
is positive semidefinite and therefore cannot decrease the positive part of the eigenvalues appearing in the relaxed density $W_{\rm rel}$. Moreover, the boundary loading depends only on the radial component $u_r$, while nontrivial angular dependence of the in-plane displacement produces additional shear contributions through the off-diagonal term
\[
\frac1{2r}\partial_\theta u_r+\frac12\partial_r u_\theta-\frac1{2r}u_\theta .
\]
This quantity measures the tendency of neighboring material elements to slide tangentially relative to each other. Such distortions are energetically unfavorable since they increase the positive part of the strain. Consequently, minimizers of the relaxed problem may be sought among radially symmetric configurations, namely
\[
u_r=u_r(r),
\qquad
u_\theta=0,
\qquad
\xi=0.
\]

	We next consider the one-dimensional minimization problem
	\begin{equation*}\label{min-relaxed-pb}
		\inf_{u_r(r)}
		E_{\rm rel}(u_r,0,0).
	\end{equation*}

\noindent\textbf{One-dimensional solution to the relaxed problem.} Assume $u_r=v(r)$, then   $$E_{\rm rel}(v,0,0)=F(v)$$ with
\begin{equation*}\label{eq:Erel-radial}
F(v):=2\pi	\left({ \frac12}\int_{R_{\rm in}}^{R_{\rm out}} \Big((\dot v)_+^2+ \Big(\frac{v}{r}\Big)_+^2 \Big)r\di r+ T_{\rm in}R_{\rm in}v(R_{\rm in})-
	T_{\rm out}R_{\rm out}v(R_{\rm out})
	\right)\,.
\end{equation*}
We then look for a solution to the one-dimensional minimization problem
\begin{equation}\label{min-radial-relaxed-pb}
{\rm m_{\rm rad}}:=	\inf_{v\in {W^{1,2}(\Rin,\Rout)}}F(v)\,.
\end{equation}
For the reader's convenience we recall the following result in \cite{BeKo16Lame}.
\begin{lemma}[Existence of a solution to $m_{\rm rad}$]\label{lem:1dim-sol}
	Let $T_{\rm in}> T_{\rm out}>0$, $R_{\rm out}>R_{\rm in}>0$ be 
fixed.

 If
	\begin{equation}\label{eq:hyp}
\Tin\Rin<\Tout\Rout\quad\text{ and }\quad\frac{\Tin}{\Tout}>2\frac{\Rout^2}{\Rin^2+\Rout^2}\,,
	\end{equation}
then \eqref{min-radial-relaxed-pb} has a unique minimizer $u^\star$. Moreover $u^\star$ satisfies the following:
\begin{enumerate}[label=$(\roman*)$]
	\item $u^\star\in W^{1,2}(\Rin,\Rout)\cap C(\Rin,\Rout)$;
\item There exists ${R}_{0}\in (\Rin,\Rout)$ such that
\begin{equation*}\label{sign-u-star}
u^\star<0 \,\text{ in }\, (\Rin,R_0)\quad\text{ and }\quad u^\star>0 \,\text{ in }\, (R_0,\Rout)\,;
\end{equation*}
\item $u^\star$ satisfies the Euler-Lagrange equation
\begin{equation}\label{EL+}
	\begin{cases}
		\dfrac{\di}{\di r}(r\dot u^\star(r))=\Big(\dfrac{u^\star(r)}{r}\Big)_+&\text{ in }(\Rin,\Rout)\,,\\[1em]
		\dot u^\star(\Rin)=\Tin\,,\\[1em]
	\dot u^\star(\Rout)=\Tout\,;
	\end{cases}
\end{equation}
\item For every $r\in (\Rin,\Rout)$
\begin{equation*}\label{sign-u-star-prime}
 \Tin\frac{\Rin}{\Rout}\le \dot u^\star(r)\le \Tout\frac{\Rout}{\Rin} \,;
\end{equation*}
\item For every $r\in (\Rin,R_0)$, the function $\beta(r):=\sqrt{-u^\star(r)r}$ satisfies
\begin{equation*}\label{eq:estimate-beta}
|\beta(r)|\lesssim\sqrt{R_0-r}\,, \quad  |\dot \beta(r)|\lesssim\frac1{\sqrt{R_0-r}}\,, \quad  |\ddot \beta(r)|\lesssim\frac1{(\sqrt{R_0-r})^3}\,.
\end{equation*}
\end{enumerate}
  If $\Tin\Rin=\Tout\Rout$ then \eqref{min-radial-relaxed-pb} has a unique negative solution (up to additive constant). In all other cases \eqref{min-radial-relaxed-pb} has no solution.
\end{lemma}
\begin{remark}\label{rem:solution}
The solution to the Euler-Lagrange equation \eqref{EL+} is given by
\begin{equation}\label{ustar}
	u^\star(r)= \begin{cases} \Tin\Rin\log\left(\dfrac r{R_0}\right)&\text{in }(\Rin,R_0)\,,\\[1em]
		\dfrac{\Tin-\Tout}{\Rout^{-2}-\Rin^{-2}}\dfrac1r+ \dfrac{\Tout\Rout^2-\Tin\Rin^2}{\Rout^2-\Rin^2}r& \text{in }(R_0,\Rout)\,.
	\end{cases}
\end{equation}
Therefore, under the assumption \eqref{eq:hyp}, the relaxed problem admits a unique minimizer $u^\star$ explicitly given by \eqref{ustar}. 

We then set 
\begin{equation}\label{Boh}
	\mathcal{E}_0:=2\pi	\left({ \frac12}\int_{R_{\rm in}}^{R_{\rm out}} \Big((\dot u^\star)_+^2+ \Big(\frac{u^\star}{r}\Big)_+^2 \Big)r\di r+ T_{\rm in}R_{\rm in}u^\star(R_{\rm in})-
	T_{\rm out}R_{\rm out}u^\star(R_{\rm out})
	\right)\,.
\end{equation}

The corresponding relaxed configuration exhibits two qualitatively different regimes separated by the radius $R_0$. \\

Indeed, property $(iv)$ implies that
$$\dot u^\star>0\quad\text{ for every }r\in(\Rin,\Rout),$$
so the radial strain is always tensile. Physically, this reflects the fact that the applied loads stretch the sheet in the radial direction. On the other hand, by (ii) the azimuthal strain  $u^\star/r$ changes sign across the annulus. More precisely 
$$\frac{u^\star(r)}{r} <0\quad \text{ in }(\Rin,R_0),\quad\quad \frac{u^\star(r)}{r}>0\quad\text{ in }(R_0,\Rout).$$

Thus, the outer region is under azimuthal tension, while the inner region undergoes azimuthal compression. The latter is precisely the mechanism responsible for wrinkle formation in the full elastic problem. The radius $R_0$
therefore identifies the free boundary separating the wrinkled region from the purely tensile one.
\end{remark}

\subsection{Scaling of the energy and heuristics} \label{sec:scaling}
In this section we will perform several change of variables in the energy \eqref{energy1} which are convenient to our purposes. \\

\noindent\textbf{First change of variables: passage to polar coordinates.}
Let $$(u_r,u_\theta,\xi)\in W^{1,2}((\Rin,\Rout)\times[\pi,\pi);\R^2)\times W^{2,2}((\Rin,\Rout)\times[\pi,\pi))$$
be a displacement in polar coordinates, such that $(u_r,u_\theta,\xi)$ is $2\pi$-periodic in the $\theta$-variable. Then the functional 
\eqref{energy1} equals
\begin{equation}\label{energy-polar-coord}
	\begin{split}
		E_h(u_r,&u_\theta,\xi)={ \frac12}\int_{-\pi}^{\pi}\int_{R_{\rm in}}^{R_{\rm out}} \bigg[
		\Big(\partial_r u_r+\frac12(\partial_r\xi)^2\Big)^2 + \Big(\frac{u_r}{r}+\frac{\partial_\theta u_\theta}{r}+\frac{(\partial_\theta\xi)^2}{2r^2}\Big)^2\\
		&
		+ 2 \Big(\frac{\partial_\theta u_r}{2r}+\frac{\partial_r u_\theta}{2}-\frac{u_\theta}{2r}+\frac{\partial_r\xi\partial_\theta\xi}{2r}\Big)^2
		+h^2 \Big((\partial_{rr}\xi)^2+\frac{2(\partial_{r\theta}\xi)^2}{r^2}+\frac{(\partial_{\theta\theta}\xi)^2}{r^4}\Big)	\bigg]r\di r\di\theta\\
		& + T_{\rm in}R_{\rm in}\int_{-\pi}^{\pi}u_r(R_{\rm in},\theta)\di\theta- 
		T_{\rm out}R_{\rm out}\int_{-\pi}^{\pi}u_r(R_{\rm out},\theta)\di\theta\,.
	\end{split}
\end{equation}
Therefore the excess energy is
\begin{equation*}
	\begin{split}
	E_h(u_r,u_\theta,\xi)-\mathcal{E}_0&=
	{ \frac12}\int_{-\pi}^{\pi}\int_{R_{\rm in}}^{R_{\rm out}} \bigg[ 
	\dot u^\star\left(\partial_r\xi\right)^2+
	\Big(\partial_r u_r+\frac12(\partial_r\xi)^2-\dot u^\star\Big)^2\\
	& +\frac{(u^\star)_+}{r}\frac{(\partial_\theta\xi)^2}{r^2} + \Big(\frac{u_r}{r}+\frac{\partial_\theta u_\theta}{r}+\frac{(\partial_\theta\xi)^2}{2r^2}-\frac{(u^\star)_+}{r}\Big)^2\\
	&
	+ 2 \Big(\frac{\partial_\theta u_r}{2r}+\frac{\partial_r u_\theta}{2}-\frac{u_\theta}{2r}+\frac{\partial_r\xi\partial_\theta\xi}{2r}\Big)^2
	+h^2 \Big((\partial_{rr}\xi)^2+\frac{2(\partial_{r\theta}\xi)^2}{r^2}+\frac{(\partial_{\theta\theta}\xi)^2}{r^4}\Big)	\bigg]r\di r\di\theta\,.
	\end{split}
\end{equation*}
This follows by observing that:
\begin{itemize}
\item the first integrand on the right hand-side of \eqref{energy-polar-coord} can be rewritten as
\begin{equation*}
\frac12\Big(\partial_r u_r+\frac12(\partial_r\xi)^2\Big)^2=	\frac12\left[\dot u^\star\left(\partial_r\xi\right)^2+ \Big(\partial_r u_r+\frac12(\partial_r\xi)^2-\dot u^\star\Big)^2-(\dot u^\star)^2+ 2\partial_ru_r\dot u^\star\right]\,;
\end{equation*}
\item the second integrand on the right hand-side of \eqref{energy-polar-coord} can be rewritten as
\begin{equation*}
	\begin{split}
	\frac12	\Big(\frac{u_r}{r}+\frac{\partial_\theta u_\theta}{r}+\frac{(\partial_\theta\xi)^2}{2r^2}\Big)^2&=
\frac12\bigg[	\frac{(u^\star)_+}{r}\frac{(\partial_\theta\xi)^2}{r^2} + \Big(\frac{u_r}{r}+\frac{\partial_\theta u_\theta}{r}+\frac{(\partial_\theta\xi)^2}{2r^2}-\frac{(u^\star)_+}{r}\Big)^2\\
	& - \frac{(u^\star)^2_+}{r^2}+2 \frac{u_r}{r} \frac{(u^\star)_+}{r}+2\frac{\partial_\theta u_\theta}{r}\frac{(u^\star)_+}{r}\bigg]\,;
	\end{split}
\end{equation*}
\item for $\theta\in [-\pi,\pi)$ the integrand of the boundary term on the right hand-side of \eqref{energy-polar-coord} can be rewritten as
\begin{equation*}
	\begin{split}
		\Tin\Rin u_r(\Rin,\theta)- 	&	\Tout\Rout u_r(\Rout,\theta)=
		\dot u^\star(\Rin) \Rin u_r(\Rin,\theta)- 			\dot u^\star(\Rout) \Rout u_r(\Rout,\theta)\\
			&=  \int_{\Rin}^{\Rout}\partial_r \Big(u_rr\dot u^\star \Big)\di r= 
			 \int_{\Rin}^{\Rout} u_r\partial_r (r\dot u^\star )\di r +  \int_{\Rin}^{\Rout} \partial_ru_rr\dot u^\star \di r\\
			 & = \int_{\Rin}^{\Rout}\frac{u_r}{r}\frac{(u^\star)_+}{r}r\di r
			 +  \int_{\Rin}^{\Rout} \partial_ru_rr\dot u^\star \di r\,,
	\end{split}
\end{equation*}
where we used \eqref{EL+};
\item analogously for $\theta\in [-\pi,\pi)$ the integrand of the boundary term on the right hand-side of \eqref{Boh} satisfies
\begin{equation*}
	\Tin\Rin u^\star(\Rin)- 	\Tout\Rout u^\star(\Rout)= 
	\int_{\Rin}^{\Rout} \frac{(u^\star)_+^2}{r^2}r\di r+ \int_{\Rin}^{\Rout} (\dot u^\star)^2\di r\,:
\end{equation*}
\item since $u^\star$ is independent of $\theta$ it holds
\begin{equation*}
\int_0^{2\pi} \int_{\Rin}^{\Rout} \frac{\partial_\theta u_\theta}{r}\frac{(u^\star)_+}{r}r\di r\di \theta=0\,.
\end{equation*}
\end{itemize}

\noindent\textbf{Second change of variables: rescaling the wrinkle length-scale.}
{From~\cite{BeKo14+,BeKo16Lame} it is known that the excess energy $E_h-\mathcal E_0$
	 scales linearly in $h$. This naturally suggests working with
  the rescaled  excess energy $\frac{E_h-\mathcal E_0}{h}$. For $\Rin \le r\le R_0$ one expects that the sheet develops  out-of-plane wrinkles oscillations in the  $\theta$-direction.
 The linear-in-$h$ energy scaling implies that the bending term satisfies
 $h^2|\nabla^2 \xi|^2 \sim h$; in particular, its dominant component $\xi_{,\theta\theta}$ should be of order $h^{-1/2}$. Consequently, 
 the characteristic wavelength of the wrinkles in the bulk is expected to be of order $h^{1/2}$. Not surprisingly, this coincides with the scale employed in the upper-bound constructions of~\cite{BeKo14+,BeKo16Lame}.\\

  To analyse the limiting structure of the wrinkles as $h \to 0$, we therefore rescale the angular variable $\theta$ by the factor $L := h^{-1/2}$. Equivalently, this choice ensures that the characteristic wrinkle wavelength becomes of order one in the new angular variable. More precisely, after performing the change of variables
}
 
\begin{equation*}
	\hat u_r(r,\theta):=u_r(r,L^{-1}\theta), \quad \hat u_\theta(r,\theta):=L u_\theta(r,L^{-1}\theta),\quad \hat \xi(r,\theta):=L\xi(r,L^{-1}\theta),
\end{equation*}
The energy $E_h-\mathcal E_0$ becomes  the functional   $$\mathcal E_L\colon \mathcal{A}_L^{\rm in}\times	\mathcal{A}_L^{\rm out}\to[0,+\infty]\,,$$ 
where the function spaces describing admissible deformations have the form 
\begin{equation*}\label{def:in-plane}
	\mathcal{A}_L^{\rm in}:=\Big\{(u_r,u_\theta)\in W^{1,2}_{ \textrm{loc}}((\Rin,\Rout)\times\R;\R^2)\colon (u_r(r,\cdot),u_\theta(r,\cdot)) \text{ is $2\pi L$-periodic $\forall r\in(\Rin,\Rout)$}\Big\},
\end{equation*}
\begin{equation*}\label{def:out-of-plane}
	\mathcal{A}_L^{\rm out}:=\Big\{\xi\in W^{2,2}_{ \textrm{loc}}((\Rin,\Rout)\times\R)\colon \xi(r,\cdot) \text{ is $2\pi L$-periodic $\forall r\in(\Rin,\Rout)$}\Big\}\,,
\end{equation*} and
\begin{equation*}\label{eqn:RL}
	\begin{split}
	\mathcal E_L(u_r,u_\theta,\xi):= \pi\strokedint_{-\pi L}^{\pi L} \int_{\Rin}^{\Rout}
	\bigg[ &
\dot u^\star\frac{\left(\partial_r\xi\right)^2}{L^2}+
	\Big(\partial_r u_r+\frac12\frac{\left(\partial_r\xi\right)^2}{L^2}-\dot u^\star\Big)^2\\
	& +\frac{(u^\star)_+}{r}\frac{(\partial_\theta\xi)^2}{r^2} + \Big(\frac{u_r}{r}+\frac{\partial_\theta u_\theta}{r}+\frac{(\partial_\theta\xi)^2}{2r^2}-\frac{(u^\star)_+}{r}\Big)^2\\
	&
	+ 2 \Big(L\frac{\partial_\theta u_r}{2r}+\frac{\partial_r u_\theta}{2L}-\frac{u_\theta}{2rL}+\frac{\partial_r\xi\partial_\theta\xi}{2rL}\Big)^2\\
&	+\frac{1}{L^4} \Big(\frac{(\partial_{rr}\xi)^2}{L^2}+\frac{2(\partial_{r\theta}\xi)^2}{r^2}+L^2\frac{(\partial_{\theta\theta}\xi)^2}{r^4}\Big)	\bigg]r\di r\di\theta\,.
	\end{split}
\end{equation*}
Furthermore, $\frac{E_h - \mathcal{E}_0}{h}$ turns into 
$\mathcal{F}_L\colon	\mathcal{A}_L^{\rm in}\times	\mathcal{A}_L^{\rm out}\to\R$ defined as
\begin{equation}\label{def:F_L}\begin{split}
	\mathcal{F}_L(u_r,u_\theta,\xi)&:=L^2\mathcal E_L(u_r,u_\theta,\xi)\\
&=	 \pi\strokedint_{-\pi L}^{\pi L} \int_{\Rin}^{\Rout}
	\bigg[ 
	L^2
	\Big(\partial_r u_r+\frac12\frac{\left(\partial_r\xi\right)^2}{L^2}-\dot u^\star\Big)^2+L^2\frac{(u^\star)_+}{r}\frac{(\partial_\theta\xi)^2}{r^2} \\
	& + L^2\Big(\frac{u_r}{r}+\frac{\partial_\theta u_\theta}{r}+\frac{(\partial_\theta\xi)^2}{2r^2}-\frac{(u^\star)_+}{r}\Big)^2\\
	&
	+ 2 \Big(L^2\frac{\partial_\theta u_r}{2r}+\frac{\partial_r u_\theta}{2}-\frac{u_\theta}{2r}+\frac{\partial_r\xi\partial_\theta\xi}{2r}\Big)^2\\
	&	
	+ 	\Big(\dot u^\star{\left(\partial_r\xi\right)^2}
	+\frac{(\partial_{\theta\theta}\xi)^2}{r^4}\Big)\\
	&+\frac{1}{L^2} \Big(\frac{(\partial_{rr}\xi)^2}{L^2}+\frac{2(\partial_{r\theta}\xi)^2}{r^2}\Big)	
	\bigg]r
	\di r\di\theta
	\,.	\end{split}
\end{equation}
\noindent\textbf{Heuristics of the asymptotics.}
Before turning to the rigorous analysis, we briefly discuss the heuristic structure of the functional $\mathcal F_L$ and the expected limiting regime. Recall that the limit $h\to0$ corresponds to $L\to\infty$. Since most terms in the energy are quadratic and the relevant configurations oscillate on intervals whose length diverges with $L$, it is natural to analyse the problem in Fourier variables.

Assuming that the limit of $\mathcal F_L$ exists, and in particular that minimizing sequences remain uniformly bounded as $L\to\infty$, the first four terms in the energy must vanish asymptotically. The first term can be made small by choosing
\[
\partial_r u_r \sim \dot u^\star + o(L^{-1})
\]
and keeping $\partial_r\xi$ sufficiently small. The smallness of the second term reflects the fact that $(u^\star)_+=0$ in $(\Rin,R_0)$, while the out-of-plane displacement $\xi$, and therefore $\partial_\theta\xi$, is expected to vanish in $(R_0,\Rout)$, corresponding to the unwrinkled region.

Integrating the third term with respect to $\theta$ and using the periodicity of $(u_r,u_\theta)$ yields the constraint
\begin{equation}\label{constraint}
	\fint_{-\pi L}^{\pi L}
	(\partial_\theta\xi)^2
	\di\theta
	=
	-2ru^\star(r)
	+
	o(L^{-1})
	\qquad
	\text{in }(\Rin,R_0),
\end{equation}
which represents the excess arclength that must be accommodated through wrinkling. The fourth term can also be made negligible, whereas the last term is expected to vanish as $L\to\infty$.

By contrast, the remaining contribution --- corresponding to the competition between radial stretching and azimuthal bending generated by the wrinkles --- is of order one and therefore survives in the limit. This suggests that the limiting functional should arise as the $\Gamma$-limit of
\[
\pi
\strokedint_{-\pi L}^{\pi L}
\int_{\Rin}^{R_0}
\left(
\dot u^\star(\partial_r\xi)^2
+
\frac{(\partial_{\theta\theta}\xi)^2}{r^4}
\right)
r\di r\di\theta
\]
subject to the constraint \eqref{constraint}.

To obtain a formulation stable under the limit $L\to\infty$, it is convenient to work with the squares of the Fourier coefficients and with suitably defined measures as the primary objects of study. This choice is natural since the arclength constraint involves quadratic quantities, and therefore becomes linear when expressed in terms of the squared Fourier coefficients.
 In what follows, we denote by $k\in\R$ the Fourier variable associated with the $\theta$-direction; the same notation will also be used for the second variable in the measure-theoretic formulation.

Finally, let us remark that one could also include in the energy a term accounting for the presence of a substrate. Typically, such a contribution takes the form
\begin{equation*}\label{substrate}
	K_{\rm sub}\int_\Omega \xi^2\dx
	=
	K_{\rm sub}
	\int_{-\pi}^{\pi}
	\int_{\Rin}^{\Rout}
	\xi^2\,r\di r\di\theta
	=
	\frac1{L^2}
	\fint_{-\pi L}^{\pi L}
	\int_{\Rin}^{\Rout}
	\hat\xi^2\,r\di r\di\theta,
\end{equation*}
where $K_{\rm sub}$ is a nondimensional parameter measuring the relative stiffness of the substrate with respect to the film. After multiplication by the prefactor $L^2$, this contribution becomes of order one and therefore persists in the limit.
Since this term behaves trivially in the $\Gamma$-convergence analysis, we omit it from the present discussion. 
\section{Passage to measures $\Gamma$-convergence}\label{sec:main-result}  In this section we collect the main result of the paper Theorem \ref{theo:main_theo}. In order to do that, we need to pass to a measure-theoretic framework.\\
To each $\xi\in \mathcal{A}_L^{\rm out} $ we associate a measure $\mu^L\in \mathcal{M}_b^+((\Rin,\Rout)\times\R)$ defined through the Fourier coefficients of $\partial_\theta\xi$. 
 Let   $a_k(r)$ be the $k$-th Fourier coefficient in the $\theta$-variable for $k\in\frac{\mathbb Z}{L}$, that is
 \begin{equation*}\label{def:coef_a}
 	a_k(r):=
 	\begin{cases}\displaystyle
 		\sqrt 2	\strokedint_{-\pi L}^{\pi L} \xi(r,\theta) \sin(k\theta) \di\theta&k\in \dfrac{\mathbb{Z}}{L}, k>0 \,,\\[1em]
 		\displaystyle
 		\sqrt 2 \strokedint_{-\pi L}^{\pi L} \xi(r,\theta) \cos(k\theta) \di\theta&k\in \dfrac{\mathbb{Z}}{L}, k< 0\,,\\[1em]
 		\displaystyle
 		\strokedint_{-\pi L}^{\pi L}\xi(r,\theta)\di\theta&k=0\,.
 	\end{cases}
 \end{equation*}
Let also 
\begin{equation}\label{coef-squared}
 {a}(r,k):=k^2a_k^2(r)\quad\text{ for } (r,k)\in (\Rin,\Rout)\times\frac{\mathbb Z}{L}\,.
\end{equation}
 Then we have
\begin{equation*}\label{eq:fourier_repr}
	\begin{aligned}
		\xi(r,\theta)&=a_0(r)+ \sum_ {k\in \frac{\mathbb{Z}}{L}, k>0 } a_k(r)\sqrt2\sin(k\theta) + \sum_ {k\in \frac{ \mathbb{Z}}{L}, k < 0 } a_k(r)\sqrt2\cos(k\theta) \\
		&=a_0(r)+ \sum_ {k\in \frac{ \mathbb{Z}}{L}, k>0 } {\rm sign}(a_k(r))\frac{\sqrt{a(r,k)}}{k} \sqrt{2}\sin(k\theta) + \sum_ {k\in \frac{ \mathbb{Z}}{L}, k < 0 } {\rm sign}(a_k(r))\frac{\sqrt{a(r,k)}}{-k} \sqrt{2}\cos(k\theta)
		\,.
	\end{aligned}
\end{equation*}  
\begin{remark}[Plancherel's Identity]
From Plancherel equality we have
\begin{equation}\label{eq:plancherel-u}
	\fint_{-\pi L}^{\pi L}\xi^2\di\theta= a^2_0(r)+\sum_ {k\in \frac{\mathbb{Z}}{L},k\ne0}a^2_k(r)=a^2_0(r)+
	\sum_ {k\in \frac{ \mathbb{Z}}{L},k\ne0}\frac{{a}(r,k)}{k^2}\,.
\end{equation}
The same holds for partial derivatives of $\xi$, that is
\begin{equation}\label{plancherel-derivatives}
	\begin{split}
		\fint_{-\pi L}^{\pi L}(D^\alpha \xi)^2\di\theta&=
		(D^\alpha a_0(r))^2+
		\sum_ {k\in \frac{ \mathbb{Z}}{L},k\ne0}\Big(\frac{\di^{\alpha_1}}{\di r^{\alpha_1}} a_k(r)k^{\alpha_2}\Big)^2 \\
		&	= (D^\alpha a_0(r))^2+
		\sum_ {k\in \frac{ \mathbb{Z}}{L},k\ne0}\Big(\frac{\partial^{\alpha_1}}{\partial r^{\alpha_1}} \big(\sqrt{a(r,k)}\big)k^{\alpha_2-1}\Big)^2 
		\,, 
	\end{split}
\end{equation}
with $\alpha=(\alpha_1,\alpha_2)$ multi-index with $|\alpha|\le2$.  In case $\xi$ has higher regularity, i.e., $\xi\in W^{k,2}((\Rin,\Rout)\times\R)$  with $k>2$, then the same applies for the higher derivatives, i.e., for $|\alpha|\le k$.
For later convenience we also note that 
\begin{equation}\label{derivatives-a}
	\begin{split}
	{\partial_r} \big(\sqrt{a(r,k)}\big)= \frac{\partial_ra(r,k)}{2\sqrt{a(r,k)}}\,, \quad{\partial_{rr}} \big(\sqrt{a(r,k)}\big)=\frac{\partial_{rr}a(r,k)}{2\sqrt{a(r,k)}}- \frac{(\partial_ra(r,k))^2}{4\sqrt{a^3(r,k)}}\,.
	\end{split}
\end{equation}
\end{remark}

\begin{definition}[Measures $\mu^L$ and $\mu^L_{,r}$]\label{def:muL}
	Let $\xi\in \mathcal{A}_L^{\rm out}$.  We denote by $\mu^L(\xi)\in\mathcal{M}_b^+((\Rin,\Rout)\times\R)$ the measure given by 
	\begin{equation*}
		\mu^L(\xi):= \sum_{
			k\in \frac{\mathbb{Z}}{L}} a(r,k)  \L^1\res{(\Rin,\Rout)}\times \delta_k\,,
	\end{equation*}
with $a(r,k)$ defined as in \eqref{coef-squared}
Moreover we denote by $\mu^L_{,r}(\xi)$ the distributional $r$-derivative of $\mu^L(\xi)$.
	\end{definition}
\begin{remark}\label{rem:int-byparts}

\begin{enumerate}[label=$(\roman*)$]
		\item\label{rem-i} The distributional $r$-derivative of a measure $\mu\in\mathcal{M}_b^+(I\times\R)$ (with $I\subset\R$ interval) is defined as follows:
		for all $\varphi\in C^\infty_c(I\times\R)$ we have
		\begin{equation*}
			\langle\mu_{,r},\varphi\rangle:=-\int_{I\times\R}\varphi_{,r}\di\mu\,.
		\end{equation*}
		Moreover by a density argument $\mu_{,r}$ can be extended to functions $\varphi(r,k)=\phi(r) \chi_A(k)$ with $\phi\in C^\infty_c(I)$ and $A\subset\R$ bounded and measurable as 
		\begin{equation*}
			\langle\mu_{,r},\phi(r) \chi_A(k)\rangle:=-\int_{I\times A}\dot\phi(r)\di\mu\,;
		\end{equation*}
	If   $\mu_{,r}\ll\mu$ we denote by $\frac{\di \mu_{,r}}{\di\mu}$ the Radon-Nikodym derivative.
	\item\label{rem-ii} Let $\mu\in\mathcal{M}_b^+(I\times\R)$ be of the form
	\begin{equation*}
		\mu=\sum_{k\in K}a(r,k)\mathcal{L}^1\res I\times\delta_k\,,
	\end{equation*}
	with $K\subset\R$ countable and $a(\cdot,k)\in W^{1,1}(I)$ for all $k\in K$. Then 
	\begin{equation*}
	\mu_{,r}= \sum_{k\in K}\partial_ra(r,k)\mathcal{L}^1\res I\times\delta_k\,.
	\end{equation*}
	Moreover as
	$\partial_r a(\cdot,k)=0$ a.e. in $\{r\in I\colon {a}(r,k)=0\}$ it follows $\mu_{,r}\in \mathcal{M}(I\times\R)$, $\mu_{,r}\ll\mu$ and 
	\begin{equation*}
\frac{\di \mu_{,r}}{\di \mu}= \frac{\partial_ra(r,k)}{a(r,k)}\,.
	\end{equation*}
	\end{enumerate}
\end{remark}
\begin{definition}[Convergence] \label{def:convergence}
	For $L>0$ let $(u_r^L,u_\theta^L,\xi^L)\in\mathcal{A}_L^{\rm in}\times \mathcal{A}_L^{\rm out}$.
	We say a sequence $(u_r^L,u_\theta^L,\xi^L)$ converges, as $L\to\infty$, to $\mu\in\M_{b}^+((\Rin,\Rout)\times\R)$, if 
	$$(\mu^L(\xi^L),\mu^L_{,r}(\xi^L)) \text{ weakly-* converge  to } (\mu,\mu_{,r})\,.$$
\end{definition}

\noindent 
	We introduce the class of measures
	\begin{equation}\label{def:limit_measures}
		\begin{split}
			\mathcal{M}_\infty:=\bigg\{\mu\in &\M_b^+((\Rin,\Rout)\times\R) \colon\mu( [R_0,\Rout)\times\R)=0,\ \mu_{,r}\in\M_b((\Rin,\Rout)\times\R)\,, \\& \mu_{,r}\ll \mu\,,
			\int_{(\Rin,R_0)\times\R}{\phi(r)} \di\mu(r,k)=-	\int_{\Rin}^{R_0}	2r{u^\star(r)}\phi(r)\di r\quad \forall\phi\in C_c^{\infty}(\Rin,R_0)
			\bigg\}\,,
		\end{split}
	\end{equation}
and the functional $\mathcal F_\infty\colon \M_\infty \to[0,+\infty]$
	\begin{equation}\label{def:F_infty}
	\mathcal F_\infty(\mu)=\pi \int_{(\Rin,R_0)\times\R} \biggl[
	\frac{\dot u^\star(r)}{4k^2}\Bigl(\frac{\di \mu_{,r}}{\di\mu}(r,k)\Bigr)^2+	\frac{k^2 }{r^4} \biggr]r\di\mu(r,k)\,.
\end{equation}
We are now ready to state our main result.
\begin{theorem}
	\label{theo:main_theo}
	\black
	Let $\mathcal F_L$ and $\mathcal{F}_\infty$ be as in \eqref{def:F_L} and \eqref{def:F_infty} respectively.	Then the following holds:
	\begin{itemize}
		\item[$a)$] $($Compactness$)$. For $L>0$ let $(u_r^L,u_\theta^L,\xi^L)\in \mathcal{A}_L^{\rm in}\times \mathcal{A}_L^{\rm out}$ be such that $$\sup_L\mathcal F_L(u_r^L,u_\theta^L,\xi^L)<+\infty\,.$$ Then there exists a subsequence (not relabeled) and $\mu\in\mathcal M_\infty$ such that $(u_r^L,u_\theta^L,\xi^L)$ converges as $L\to +\infty$ in the sense of Definition~\ref{def:convergence} to $\mu$.
		\item[$b)$] $($$\Gamma$-convergence$)$. As $L\to+\infty$ the functionals $\mathcal F_L$ $\Gamma$-converge, with respect to the convergence in Definition~\ref{def:convergence}, to the functional $\mathcal F_\infty$.
	\end{itemize}
\end{theorem}

\begin{remark}\label{rem:limit}
	\begin{enumerate}[label=$(\roman*)$]
\item Before giving the rigorous proof of Theorem \ref{theo:main_theo}, we outline the underlying the heuristics.
For $\xi\in\mathcal{A}_L^{\rm out}$ let $\mu^L=\mu^L(\xi)$ be defined as in definition~\ref{def:muL}. Recall that $$\sqrt{a(r,k)}=ka_k(r)\quad k\in \frac{\mathbb Z}L$$  corresponds to the $k$-th Fourier coefficient of $\partial_\theta\xi(r,\cdot)$.  From the heuristic discussion above, we expect $\xi$ to vanish on the interval $(R_0,\Rout)$ for  $L$ large, which in turn implies $\sqrt{a(r,k)}=0$ for each $k$ in this region,
and therefore  $\mu^L=0$		on $(R_0,\Rout)$.

Combining this observation with  Plancherel's identity \eqref{plancherel-derivatives} yields 
\begin{equation*}
	\sum_{
		k\in \frac{\mathbb{Z}}{L}} a(r,k) =\fint_{-\pi L}^{\pi L}(\partial_\theta\xi(r,\theta))^2\di\theta=-2ru^\star(r)+o(L^{-1})\quad\text{ for }r\in(\Rin,R_0)
	\,,
\end{equation*}
which, in the limit $L\to\infty$, leads precisely to the integral constraint in \eqref{def:limit_measures} satisfied by the limit measure.
Similarly using again \eqref{plancherel-derivatives}, \eqref{derivatives-a}, we compute
\begin{equation*}\begin{split}
	&	\pi\strokedint_{-\pi L}^{\pi L} \int_{\Rin}^{R_0}\Big(\dot u^\star{\left(\partial_r\xi\right)^2}
		+\frac{(\partial_{\theta\theta}\xi)^2}{r^4}\Big)r\di r\di\theta\\
&	= \pi\sum_{k \in \frac{ \mathbb Z}{L}} \int_{\Rin}^{R_0}\Big( \dot u^\star(r)\dot a_k^2(r) + \frac{a_k^2(r) k^4 }{r^4}\Big)r\di r\\
&= \pi\sum_{k \in \frac{ \mathbb Z}{L}} \int_{\Rin}^{R_0}\Big( \frac{\dot u^\star(r)}{4k^2}\frac{(\partial_ra(r,k))^2}{a(r,k)}
 + \frac{a(r,k) k^2 }{r^4}\Big)r\di r
= 	\mathcal F_\infty(\mu^L)
		\,,
	\end{split}
\end{equation*}
and this expression converges, as $L\to\infty$,  to the claimed $\Gamma$-limit.
	\item The constraint 
\begin{equation}\label{constr2}
	\int_{(\Rin,R_0)\times\R}\phi(r)\di\mu(r,k)=-	\int_{\Rin}^{R_0}	2{u^\star}(r){r}\phi(r)\di r\quad \forall\phi\in C_c^{\infty}(\Rin,R_0)
\end{equation}
may be interpreted as a condition on the marginal of the measure $\mu$ with respect to the radial variable. Moreover it
is equivalent to 
\begin{equation}\label{constr1}
	\int_{(\Rin,R_0)\times\R}\frac{\phi(r)}{2r^2} \di\mu(r,k)=-	\int_{\Rin}^{R_0}	\frac{u^\star(r)}{r}\phi(r)\di r\quad \forall\phi\in C_c^{\infty}(\Rin,R_0)\,,
\end{equation}
Indeed, assume \eqref{constr1} holds true and let $\phi$ be a test function. Then using that $\hat \phi=r^2\phi$ is still a test function we can prove \eqref{constr2}. The vice-versa can be proven analogously using $\tilde\phi=\frac\phi{r^2}$.
		\item\label{rem:limit(i)} When convenient we will identify the class $\mathcal M_\infty$ with the class of measures
		\begin{equation} 
			\begin{split}
				\bigg\{\mu\in \M_b^+&((\Rin,R_0)\times\R) \colon \mu_{,r}\in\M_b((\Rin,R_0)\times\R)\,, \ \mu_{,r}\ll \mu\,,
				\\& 	\int_{(\Rin,R_0)\times\R}\frac{\phi(r)}{2r^2}\di\mu(r,k)=-	\int_{\Rin}^{R_0}	\frac{u^\star(r)}{r}\phi(x)\di r\quad \forall\phi\in C_c^{\infty}(\Rin,R_0)
				\bigg\}\,;
			\end{split}
		\end{equation}
\item\label{rem:limit(ii)} The functional $\mathcal{F}_\infty$ admits the alternative representation
\begin{equation}\label{def:F_infty2}
	\mathcal F_\infty(\mu)=
	\pi
	\int_{(\Rin,R_0)\times\R}
	\frac{\dot u^\star(r)r}{4k^2}
	\Bigl(\frac{\di \mu}{\di|\tilde\mu|}\Bigr)^{-1}
	\Bigl(\frac{\di \mu_{,r}}{\di|\tilde\mu|}\Bigr)^2
	\di|\tilde \mu|
	+
	\pi
	\int_{(\Rin,R_0)\times\R}
	\frac{k^2}{r^3}\di\mu,
\end{equation}
where $\tilde \mu:=(\mu,\mu_{,r})$ and $|\tilde\mu|$ denotes its total variation. Indeed, since $\mu_{,r}\ll\mu\ll|\tilde\mu|$, we have
\[
\frac{\di\mu_{,r}}{\di|\tilde\mu|}
=
\frac{\di\mu_{,r}}{\di\mu}
\frac{\di\mu}{\di|\tilde\mu|},
\]
and therefore
\begin{equation*}
	\begin{split}
		\int_{(\Rin,R_0)\times\R}
		\frac{\dot u^\star(r)r}{4k^2}
		\Bigl(\frac{\di \mu_{,r}}{\di\mu}\Bigr)^2
		\di\mu
		&=
		\int_{(\Rin,R_0)\times\R}
		\frac{\dot u^\star(r)r}{4k^2}
		\Bigl(\frac{\di \mu}{\di|\tilde\mu|}\Bigr)^{-1}
		\Bigl(\frac{\di \mu_{,r}}{\di|\tilde\mu|}\Bigr)^2
		\di|\tilde\mu|.
	\end{split}
\end{equation*}
This representation highlights that the stretching contribution to the energy is a particular instance of the Benamou--Brenier functional arising in optimal transport theory; see \cite[Proposition~5.18]{Sa15}. In particular, it follows that the functional $\mathcal F_\infty$ is 1-homogeneous, convex and lower semicontinuous.
	\end{enumerate}
\end{remark}
The next two sections are devoted to the proof of Theorem~\ref{theo:main_theo}. For the reader's convenience, we divide the argument into three parts: the compactness result, the $\Gamma$-liminf inequality, both contained in Section~\ref{sec:compactness_lwb}, and the $\Gamma$-limsup inequality, proved in Section~\ref{sec:upb}.
\section{Compactness and lower bound}\label{sec:compactness_lwb}
We begin by establishing the compactness of sequences with uniformly bounded energy.
\begin{prop}[Compactness]\label{prop:compactness}
Let $L>0$ and let $(u_r^L,u_\theta^L,\xi^L)\in \mathcal{A}_L^{\rm in}\times \mathcal{A}_L^{\rm out}$ satisfy $$\sup_L\mathcal F_L( u_r^L,u_\theta^L,\xi^L)<+\infty\,.$$ Then, up to a  (not relabeled) subsequence, there exists  $\mu\in\mathcal M_\infty$ such that $(u_r^L,u_\theta^L,\xi^L)$ converges to $\mu$, as $L\to+\infty$, in the sense of Definition~\ref{def:convergence}.
\end{prop}
\begin{proof}
 Let 
$\mu^L:=\mu^L(\xi^L)$
and $\mu^L_{,r}:=\mu^L_{,r}(\xi^L)$ be defined  as in Definition \ref{def:muL}. Then there exist functions ${a}^L(r,k)$ such that $a(\cdot,k)\in W^{1,1}(\Rin,\Rout)$ and
\begin{equation*}
	\mu^L= \sum_{
		k\in \frac{ \mathbb{Z}}{L}}  a^L(r,k) \L^1\res{(\Rin,\Rout)}\times \delta_k\,,
	\quad
	\mu^L_{,r}= \sum_{
		k\in \frac{\mathbb{Z}}{L}}  a_{,r}^L(r,k)  \L^1\res{(\Rin,\Rout)}\times\delta_k\,.
\end{equation*} 
\textit{Step 1:}  
we show that there exists $\mu\in \mathcal{M}_b^+((\Rin,\Rout)\times\R)$ with $\mu_{,r}\in \mathcal{M}_b((\Rin,\Rout)\times\R)$ and such that $(\mu^L,\mu^L_{,r})\stackrel{*}{\rightharpoonup}(\mu,\mu_{,r})$.

Since $$0<C_0:=\sup_L\mathcal{F}_L(u_r^L,u_\theta^L,\xi^L)<+\infty\,,$$ we obtain the energy bound
\begin{equation}\label{energy-bound}
	\begin{split}
	C_0\ge
		\pi\strokedint_{-\pi L}^{\pi L} \int_{\Rin}^{\Rout}
		\bigg[ &
		L^2
		\Big(\partial_r u^L_r+\frac12\frac{\left(\partial_r\xi^L\right)^2}{L^2}-\dot u^\star\Big)^2+L^2\frac{(u^\star)_+}{r}\frac{(\partial_\theta\xi^L)^2}{r^2} \\
		& + L^2\Big(\frac{u_r^L}{r}+\frac{\partial_\theta u^L_\theta}{r}+\frac{(\partial_\theta\xi^L)^2}{2r^2}-\frac{(u^\star)_+}{r}\Big)^2\\
		&
		+ 2 \Big(L^2\frac{\partial_\theta u^L_r}{2r}+\frac{\partial_r u^L_\theta}{2}-\frac{u^L_\theta}{2r}+\frac{\partial_r\xi^L\partial_\theta\xi^L}{2r}\Big)^2\\
		&	
		+ 	\Big(\dot u^\star{\left(\partial_r\xi^L\right)^2}
		+\frac{(\partial_{\theta\theta}\xi^L)^2}{r^4}\Big)\\
		&+\frac{1}{L^2} \Big(\frac{(\partial_{rr}\xi^L)^2}{L^2}+\frac{2(\partial_{r\theta}\xi^L)^2}{r^2}\Big)	
		\bigg]r
		\di r\di\theta
		\,.
	\end{split}
\end{equation}
From Lemma \ref{min-radial-relaxed-pb} we have that $$u^\star<0 \text{ in }(\Rin,R_0),\quad u^\star>0 \text{ in }(R_0,\Rout), \quad \dot u^\star\ge c>0\,.$$
This combined with the bound of the second term on the right-hand-side of
\eqref{energy-bound} yields
\begin{equation}\label{xi-theta-outer-part}
\strokedint_{-\pi L}^{\pi L} \int_{R_0+\beta}^{\Rout}  (\partial_\theta\xi^L)^2 \di r\di\theta\lesssim\frac 1{L^2}\,,
\end{equation}
for $0<\beta<(\Rout-R_0)/2$ fixed. 
Next we show that 
\begin{equation}\label{closeness-ur-ustar}
\left|
\strokedint_{-\pi L}^{\pi L} (u^\star(r)-u_r^L(r,\theta))\di\theta
\right|\le \frac CL\quad\text{ for all } r\in (\Rin,\Rout)\,.
\end{equation}
To this purpose set 
\begin{equation*}
	v^L(r,\theta):=u^\star(r)-u_r^L(r,\theta)\quad \text{ and }\quad
V^L(r):=\strokedint_{-\pi L}^{\pi L}v^L(r,\theta)\di\theta\,.
\end{equation*}
From \eqref{energy-bound} we deduce that 
\begin{equation*}\label{energy-bound-2}
	\begin{split}
 \strokedint_{-\pi L}^{\pi L}\int_{R_0+\beta}^{\Rout} \Big(\frac{v^L}{r}+\frac{\partial_\theta u^L_\theta}{r}+\frac{(\partial_\theta\xi^L)^2}{2r^2}\Big)^2r\di r\di \theta&\lesssim \frac1{L^2} \,, \\[1em]
 \strokedint_{-\pi L}^{\pi L}\int_{\Rin}^{\Rout}\Big(\partial_rv^L+\frac12\frac{\left(\partial_r\xi\right)^2}{L^2}\Big)^2r\di r\di\theta& \lesssim\frac1{L^2}\,.
	\end{split}
\end{equation*}
This together with \eqref{xi-theta-outer-part} and Hölder's inequality imply 
\begin{equation*}
\begin{split}
\left|\int_{R_0+\beta}^{\Rout}V^L(r)\di r\right| 
&=  \left|\strokedint_{-\pi L}^{\pi L} \int_{R_0+\beta}^{\Rout}\frac{v^L}{r}r\di r\di\theta
\right|\\
& \le \left| \strokedint_{-\pi L}^{\pi L}\int_{R_0+\beta}^{\Rout} \Big(\frac{v^L}{r}-\frac{\partial_\theta u^L_\theta}{r}-\frac{(\partial_\theta\xi^L)^2}{2r^2}\Big)r\di r\di \theta\right|\\&+ 
\strokedint_{-\pi L}^{\pi L}\int_{R_0+\beta}^{\Rout} \frac{(\partial_\theta\xi^L)^2}{2r^2}r\di r\di \theta\lesssim \frac 1L\,.
\end{split}
\end{equation*}
Then by the mean value theorem there exists $\tilde r=\tilde r(L)\in (R_0+\beta,\Rout)$ such that
\begin{equation}\label{mean-value-thm}
	|V^L(\tilde r)| \lesssim \frac 1L\,.
\end{equation}
In a similar manner we get for any $\Rin<r_0<r_1<\Rout$
\begin{equation}
	\begin{split}\label{estimate-F-L}
|V^L(r_1)-V^L(r_0)|&= \left|\strokedint_{-\pi L}^{\pi L}\int_{r_0}^{r_1} \partial_rv^L\di r\di \theta
\right|\\
&\lesssim \left|
 \strokedint_{-\pi L}^{\pi L}\int_{\Rin}^{\Rout}\Big(\partial_rv^L+\frac12\frac{\left(\partial_r\xi^L\right)^2}{L^2}\Big)r\di r\di\theta 
\right|\\&+
\left|
 \strokedint_{-\pi L}^{\pi L}\int_{\Rin}^{\Rout}\frac12\frac{\left(\partial_r\xi^L\right)^2}{L^2}r\di r\di\theta 
\right|\lesssim \frac 1L\,.
	\end{split}
\end{equation}
Therefore combining \eqref{mean-value-thm} and \eqref{estimate-F-L} we find for all $r\in (\Rin,\Rout)$
\begin{equation*}
\begin{split}
\left|
\strokedint_{-\pi L}^{\pi L} (u^\star(r)-u_r^L(r,\theta))\di\theta
\right|=
|V^L(r)|\le |V^L(r)-V^L(\tilde r)|+ |V^L(\tilde r)|\lesssim \frac 1L\,,
\end{split}
\end{equation*}
and \eqref{closeness-ur-ustar} is proven.
Then by Hölder inequality, \eqref{energy-bound} and \eqref{closeness-ur-ustar} we find
\begin{equation}\label{xi-theta-inner}
	\begin{split}
		\strokedint_{-\pi L}^{\pi L}\int_{\Rin}^{R_0} \frac{(\partial_\theta\xi^L)^2}{2r^2}r\di r\di\theta
		&\le
		\strokedint_{-\pi L}^{\pi L}\int_{\Rin}^{R_0}
		\Big|\frac{u_r^L}{r}+\frac{\partial_\theta u_\theta^L}{r}+\frac{(\partial_\theta\xi^L)^2}{2r^2}\Big|r\di r\di\theta\\
		& + \int_{\Rin}^{R_0} \left|
		\strokedint_{-\pi L}^{\pi L} (u^\star(r)-u_r^L(r,\theta))\di\theta
		\right|	\di r\\ &
		- \int_{\Rin}^{R_0}
		{u^\star}\di r\\&\lesssim \frac{1}{L}+1\lesssim 1\,,
	\end{split}
\end{equation}
and
\begin{equation}\label{xi-theta-outer}
\begin{split}
	\strokedint_{-\pi L}^{\pi L}\int_{R_0}^{\Rout} \frac{(\partial_\theta\xi^L)^2}{2r^2}r\di r\di\theta& \le 
		\strokedint_{-\pi L}^{\pi L}\int_{\Rin}^{R_0}
	\Big|\frac{u_r^L}{r}+\frac{\partial_\theta u_\theta^L}{r}+\frac{(\partial_\theta\xi^L)^2}{2r^2}-\frac{u^\star}{r}\Big|r\di r\di\theta\\
	&
	+ \strokedint_{-\pi L}^{\pi L}\int_{R_0}^{\Rout}
	\Big|\frac{u^\star}{r}-\frac{u_r^L}{r}\Big|r\di r\di\theta\lesssim \frac 1L\,.
\end{split}
\end{equation}
Hence by \eqref{plancherel-derivatives} and the two above estimates we infer 
\begin{equation*}
	\begin{split}
		|\mu^L|((\Rin,\Rout)\times\R)&=
			\mu^L((\Rin,\Rout)\times\R)\\
			&=
		\int_{\Rin}^{\Rout}
		 \sum_{k\in\frac{\mathbb Z}{L}}a^L(r,k)\di r=
		 \strokedint_{-\pi L}^{\pi L}\int_{\Rin}^{\Rout} {(\partial_\theta\xi^L)^2}\di r\di\theta\\
		 &\lesssim
		\strokedint_{-\pi L}^{\pi L}\int_{\Rin}^{\Rout} \frac{(\partial_\theta\xi^L)^2}{2r^2}r\di r\di\theta\lesssim 1\,.
	\end{split}
\end{equation*}
\black
As a consequence, there exists  a (not relabeled) subsequence  and $\mu\in\mathcal{M}^+_b((\Rin,\Rout)\times\R)$ such that 
$\mu^L\stackrel{	*}{\rightharpoonup}\mu$.
Combining \eqref{energy-bound}  with \eqref{plancherel-derivatives} and \eqref{derivatives-a} we find
\begin{equation}\label{eq:compactness}
	\begin{split}
		C_0&\ge  \pi\strokedint_{-\pi L}^{\pi L}\int_{\Rin}^{\Rout}
		\Big(\dot u^\star{\left(\partial_r\xi^L\right)^2}
		+\frac{(\partial_{\theta\theta}\xi^L)^2}{r^4}\Big) r\di r\di\theta
	\\
		& \ge \pi\int_{\Rin}^{\Rout} \Bigg(
		\sum_ {k\in \frac{ \mathbb{Z}}{L},k\ne0}\frac{\dot u^\star}{4k^2}\frac{(\partial_r{a}^L(r,k))^2}{a^L(r,k)}+
		 \sum_ {k\in \frac{\mathbb{Z}}{L}} \frac{a^L(r,k)k^2}{r^4}\Bigg)r\di r\\
		&\ge \pi \int_{\Rin}^{\Rout} \frac{(\dot u^\star)^{1/2}}{r}  \sum_ {k\in \frac{ \mathbb{Z}}{L}}
		|\partial_r{a}^L(r,k) |
		\di r\ge C|\mu^L_{,r}|((\Rin,\Rout)\times\R)\,,
	\end{split}
\end{equation}
where the last inequality follows by Young's inequality. 
 Hence, up to subsequence, we may deduce that there exists $\tilde\mu\in\mathcal{M}_b((\Rin,\Rout)\times\R)$ such that $\mu^L_{,r}\stackrel{*}{\rightharpoonup}\tilde\mu$. Moreover given any $\varphi\in C^\infty_c((\Rin,\Rout)\times\R)$, it holds
\begin{equation*}\begin{split}
		\int_{(\Rin,\Rout)\times\R}\varphi \di\tilde\mu&=\lim_{L\to+\infty}
	\int_{(\Rin,\Rout)\times\R}\varphi \di\mu^L_{,r}\\
	&= - \lim_{L\to+\infty} \int_{(\Rin,\Rout)\times\R}\partial_r \varphi \di\mu^L=-\int_{(\Rin,\Rout)\times\R}\partial_r\varphi\di\mu
	\end{split}
\end{equation*}
which in turn implies $\tilde\mu=\mu_{,r}$. \\

\noindent
\textit{Step 2:} we show that $\mu_{,r}\ll\mu$.  By Remark \ref{rem:int-byparts} \ref{rem-ii} we have that $\mu^L_{,r}\ll\mu^L$. Let $N\in\mathbb N$ be fixed and define the restricted measures $$\mu^L_N:=\mu^L\res(\Rin,\Rout)\times(-N,N)\quad \text{ and } \quad\mu_N:=\mu\res(\Rin,\Rout)\times(-N,N)\,.$$ Then we have that
\begin{equation*}
\mu^L_{N,r}:=\mu_{,r}^L\res(\Rin,\Rout)\times(-N,N)\,, \quad \mu_{N,r}:=\mu_{,r}\res(\Rin,\Rout)\times(-N,N)\,,
\end{equation*}
\begin{equation}\label{eq:abs}
\mu^L_{N,r}\ll\mu^L_{N}\,,\quad (\mu^L_N,\mu^L_{N,r} )  \stackrel{*}{\rightharpoonup}(\mu_N,\mu_{N,r})\,,
\end{equation}
and 
\begin{equation*}
\frac{\di\mu^L_{N,r}}{\di\mu^L_N}(r,k)=\frac{\di\mu^L_{,r}}{\di\mu^L}(r,k)\res(\Rin,\Rout)\times(-N,N)
\,.
\end{equation*}
Hence from the definition of $\mu^L$ and \eqref{eq:compactness} we have 
\begin{equation}\label{eq:bound}
	\begin{split}
	\int_{(\Rin,\Rout)\times(-N,N)}\frac1{4N^2} \Big(\frac{\di\mu^L_{N,r}}{\di\mu^L_N}(r,k)\Big)^2\di\mu^L_N&\le 
\int_{(\Rin,\Rout)\times(-N,N)}\frac1{4k^2} \Big(\frac{\di\mu^L_{,r}}{\di\mu^L}(r,k)\Big)^2\di\mu^L\\
&\le \int_{\Rin}^{\Rout} \sum_ {k\in \frac{ \mathbb{Z}}{L},k\ne0} \frac{1}{4k^2}\frac{(\partial_r{a}^L(r,k))^2}{ a^L(r,k)}\di r\le C\,.
	\end{split}
\end{equation}
From \eqref{eq:abs}, \eqref{eq:bound} and \cite[Example 2.36 pg. 67, and discussion at pg. 66]{AmFuPa:00} we deduce that $\mu_{N,r}\ll\mu_N$ for every $N\in\mathbb N$ and hence $\mu_{,r}\ll\mu$.\\

\noindent
\textit{Step 3:} we show that $\mu\in\mathcal{M}^+_b((\Rin,R_0)\times\R)$, that is, $\mu ((R_0,\Rout]\times\R)=0$, and that 
\begin{equation}\label{eq:constr}
	\int_{(\Rin,R_0)\times\R}\frac1{2r^2}\phi(r)\di\mu=-\int_{\Rin}^{R_0}\frac{u^\star}{r}\phi(r)\di r,
\end{equation}
for all $\phi\in C_c^{\infty}((\Rin,\Rout))$.
For any fixed $0<\delta<R_0/2$ by \eqref{energy-bound} we have 
\begin{equation*}
	\begin{split}
		 \mu^L((R_0-\delta,\Rout)\times\R)&=	\int_{R_0-\delta}^{\Rout}\sum_{k\in \frac{ \mathbb{Z}}{L}} a^L(r,k)\di r\\
		 &
		 = \strokedint_{-\pi L}^{\pi L}\int_{R_0-\delta}^{\Rout}
		 (\partial_\theta\xi^L)^2 \di r\di\theta
		 \\
&	= \strokedint_{-\pi L}^{\pi L}\int_{R_0}^{\Rout}
(\partial_\theta\xi^L)^2 \di r\di\theta
+\strokedint_{-\pi L}^{\pi L}\int_{R_0-\delta}^{R_0}
(\partial_\theta\xi^L)^2 \di r\di\theta
\\
&\le \frac{C}{L}+C\delta\,,
	\end{split}
\end{equation*}
where the last inequality follows by estimating the two terms as in \eqref{xi-theta-outer} and \eqref{xi-theta-inner} respectively and using that $-u^\star\le C$ in $(R_0-\delta,R_0)$ for some $C>0$.
Now using the lower semicontinuity with respect to the weak* convergence  we deduce
\begin{equation*}
	 \mu((R_0,\Rout]\times\R)\le  \mu((R_0-\delta,\Rout)\times\R)\le 
	\liminf_{L\to \infty}\mu^L((R_0-\delta,\Rout)\times\R)\le C\delta\,.
\end{equation*}
By sending $\delta\to0$ we deduce
$
 \mu((R_0,\Rout]\times\R)=0.
$
It remains to show \eqref{eq:constr}. Given $\phi\in C_c^\infty(\Rin,R_0)$ it holds
\begin{equation*}
	\begin{split}
\int_{(\Rin,R_0)\times\R}\frac{\phi(r)}{2r^2}\di\mu^L&= \int_{\Rin}^{R_0}\phi(r)\Big(\frac1{2r^2}\sum_{k\in \frac{ \mathbb{Z}}{L}} a^L(r,k)+\frac{u^\star}{r}\Big)\di r- \int_{\Rin}^{R_0}\frac{u^\star}{r}\phi(r)\di r.\\
	\end{split}
\end{equation*}
From \eqref{plancherel-derivatives}, \eqref{energy-bound} and \eqref{closeness-ur-ustar} it follows that
\begin{equation*}
	\begin{split}
&	\int_{\Rin}^{R_0}|\phi(r)|\Big|\frac1{2r^2}\sum_{k\in \frac{ \mathbb{Z}}{L}} a^L(r,k)+\frac{u^\star}{r}\Big|\di r
\lesssim\|\phi\|_\infty \int_{\Rin}^{R_0}
\Big| \strokedint_{-\pi L}^{\pi L}\frac{(\partial_r\xi^L)^2}{2r^2}\di\theta
+\frac{u^\star}{r}\Big|r\di r
\\
& \lesssim \|\phi\|_\infty \int_{\Rin}^{R_0}
\strokedint_{-\pi L}^{\pi L}  
\Big|\frac{u_r^L}{r}+\frac{\partial_\theta u_\theta^L}{2}+
\frac{(\partial_r\xi^L)^2}{2r^2}\Big|r\di\theta
\di r
+\|\phi\|_\infty \int_{\Rin}^{R_0} \Big|
\strokedint_{-\pi L}^{\pi L} ( {u^\star}-{u_r^L}) \di\theta\Big|\di r
\\
	&
\le\frac{C}{L}\to0, 
	\end{split}
\end{equation*}
as $L\to+\infty$, so that 
\begin{equation}\label{eq:constr1}
\lim_{L\to +\infty} \int_{(\Rin,R_0)\times\R}\frac{\phi(r)}{2r^2}
\di\mu^L= -  \int_{\Rin}^{R_0}\frac{u^\star}{r}\phi(r)\di r\,.
\end{equation}
Next we fix $K\ge1$ and choose $\psi_K\in C^\infty_c(\R)$ such that  $0\le \psi_K\le1$, $\psi_K(k)\equiv 1$ if $|k|<K$ and $\psi_K(k)\equiv 0$ if $|k|>K+1$. Thus
\begin{equation}\label{eq:cutoff}
	\int_{(\Rin,R_0)\times\R}\frac{\phi(r)}{2r^2}
	 \di\mu^L=
	\int_{(\Rin,R_0)\times\R}\frac{\phi(r)}{2r^2}
\psi_K(k)\di\mu^L+ 
		\int_{(\Rin,R_0)\times\R}\frac{	\phi(r)}{2r^2}
(1-\psi_K(k))\di\mu^L\,.
\end{equation}
The weak* convergence implies
\begin{equation*}
\lim_{L\to +\infty}		\int_{(\Rin,R_0)\times\R}\frac{\phi(r)}{2r^2}
\psi_K(k)\di\mu^L=		\int_{(\Rin,R_0)\times\R}\frac{\phi(r)}{2r^2}
\psi_K(k)\di\mu\,,
\end{equation*}
whereas for the second term on the right hand-side of \eqref{eq:cutoff} we estimate
\begin{equation*}
	\begin{split}
		\int_{(\Rin,R_0)\times\R}\frac1{2r^2}
		|\phi(r)(1-\psi_K(k))|\di\mu^L
		&\lesssim\int_{\Rin}^{R_0}\frac1{2r^2}|\phi(r)|\Big(\sum_{k\in \frac{ \mathbb{Z}}{L},|k|>K} a^L(r,k)
	\Big)\di r\\
	&\le\frac{\|\phi\|_\infty}{K^2}
	\int_{\Rin}^{R_0}\frac1{2r^2}\Big(\sum_{k\in \frac{ \mathbb{Z}}{L},|k|>K} a^L(r,k)k^2
	\Big)\di r\\
	&\le \frac{\|\phi\|_\infty}{K^2}
	\int_{\Rin}^{R_0}\fint_{-\pi L}^{\pi L} \frac{(\partial_{\theta\theta}\xi^L)^2}{r^4}r
	\di\theta\di r\le \frac{C}{K^2}\,,
	\end{split}
\end{equation*}
 where the last two inequalities follow from \eqref{plancherel-derivatives} and \eqref{energy-bound}.
We now pass to the limit as $L\to+\infty$ in \eqref{eq:cutoff} and get
\begin{equation*}
	\begin{split}
	&	\int_{(\Rin,R_0)\times\R}\frac{\phi(r)}{2r^2}\psi_K(k)\di\mu-\frac{C}{K^2}\le
	\lim_{L\to+\infty}	\int_{(\Rin,R_0)\times\R}\frac{\phi(r)}{2r^2}\di\mu^L\\
&	\lim_{L\to+\infty}	\int_{(\Rin,R_0)\times\R}\frac{\phi(r)}{2r^2}\di\mu^L
	\le \int_{(\Rin,R_0)\times\R}\frac{\phi(r)}{2r^2}\psi_R(k)\di\mu+ \frac{C}{K^2}\,.
	\end{split}
\end{equation*}
Finally by letting $K\to+\infty$ 
\begin{equation*}
\lim_{L\to+\infty}	\int_{(\Rin,R_0)\times\R}\frac{\phi(r)}{2r^2}\di\mu^L=\int_{(\Rin,R_0)\times\R}\frac{\phi(r)}{2r^2}\di\mu\,,
\end{equation*}
which together with \eqref{eq:constr1} yield \eqref{eq:constr}.
\end{proof}
We now proceed with the $\Gamma-\liminf$ inequality.
\begin{prop}[Lower bound]\label{prop:lower_bound}
	Let $\mathcal F_L$ and $\mathcal F_\infty$ be as in \eqref{def:F_L} and \eqref{def:F_infty} respectively. Let $L>0$ and let $(u_r^L,u_\theta^L,\xi^L)\subset \mathcal{A}_L^{\rm in}\times \mathcal{A}_L^{\rm out}$ be a sequence converging to $\mu\in\mathcal{M}_\infty$ in the sense of Definition \ref{def:convergence}. Then there holds
	\begin{equation*}\label{eq:lowe_bound}
		\liminf_{L\to \infty}\mathcal F_L(u_r^L,u_\theta^L,\xi^L)\ge \mathcal F_\infty(\mu).
	\end{equation*}
\end{prop}
\begin{proof}
%
Let $(u_r^L,u_\theta^L,\xi^L)$ be as in the statement and let
$\mu^L:=\mu^L(\xi^L)$
and $\mu^L_{,r}:=\mu^L_{,r}(\xi^L)$ be defined accordingly to Definition \ref{def:muL}, that is,
\begin{equation*}
	\mu^L= \sum_{
		k\in \frac{\mathbb{Z}}{L}} a^L(r,k)  \L^1\res{(\Rin,\Rout)}\times \delta_k\,,
	\quad
	\mu^L_{,r}= \sum_{
		k\in \frac{\mathbb{Z}}{L}} \partial_ra^L(r,k)  \L^1\res{(\Rin,\Rout)}\times\delta_k\,.
\end{equation*} 
Combining \eqref{plancherel-derivatives} with \eqref{derivatives-a} we find
\begin{equation}\label{eq:lower_bound}
	\begin{split}
	\mathcal F_L&(u_r^L,u_\theta^L,\xi^L)\ge \pi\strokedint_{-\pi L}^{\pi L}\int_{\Rin}^{\Rout} \Big(\dot u^\star(r){\left(\partial_r\xi^L\right)^2}
	+\frac{(\partial_{\theta\theta}\xi^L)^2}{r^4}\Big)r\di r\di\theta\\
	&=\pi \int_{\Rin}^{\Rout} \Big(
	\sum_ {k\in \frac{ \mathbb{Z}}{L},k\ne0}\frac{\dot u^\star(r)}{4k^2}\frac{(\partial_r{a}^L(r,k))^2}{ a^L(r,k)}+\sum_ {k\in \frac{ \mathbb{Z}}{L}} \frac{a^L(r,k)k^2}{r^4}
	\Big)r\di r
	\\
	&=\pi \int_{(\Rin,\Rout)\times\R} \bigg(  \frac{\dot u^\star(r)}{4k^2}\Big(\frac{\di\mu_{,r}^L}{\di\mu^L}(r,k)\Big)^2+\frac{k^2}{r^4}\bigg)r\di\mu^L
	\\
	&=\pi
	 \int_{(\Rin,\Rout)\times\R} 
	\frac{\dot u^\star(r)r}{4k^2}\Bigl(\frac{\di \mu^L}{\di|\tilde\mu^L|}(r,k)\Bigr)^{-1}\Bigl(\frac{\di \mu^L_{,r}}{\di|\tilde\mu^L|}(r,k)\Bigr)^2 \di|\tilde\mu^L|
	+\pi  \int_{(\Rin,\Rout)\times\R} \frac{k^2}{r^3} \di\mu^L\,,
	\end{split}
\end{equation}
where  $\tilde\mu^L:=(\mu^L,\mu^L_{,r})$, and the last equality follows from Remark \ref{rem:limit} \ref{rem:limit(ii)}.
Now  we invoke Reshetnyak Theorem (cf. \cite[Theorem 2.38]{AmFuPa:00}) and obtain
\begin{equation}\label{eq:Resh1}
\liminf_{L\to+\infty}
\int_{(\Rin,\Rout)\times\R} \frac{k^2}{r^3} \di\mu^L\ge \int_{(\Rin,\Rout)\times\R} \frac{k^2}{r^3} \di\mu \,,
\end{equation}
and 
\begin{equation}\label{eq:Resh2}
	\begin{split}
\liminf_{L\to+\infty}	&
\int_{(\Rin,\Rout)\times\R} 	\frac{u^\star(r)r}{4k^2}
\Bigl(\frac{\di \mu^L}{\di|\tilde\mu^L|}(r,k)\Bigr)^{-1}\Bigl(\frac{\di \mu^L_{,r}}{\di|\tilde\mu^L|}(r,k)\Bigr)^2 \di|\tilde\mu^L|
\\ 
&\ge\int_{(\Rin,\Rout)\times\R} \frac{u^\star(r)r}{4k^2}
\Bigl(\frac{\di \mu}{\di|\tilde\mu|}(r,k)\Bigr)^{-1}\Bigl(\frac{\di \mu_{,r}}{\di|\tilde\mu|}(r,k)\Bigr)^2 \di|\tilde\mu|\,,
	\end{split}
\end{equation}
with $\tilde\mu:=(\mu,\mu_{,r})$.
Gathering together \eqref{eq:lower_bound}, \eqref{eq:Resh1} and \eqref{eq:Resh2} we find
\begin{equation*}
	\begin{split}
			\liminf_{L\to \infty}\mathcal F_L(u_r^L,u_\theta^L,\xi^L)\ge 
			 \mathcal{F}_\infty(\mu)\,.
	\end{split}
\end{equation*}
\end{proof}
\section{Upper bound}\label{sec:upb}
In this section we prove the $\Gamma-\limsup$ inequality.
\begin{prop}[Upper bound]\label{prop:upb}
	Let $\mu\in \mathcal{M}_\infty$. Then for $L>0$ there exists a sequence $(u_r^L,u_\theta^L,\xi^L)\in \mathcal{A}_L^{\rm in}\times \mathcal{A}_L^{\rm out}$ that converges
	to $\mu\in\mathcal{M}_\infty$ in the sense of Definition \ref{def:convergence} and such that
	\begin{equation*}
		\limsup_{L\to\infty}\F_L(u_r^L,u_\theta^L,\xi^L)\le \F_\infty(\mu)\,,
	\end{equation*}
	with $\F_L$ and $\F_\infty$ defined as in \eqref{def:F_L} and \eqref{def:F_infty} respectively.
\end{prop}
\noindent 
\textbf{Description of the proof strategy.}
The proof of Proposition \ref{prop:upb} is lengthy and technically involved, and proceeds through several steps which we outline below. While the overall strategy is inspired by \cite{BeMa25}, significant modifications are required to account for the radial geometry and the presence of the nonlinear constraint $ru^\star(r)$.\\

\noindent
Let  $\mu\in \mathcal M_\infty$ be given. Our goal is to construct a recovery sequence $(u_r^L, u_\theta^L, \xi^L) \in \mathcal{A}_L^{\rm in} \times \mathcal{A}_L^{\rm out}$. We begin by defining a sequence of out-of-plane displacements $(\xi^L)_{L>0}$ via their Fourier coefficients $a_k^L$, with $k \in \frac{\mathbb Z}{L_0}$:
\begin{equation*}
	\xi^L(r,\theta)=\sum_{k\in \frac{ \mathbb{Z}}{L_0}, k>0 }a_k^L(r)\sqrt2\sin(k\theta) +
	\sum_{k\in \frac{ \mathbb{Z}}{L_0}, k<0 }a_k^L(r)\sqrt2\cos(k\theta) \,.
\end{equation*}

Here $L_0:=L/n(L)$, where $n(L)\in\mathbb N$ is suitably chosen so that
\[
1\ll L_0\ll L.
\]
The parameter $L_0$ determines the period of the elementary building blocks used in the recovery sequence. It introduces an intermediate scale separating the microscopic wrinkle oscillations, which occur at scale one in the rescaled variables, from the macroscopic radial variation of the limiting measure. This scale makes it possible to localize the construction while keeping the error terms negligible in the limit. Moreover, every $2\pi L_0$-periodic function is also $2\pi L$-periodic, so the above construction is compatible with the admissibility conditions defining $\mathcal A_L^{\rm in}$ and $\mathcal A_L^{\rm out}$.

 \noindent
  The coefficients are required to satisfy the following properties:
\begin{itemize}
	\item  sufficient regularity  to ensure $\xi^L\in \mathcal A_L^{\rm out}$;
	\item fulfilment of the constraint 
	\begin{equation*}A^
		L(r):=\fint_{-\pi L}^{\pi L} \frac{(\partial_\theta\xi^L(r,\cdot))^2}2\di\theta=\frac12\sum_{k\in \frac{ \mathbb Z}{L}}(a_k^L(r))^2k^2\simeq-ru^\star(r) \text{ for a.e. } r\in (\Rin,R_0)\,;
	\end{equation*}
	\item  quantitative control of the derivatives of $\xi^L$, as required for estimating the energy $\mathcal F_L$.
\end{itemize} 
To achieve this, we proceed in several steps.
First, we disintegrate and discretize (Lemma \ref{lem:disint=pushfwd} and Lemma \ref{lem:discretisation}) $\mu$ in the $k$-variable. This yields  a sequence of discrete measures $(\mu^L)_{L>0}$ of the form $$\mu^L=\sum_{k \in \frac{ \mathbb Z}{L}}\bar b^L(r,k)\mathcal L^1\res(\Rin,R_0)\times \delta_k\,.$$

\noindent In the second step,  each coefficient $\bar b^L(r,k)$ is regularised by convolution with a mollifier $\rho_\eps(r)$ at scale $\eps=\eps(L)\to0$. The latter represent the square of the Fourier coefficient of the partial derivative $\partial_\theta\xi^L$. More precisely we define
$$a_k^L(r):= \frac1k\sqrt{ (\bar b^L(\cdot,k)*\rho_\eps)(r)}$$ and construct 
$\xi^L$ accordingly (cf.~Lemma \ref{lem:moll}).  Since the convolution requires $\bar b(r,k)$ to be defined  for all $r\in \R$, we need to extend appropriately so as to preserve the constraint as accurately as possible.\\
The smoothing {step introduces a small error} in the constraint, leading to $$A^L(r)= -ru^\star(r)+o_L(1)\,.$$
 To correct this and enforce the constraint exactly, we rescale $\xi^L$ by the factor $$f^L(r):=\sqrt{\frac{-ru^\star(r)}{A^L(r)}}\,.$$  
Finally, once the out-of-plane displacement is constructed, we define the corresponding in-plane displacement so that the first, third, and fourth terms in \eqref{def:F_L} vanish in the limit as $L \to \infty$. We refer to the proof of Proposition \ref{prop:upb} for the detailed construction.\\

\noindent\textbf{Disintegration and discretisation of $\mu$.}
We recall the notion of disintegration of measures only in the specific form required for this work, and refer to \cite{AmFuPa:00} for a comprehensive treatment of the general theory.
\begin{definition}[Disintegration of measures in the $r$-variable]\label{def:disintegration}
	Let $I\subset\R$ be an interval  and let $\mu\in\mathcal{M}_b(I\times\R)$. We say that the family 
	$$(\nu_r,g(r))_{r\in I} \subset \mathcal{M}_b(\R)\times \R$$ is a disintegration of $\mu$ $($in the $r$-variable$)$ if $r\mapsto\nu_r$ is Lebesgue measurable, $|\nu_r|(\R)=1$ for every $r\in I$,  $g\in L^1(I)$, and 
	\begin{equation*}\label{eq:disint}
		\int_{I\times\R}f(r,k)\di\mu=\int_I\int_{\R}f(r,k)\di\nu_r(k)g(r)\di r\,,
	\end{equation*}
	for every $f\in L^1(I\times\R;|\mu|)$. 
\end{definition}
Formally it simply means $\di\mu(r,k) = \di\nu_r (k) g(r) \di r$.
%
%
%
{We then recall \cite[Lemma 3.2]{BeMa25}}
\begin{lemma}\label{lem:disint=pushfwd}
	Let $I\subset\R$ be an interval  and let $\mu\in\mathcal{M}^+_b(I\times\R)$. Then 
	\begin{equation*}\label{eq:push-forward}
		\int_{I\times\R}\phi(r)\di\mu=	\int_Ig(r)\phi(r)\di r \quad \forall\phi\in C_c^{\infty}(I)\,,
	\end{equation*}
	for some non-negative $g\in L^1(I)$,	if and only if there exists $r\mapsto\nu_r\in\mathcal{M}^+_b(\R)$ Lebesgue measurable such that $(\nu_r,g(r))_{r\in I} $ is a disintegration of $\mu$.
\end{lemma}
%
\begin{cor}[Disintegration of $\mu\in\mathcal{M}_\infty$ in the $r$-variable]\label{cor:disint=pushfwd}
	Let $\mu\in\mathcal{M}_\infty$. Then there exists $r\mapsto\nu_r\in\mathcal{M}^+_b(\R)$ measurable such that $(\nu_r,-2{u^\star(r)}{r})_{r\in (\Rin,R_0)} $ is a disintegration of $\mu$. 
\end{cor}
\begin{proof}
	The proof follows by Lemma \ref{lem:disint=pushfwd} and from the fact that 
	\begin{equation*}\label{constraint1}
		\int_{(\Rin,R_0)\times\R}\phi(r)\di\mu=	\int_{\Rin}^{R_0}	-2{u^\star(r)}{r}\phi(r)\di r\quad \forall\phi\in C_c^{\infty}(\Rin,R_0)\,.
	\end{equation*}
\end{proof}

\begin{lemma}[Discretisation of $\mu\in \mathcal{M}_\infty$ in the $k$-variable]\label{lem:discretisation}
	Let $\mu\in \mathcal{M}_\infty$ with $\F_\infty(\mu)<+\infty$.
	Then there exists   $(\mu^L)\subset \mathcal{M}_\infty$ with the following properties: 
	\begin{enumerate}[label=$(\roman*)$]
		\item \label{(i)discr}$	\mu^L=\sum_{
			k\in \frac{\mathbb{Z}}{L}} \overline b^L(r,k)  \L^1\res{(\Rin, R_0)}\times\delta_k$ 
		with 
		\begin{equation}\label{eq:constraint}
		\overline 	b^L(\cdot,k)\in W^{1,1}(\Rin,R_0)\quad \text{and}\quad
			\sum_{
				k\in \frac{\mathbb{Z}}{L}}\overline  b^L(r,k)=-2ru^\star(r)\,,\quad \forall r\in (\Rin,R_0)\,,
		\end{equation}
		\begin{equation}\label{eq:energy}
			\F_\infty(\mu^L)= \pi 
			\int_{\Rin}^{R_0}	\sum_{
				k\in \frac{ \mathbb{Z}}{L},k\ne0}
			\biggl(
			\frac{\dot u^\star(r)}{4k^2}\frac{(\partial_r\overline b^L(r,k))^2}{\overline  b^L(r,k)}+ \frac{k^2}{r^4}
			\overline  b^L(r,k) \biggr)
	r	\di r\,;
		\end{equation}
		\item\label{(ii)discr} $
		(\mu^L,\mu^L_{,r})\stackrel{*}{\rightharpoonup}(\mu,\mu_{,r})
		$;
		\item\label{(iii)discr} $	\limsup_{L\to\infty}\F_\infty(\mu^L)\le	\F_\infty\big(\mu\big)$.
	\end{enumerate}
\end{lemma}
\begin{proof}
 Let 
$$(\nu_{r},-2ru^\star(r))_{r\in(\Rin,R_0)}$$ be the disintegration provided by Corollary \eqref{cor:disint=pushfwd}.
We introduce a measure $\mu^L\in\mathcal{M}_b^+((\Rin,R_0)\times\R)$ by setting
	\begin{equation*}
		\mu^L:= \sum_{
			k\in \frac{ \mathbb{Z}}{L}}\overline  b^L(r,k) \L^1\res{(\Rin,R_0)}\times\delta_k\,,
	\end{equation*}
	where, for $(r,k)\in(\Rin,R_0)\times \frac{ \mathbb{Z}}{L}$ we set
	\begin{equation}\label{def:coeff}
	\overline 	b^L(r,k):=\begin{cases}
			0&\text{if }k=0\,,\\[1em]
		-2ru^\star(r)\nu_{r}(I_k^L)&\text{if } k\ne0\,,
		\end{cases}
		\quad\text{and}\quad
		I^L_k:=\begin{cases}
			(k-\frac{1}{L},k]&\text{if }k>0\,,\\[1em]
			[k,k+\frac{1}{L})&\text{if }k<0\,.
		\end{cases}
	\end{equation}
For each $k\in\frac{ \mathbb{Z}}{L}$, the map $\overline  b^L(\cdot,k)$ belongs to $W^{1,1}(\lambda,R_0)$, and its derivative satisfies
	\begin{equation*}
\partial_r	\overline 	b^L(r,k)=\begin{cases}
			0&\text{if }k=0\,,\\[1em]
			\displaystyle {-2ru^\star(r)}\int_{I_k^L}\frac{\di\mu_{,r}}{\di\mu}(r,\hat k)
			\di\nu_{r}(\hat k)&\text{if } k\ne0\,.
		\end{cases}
	\end{equation*}
	This identity is immediate when $k=0$. For $k\ne0$, given $\phi\in C^\infty_c(\Rin,R_0)$, the definition of $\nu_r$ together with Remark  \ref{rem:int-byparts}-\ref{rem-i}  yields
	\begin{equation*}
		\begin{split}
			\int_{\Rin}^{R_0} &\overline  b^L(r,k)\dot\phi(r)\di r=\int_{\Rin}^{R_0}
			{-2ru^\star(r)}\dot\phi(r)
			\int_{\R}\chi_{I_k^L}(\hat k)
			\di\nu_{r}(\hat k)\,\di r\\
			&=\int_{(\Rin, R_0)\times\R}\chi_{I_k^L}(\hat k)\dot\phi(r)\di\mu
			 =- \int_{(\Rin,R_0)\times\R}\chi_{I_k^L}(\hat k)\phi(r)\di\mu_{,r}\\
			&=- \int_{(\Rin,R_0)\times\R}\chi_{I_k^L}(\hat k)\phi(r)\frac{\di\mu_{,r}}{\di\mu}(r,\hat k)
			\di\mu\\
			&
		=-\int_{\Rin}^{R_0}{-2ru^\star(r)}\phi(r)\int_{I_k^L}\frac{\di\mu_{,r}}{\di\mu}(r,\hat k)
			\di\nu_{r}(\hat k)\di r\,.
		\end{split}
	\end{equation*}
Applying Young's inequality, we estimate
\begin{equation*}
	\begin{split}
			\int_{\Rin}^{R_0}|\partial_r\overline b^L(r,k)|\di r\le
		\int_{ I_k^L\times(\Rin, R_0)}\left|\frac{\di\mu_{,r}}{\di\mu}\right|\di\mu
		\le \frac12
		\int_{I_k^L\times(\Rin,R_0)} \left(
	\frac1{k^2}\left(\frac{\di\mu_{,r}}{\di\mu}\right)^2+k^2\right)\di\mu\le C\,.
	\end{split}
\end{equation*}
Summing over all $k$ gives
\begin{equation*}
	\int_{\Rin}^{R_0}\sum_{k\in\frac{\mathbb Z}{L}}|\partial_r\overline b^L(r,k)|\di r\le C \mathcal{F}_\infty(\mu)\le C\,.
\end{equation*}
Consequently, $\mu_{,r}^L\in\mathcal{M}_b((\Rin,R_0)\times\R)$ with
	\begin{equation*}
		\mu_{,r}^L= \sum_{
			k\in \frac{\mathbb{Z}}{L}}\partial_r\overline  b^L(r,k)  \L^1\res{(\Rin,R_0)}\times\delta_k\,,
	\end{equation*}
	and Remark \ref{rem:int-byparts} \ref{rem-ii} ensures that $\mu_{,r}^L\ll\mu^L$.\\
	
Since each $\nu_{r}$ is a probability measure, we have 
	\begin{equation*}
		\sum_{
			k\in \frac{ \mathbb{Z}}{L}}\overline  b^L(r,k)={-2ru^\star(r)} \bigg(\sum_{
			k\in \frac{ \mathbb{Z}}{L},k\ne0} \nu_{r}(I_k^L)\bigg)={-2ru^\star(r)}\,,\quad \forall r\in (\Rin,R_0)\,,
	\end{equation*}
which shows that the constraint in the definition of $\mathcal{M}_\infty$ is satisfied, and so $\mu^L\in\mathcal{M}_\infty$. Property \eqref{eq:energy}  follows immediately, concluding the proof of \ref{(i)discr}. To establish \ref{(ii)discr},
	let $\varphi\in C^\infty_c((\Rin,R_0)\times\R)$. Using \eqref{def:coeff} we obtain
	\begin{equation}\label{eq:show1}
		\begin{split}
			\int_{(\Rin,R_0)\times\R}\varphi\di\mu^L
			&=\int_{\Rin}^{R_0} 
			\sum_{
				k\in \frac{ \mathbb{Z}}{L},k\ne0}\varphi(r,k) (-2ru^\star(r))\nu_{r}(I_k^L)\di r\\
			& = \int_{\Rin}^{R_0} 
			\sum_{
				k\in \frac{\mathbb{Z}}{L},k\ne0} \int_{I_k^L}
			\big(	\varphi(r,k)-\varphi(r,\hat k)\big)(-2ru^\star(r)) \di\nu_{r}(\hat k)
				\di r\\&+\int_{(\Rin,R_0)\times\R}\varphi\di\mu\,.
		\end{split}
	\end{equation}
By uniform continuity of $\varphi$, for every $\varepsilon>0$ there exists $L_0>1$ such that for all $L\ge L_0$ 
	\begin{equation*}
		|	\varphi(r,k)-\varphi(r,\hat k)|<\varepsilon\quad\forall r\in(\Rin,R_0)\,,\  \forall k\in\frac{\mathbb Z}{L},\ \forall\hat{k}\in I_k^L \,.
	\end{equation*}
Therefore,
	\begin{equation}\label{eq:show2}
		\begin{split}
	\int_{\Rin}^{R_0} 
\sum_{
	k\in \frac{\mathbb{Z}}{L},k\ne0} &\int_{I_k^L}
\big|	\varphi(r,k)-\varphi(r,\hat k)\big| (-2ru^\star(r))\di\nu_{r}(\hat k)
\di r
\le  \mu((\Rin,R_0)\times\R)\varepsilon\,.
		\end{split}
	\end{equation}
	Combining \eqref{eq:show1}, \eqref{eq:show2}, we deduce $$\mu^L\stackrel{*}{\rightharpoonup}\mu\quad  \text{ as } L\to+\infty\,.$$  
	An analogous argument applied to $\mu^L_{,r}$ gives $$\mu_{,r}^L\times\R)\stackrel{*}{\rightharpoonup}\mu_{,r} \quad\text{ as } L\to+\infty\,.$$
	In remains to show \ref{(iii)discr}. For  $\delta>0$ and  $\hat k\in I_k^L$, we have
 $$k^2\le \left(\hat k+\frac1 L\right)^2\le (1+\delta)\hat k^2+ (1+\delta^{-1})\frac{1}{L^2}\,,$$
from which we deduce that
	\begin{equation*}\label{eq:lim_first_term}
		\begin{split}
			\int_{\Rin}^{R_0}
				\sum_{
				k\in \frac{ \mathbb{Z}}{L}} 	\frac{k^2}{r^3} \overline  b^L(r,k)\di r&
			=\int_{\Rin}^{R_0} 
			\sum_{
				k\in \frac{\mathbb{Z}}{L},k\ne0} \int_{I_k^L} \frac{k^2}{r^3} (-2ru^\star(r)) \di\nu_{r}(\hat k)\, \di r\\
			&= \sum_{
				k\in \frac{ \mathbb{Z}}{L},k\ne0} \int_{(\Rin,R_0)\times I_k^L} \frac{k^2}{r^3}\di\mu(r,\hat k)\\
			&\le \sum_{
				k\in \frac{ \mathbb{Z}}{L},k\ne0} \int_{(\Rin,R_0)\times I_k^L}  \frac1{r^3}\Big({(1+\delta)\hat k^2  
			+(1+\delta^{-1})\frac{1}{L^2}}\Big)
			\di\mu(r,\hat k)\\
			&=(1+\delta)	\int_{(\Rin,R_0)\times\R}\frac{\hat k^2}{r^3}\di\mu
			+\mu((\Rin,R_0)\times\R)(1+\delta^{-1})\frac{C}{L^2}\,.
		\end{split}
	\end{equation*}
For the remaining term, we use that $1/|k|\le 1/|\hat k|$ on $I_k^L$ to get
\begin{equation}\label{claim}
	\frac{1}{2 |k|}
	\frac{\partial_r\overline b^L(r,k)}{\overline  b^L(r,k)}=\frac{1}{2 |k|} \strokedint_{I_k^L}   \frac{\di\mu_{,r}}{\di\mu}(r,\hat k)
	\di\nu_{r}(\hat k)
	\le \strokedint_{I_k^L}  \frac{1}{2 |\hat k|} \frac{\di\mu_{,r}}{\di\mu}(r,\hat k)
	\di\nu_{r}(\hat k)\,.
\end{equation}
Then, rewriting 
$$ \frac{(\partial_r\overline  b^L(r,k))^2}{\overline  b^L(r,k)} =	\bigg(\frac{\partial_r\overline b^L(r,k)}{\overline  b^L(r,k)}\bigg)^2\overline  b^L(r,k)\,,$$
	 and combining this with \eqref{claim} and Jensen's inequality, we estimate
	\begin{equation}\label{eq:lim_second_term_1}
		\begin{split}
		&	\int_{\Rin}^{R_0}	\sum_{
			k\in \frac{ \mathbb{Z}}{L},k\ne0} 	\frac{\dot u^\star(r)}{4k^2} \frac{(\partial_r\overline  b^L(r,k))^2}{\overline  b^L(r,k)}r\di r\\
			&\le\int_{\Rin}^{R_0}\sum_{
				k\in \frac{ \mathbb{Z}}{L},k\ne0}\dot u^\star(r)
			\bigg(	\strokedint_{I_k^L}  \frac{1}{2 |\hat k|} \frac{\di\mu_{,r}}{\di\mu}(r,\hat k)
			\di\nu_{r}(\hat k)\bigg)^2
			\nu_{r}(I_k^L)(-2ru^\star(r))r\di r\\
			&\le 
			\int_{\Rin}^{R_0}\sum_{
				k\in \frac{ \mathbb{Z}}{L},k\ne0}\dot u^\star(r)
			\int_{I_k^L}  \frac{1}{4 \hat k^2}		\bigg( \frac{\di\mu_{,r}}{\di\mu}(r,\hat k)\bigg)^2
			\di\nu_{r}(\hat k) (-2ru^\star(r))
		r\di r\\
			&
			=\int_{(\Rin,R_0)\times\R}\frac{\dot u^\star(r)r}{4 \hat k^2}		\bigg( \frac{\di\mu_{,r}}{\di\mu}(r,\hat k)\bigg)^2\di\mu
			\,.
		\end{split}
	\end{equation}
Combining the two bounds yields
	\begin{equation*}
		\limsup_{L\to\infty}	\F_\infty( \mu^L)
		\le (1+\delta) \F_\infty(\mu)\,.
	\end{equation*}
Letting $\delta\to0$ concludes the proof of \ref{(iii)discr}.\\
\end{proof}

\noindent\textbf{Regularisation via mollification and construction of out-of-plane displacement.} 
This step is the most delicate part of the construction. In the next lemma, we improve the regularity of the Fourier coefficients obtained in Lemma~\ref{lem:discretisation}, which in turn allows us to define a smoother out-of-plane displacement.
	A key point is that the constraint depends on the sum of the squares of the coefficients. For this reason, instead of smoothing the coefficients themselves, we smooth their squares. This choice preserves the structure of the constraint and ensures that it is only slightly affected by the regularisation.
	
	Another important issue is the choice of mollifier. We need a kernel whose derivative can be controlled in terms of the kernel itself, which rules out compactly supported kernels. Exponentially decaying kernels, such as $e^{-|r|}$, have the required property and are therefore a natural choice. However, using such kernels forces us to extend the squared coefficients beyond the interval $(\Rin, R_0)$.
	This extension is straightforward for $r > \Rout$, where we simply set the coefficients to zero. For $r < \Rin$, the construction is more delicate: we effectively shift the region where they are active by dilating the interval $(\Rin, (R_0+\Rin)/2)$ into $(\Rin - \sqrt{\varepsilon}, (R_0+\Rin)/2)$, and then suitably extend the coefficients as constants for $r < r_\varepsilon<\Rin$ (with $r_\eps$ suitably chosen). Although the constraint is not satisfied in these auxiliary regions, the error introduced by the convolution is exponentially small as long as we remain sufficiently far from them.
	On the interval of interest $(\Rin, R_0)$, these regions have little influence: the left boundary lies beyond the effective range of the mollifier, while the neighbourhood of $r = R_0$, where the error is not negligible, is confined to a thin layer that can be controlled. 
	
	To restore the constraint exactly, we then rescale the construction by a correcting factor $f^L$. In addition, we introduce a finer cut-off near $R_0$ to better handle boundary effects (see the proof of Proposition~\ref{prop:upb}).
	Although this procedure—first perturbing the constraint through smoothing and then correcting it—may seem somewhat indirect, it is well adapted to the problem. A key advantage is that the correction factor $f^L$ admits an explicit expression in terms of the mollified quantities, which allows us to control both its size and its derivatives, and ultimately ensures the desired regularity.

\begin{lemma}[Construction of  $\xi$]\label{lem:moll}
		Let $\mu\in\mathcal{M}_\infty$ be such that $\F_\infty(\mu)<+\infty$. Let  $\eps=\eps(L)>0$  and $n=n(L)\in \mathbb{N}$ satisfy
		$$\lim_{L\to +\infty}\eps(L)=0\,,\quad
		\lim_{L\to +\infty}n(L)=\lim_{L\to +\infty}\frac L{n(L)}=+\infty\,.$$
	Define $L_0:=L/n(L)$.
	Then there exists $$\hat \xi^L\in \mathcal{A}_L^{\rm out}\cap \mathcal{A}_{L_0}^{\rm out}$$   such that, for every $r\in (\Rin,R_0)$, setting
		$$A^L(r):=\fint_{-\pi L}^{\pi L}\frac{(\partial_\theta\hat{\xi}^L(r,\cdot))^2}{2}\di \theta\,,
	\quad	f^L(r):=\sqrt{\frac{-ru^\star(r)}{A^L(r)}}\,,$$
	the following properties hold.
	
	\medskip 
	\noindent\textbf{(i) Size of the constraint and correction factor.}
			\begin{equation}\label{eq:AL}\begin{split}
\max\{(R_0-r),\eps\}
		\lesssim A^L(r) \lesssim  	\max\{(R_0-r),\eps\}\,,
				\end{split}
		\end{equation}
	\begin{equation}\label{eq:f(x)0}
(	f^L(r))^2\lesssim \frac{(R_0-r)}{\max\{(R_0-r),\eps\}}\,.
	\end{equation}
	Moreover, for every $N\in\mathbb{N}$,
		{ \begin{equation}\label{eq:f(x)}
				\begin{cases}
		(f^L(r))^2= 1+o_L(1)&\text{ if }r\in (\Rin,(\Rin+R_0)/2)\\[1em]
			(	f^L(r))^2= 1+o_N(1)& \text{ if }r\in ((\Rin+R_0)/2,R_0-N\eps)
		\\[1em]
		(f^L(r))^2\le 1&\text{ if }r\in (R_0-N\eps,R_0)
				\end{cases}\,.
		\end{equation}}
		%
		\medskip
		\noindent\textbf{(ii) Estimates on derivatives of $f^L$.}
		\begin{equation}\label{eq:f'(x)-f''(x)}
			\begin{split}
			&	(\dot	f^L(r))^2\lesssim
			\frac{\max\{e^{-\frac {r-R_0}\eps}, e^{-\frac{1}{\sqrt\eps}},\eps{ e^{-\frac{|(R_0+\Rin)/2-r|}{\eps}} }\} }{(R_0-r)\eps}\,, \\[1em]& 
			(	\ddot f^L(r))^2\lesssim
			\frac1{(R_0-r)^3\eps}
			\,.
			\end{split}
		\end{equation}
	\medskip
	\noindent\textbf{(iii) Bounds on $\hat \xi^L$.}
	{	There exists a continuous increasing function $\omega\colon[0,+\infty)\to[0,+\infty)$ with $\omega(0)=0$ such that,}
	{if $a_\eps\to0$, $a_\eps\ge N\eps$, then	\begin{equation}\label{eq:est-u}
		\fint_{-\pi L}^{\pi L}(\hat \xi^L(r,\cdot))^2\di\theta \lesssim
		\begin{cases}
				\max\{(R_0-r),\eps\}	(\omega(2a_\eps)+\frac{a_\eps}{\eps}e^{-\frac{a_\eps}{\eps}})& \text{ if }r\in [R_0-a_\eps,R_0)\\
				R_0-r&\text{ if }r\in (\Rin,R_0-a_\eps)
			\end{cases}	\,,
	\end{equation}}
	%
	\begin{equation}\label{eq:est-u^4}
		\fint_{-\pi L}^{\pi L}(\hat \xi^L(r,\cdot))^4\di\theta \lesssim
		L_0^2
	(\max\{(R_0-r),\eps\})^2
		\,,
	\end{equation}
	\begin{equation}\label{eq:uyy+ux}
		\fint_{-\pi L}^{\pi L}
		\left( (\partial_{\theta\theta}\hat \xi^L(r,\cdot))^2 + (\partial_r\hat \xi^L(r,\cdot))^2\right)
		\di\theta\lesssim\frac1\eps\,,
	\end{equation}
		\begin{equation}\label{eq:uxx-uxy-uxyy}
		\fint_{-\pi L}^{\pi L}\int_{\Rin}^{\Rout}\left((\partial_{r\theta}\hat \xi^L)^2+(\partial_{rr}\hat \xi^L)^2+ 	(\partial_{\theta\theta r}\hat \xi^L)^2\right)r\di r\di\theta
		\lesssim \frac 1{\eps^2}\,,
	\end{equation}
	\begin{equation}\label{eq:ux}
		\fint_{-\pi L}^{\pi L}\int_{\Rin}^{\Rout}(\partial_r\hat \xi^L)^4r\di r\di\theta\lesssim \frac{L_0^2}{\eps^2}\,.
	\end{equation}
	\medskip
	\noindent\textbf{(iv) Convergence and energy estimate.}
				\begin{equation}\label{eq:weak-conv}
				(\mu^L(\hat \xi^L),\mu^L_{,r}(\hat \xi^L))\stackrel{*}{\rightharpoonup}(\mu,\mu_{,r})\quad \text{in }\mathcal M_b((\Rin,\Rout)\times\R)^2\,,
			\end{equation}
			\begin{equation}\label{eq:limsup}
				\limsup_{L\to\infty}
				\fint_{-\pi L}^{\pi L}\int_{\Rin}^{\Rout}\Big(\dot u^\star(\partial_r\hat \xi^L)^2+ \frac{(\partial_{\theta\theta}\hat \xi^L)^2}{r^4}
				\Big)r \di r\di\theta\le \mathcal{F}_\infty(\mu)\,.
			\end{equation}
\medskip
Finally, since $\hat \xi^L$ and all its derivatives are $2\pi L_0$-periodic in $\theta$, all the above estimates remain valid if the averages over $[-\pi L,\pi L]$ are replaced by averages over $[-\pi L_0,\pi L_0]$.
\end{lemma}
\begin{proof}
		Let $\mu\in\mathcal{M}_\infty$, $\eps=\eps(L)$, $n=n(L)$ and $L_0$ be as in the statement.

	We  construct  $\hat \xi^L\in \mathcal A_{L_0}^{\rm out}$ and then we extend it periodically in  $(\Rin,\Rout)\times[-\pi L,\pi L]$, without relabelling it.  
Let $(\mu^{L_0})\subset \mathcal{M}_\infty$ be the sequence of discrete measures provided by Lemma \ref{lem:discretisation} for the parameter $L_0$
	$$	\mu^{L_0}=\sum_{
		k\in \frac{\mathbb{Z}}{L_0}} \overline b^{L_0}(r,k)  \L^1\res{(\Rin,R_0)}\times\delta_k\,.$$ 
By the mean value theorem, for each $\eps$, we can find $\lambda=\lambda(\eps)\in(\Rin,\Rin+\sqrt\eps/2)$ such that
	\begin{equation}\label{ub:1}
		\sum_{
			k\in \frac{\mathbb{Z}}{L_0}} \frac{k^2}{{\lambda}^3} \overline b^{L_0}({\lambda},k)\le \fint_{\Rin}^{\Rin+\sqrt\eps/2} \sum_{
			k\in \frac{\mathbb{Z}}{L_0}}\frac{k^2}{r^3} \overline b^{L_0}(r,k)\di r\,.
	\end{equation}
Let $M_0:=(\Rin+R_0)/2$ and let { $\ell_\eps\colon[\Rin,M_0 ]\to [\Rin-\sqrt\eps,M_0 ]$ be the linear dilation function such that $$\ell_\eps(\Rin)=\Rin-\sqrt\eps\quad \text{and}\quad \ell_\eps(M_0)=M_0\,,$$
	namely
	\begin{equation*}
		\ell_\eps(r):=m_\eps\left(r- M_0\right)+M_0\,,\quad m_\eps:=\frac{{R_0}-\Rin+2\sqrt\eps}{{R_0}-\Rin}\searrow1\,.
	\end{equation*}
} Observe that $r_\eps:=\ell_\eps(\lambda)<\Rin.$
We define also $g_\eps\colon [\Rin,R_0]\to [\Rin,R_0-\sqrt\eps]$ to be 
\begin{equation*}
	g_\eps(r):=\begin{cases}
	\ell_\eps(r)& \text{ if }r\in[\Rin,M_0]\\[1em]
	r&\text{ if }r\in [M_0,R_0]
	\end{cases}\,.
\end{equation*}
Then we define $ b^{L_0}\colon\R\times\frac{\mathbb{Z}}{L_0}\to \R$ as 
	\begin{equation}\label{def:b}
		b^{L_0}(r,k):=\begin{cases}\bar b^{L_0}(\lambda,k)&\text{if }r\le r_\eps\\[1em]
		\bar b^{L_0}(g_\eps^{-1}(r),k)	&r_\eps<r<R_0   \\[1em]
		0&\text{if }r\ge R_0
		\end{cases}\,.
	\end{equation}
This operation allows to mollify the coefficients without changing the constraint too much. 
 Let now $\rho(t)=\frac12 e^{-|t|}$ and $\rho_\eps(r):=\frac1\eps\rho(\frac r\eps)=\frac{1}{2\eps}e^{\frac{-|r|}{\eps}}$ 
and note that in particular 
	\begin{equation}\label{eq:prop-moll}
		|\dot\rho_\eps(r)|= \frac1\eps\rho_\eps(r)\,.
	\end{equation}
	Finally we let   $ a^{L}\colon\R\times\frac{\mathbb{Z}}{L_0}\to \R$ be defined as
	\begin{equation*}
		a^{L}(r,k):= (b^{L_0}(\cdot,k)*\rho_\eps)(r)\,,
	\end{equation*}
	and   $\hat \xi^L\in\mathcal{A}_{L_0}^{\rm out}$  be the  function 
	\begin{equation*}
	\hat	\xi^{L}(r,\theta):= \sum_ {k\in \frac{ \mathbb{Z}}{L_0}, k>0 } \frac{\sqrt{ a^{L}(r,k)}}{k}\sqrt2\sin(k\theta) + \sum_ {k\in \frac{ \mathbb{Z}}{L_0}, k < 0 }  \frac{\sqrt{ a^{L}(r,k)}}{k}\sqrt2\cos(k\theta)\,.
	\end{equation*}
Since $L=n(L)L_0$ it follows  $	\hat	\xi^{L}(r,\cdot)$ is $2\pi L$-periodic and thus  $\hat \xi^L\in\mathcal{A}_{L}^{\rm out}$. 
We next organize the proof into a number of steps. 

\noindent
\textit{Step 1:} in this step we show \eqref{eq:AL}--\eqref{eq:f(x)}.
Since $-ru^\star(r)=-\Tin\Rin r\log\frac r{R_0}$ in  $(\Rin,R_0)$, an easy calculation shows
\begin{equation}\label{linear-growth-constr}
(R_0-r)\lesssim -ru^\star(r)\lesssim (R_0-r)\quad \text{ in }(\Rin,R_0)\,.
\end{equation}
Next,  from \eqref{plancherel-derivatives} and  \eqref{eq:constraint} we have 
\begin{equation}\label{est:A}
	\begin{split}
		A^L(r)&=\frac12\fint_{-\pi L}^{\pi L}(\partial_\theta\hat{\xi})^2\di\theta= \frac12
\fint_{-\pi L_0}^{\pi L_0}(\partial_\theta\hat{\xi})^2\di\theta
\\&	= \frac12
	\sum_{k\in \frac{ \mathbb{Z}}{L_0}}{a}^{L}(r,k)
= \frac12\Big(\sum_{k\in \frac{ \mathbb{Z}}{L_0}}{b}^{L_0}(\cdot,k)\Big)*\rho_\eps(r)\\
&= -\big(g_\eps^{-1}(r)u^\star(g_\eps^{-1}(r))\chi_{(r_\eps,R_0)}+\lambda u^\star(\lambda)\chi_{(-\infty,r_\eps)}\big)*\rho_\eps(r)
\,.
	\end{split}
\end{equation}
We next prove that in $(\Rin,R_0)$
\begin{equation}\label{est:Abis}
	\begin{split}
	 -\big(g_\eps^{-1}(r)u^\star(g_\eps^{-1}(r))\chi_{(r_\eps,R_0)}+\lambda& u^\star(\lambda)\chi_{(-\infty,r_\eps)}\big)*\rho_\eps(r)
		\\&=
		-g_\eps^{-1}(r)u^\star(g_\eps^{-1}(r))\chi_{(\Rin,R_0)}+\mathcal R_\eps(r)
		\,,
	\end{split}
\end{equation}
where 
\begin{equation*}\begin{split}
	 \mathcal{R}_\eps(r)&:= 
-\frac{\eps}{2}
\Bigg[- \Tin\Rin e^{\frac{|r-R_0|}{\eps}}
+m_\eps^{-1}(\Tin\Rin+ u^\star(\lambda))e^{-\frac{|r-r_\eps|}{\eps}} \\&
+ 
(1-m_\eps^{-1})(\Tin\Rin+u^\star(M_0))e^{-\frac{|r-M_0|}{\eps}}
+\Tin\Rin \int_{\lambda}^{R_0}\frac1te^{-\frac{|r-g_\eps(t)|}{\eps}}(\dot g_\eps(t))^{-1}\dt\Bigg]
\,.
	\end{split}
\end{equation*}
{
More precisely, we have 
\begin{equation}\label{const1}
	\begin{split}
	-	\lambda	u^\star(\lambda)\chi_{(-\infty,r_\eps)}*\rho_\eps(r)= 
		&-\frac{1}{2\eps}\lambda  u^\star(\lambda)\int_{-\infty}^{r_\eps} e^{\frac{s-r}{\eps}}\ds
		=-\frac12 \lambda  u^\star(\lambda)e^{\frac{r_\eps-r}{\eps}}\,;
	\end{split}
\end{equation}
and
\begin{equation}\label{const2}
	\begin{split}
		-	g_\eps^{-1}(r)u^\star(	g_\eps^{-1}(r))\chi_{( r_\eps,R_0)}*\rho_\eps(r)
		&=-\frac1{2\eps}\int_{r_\eps}^{R_0}u^\star(g_\eps^{-1}(s))g_\eps^{-1}(s)e^{-\frac{|r-s|}{\eps}}\ds
		\\
		&=-\frac1{2\eps}\int_{\lambda}^{R_0}u^\star(t)t e^{-\frac{|r-g_\eps(t)|}{\eps}}\dot g_\eps(t)\dt
		\\
		&=- \frac{m_\eps}{2\eps}\int_{\lambda}^{M_0}u^\star(t)t e^{-\frac{|r-\ell_\eps(t)|}{\eps}}\dt 
		- \frac1{2\eps}\int_{M_0}^{R_0}u^\star(t)t e^{-\frac{|r-t|}{\eps}}\dt\,.
	\end{split}
\end{equation}
Now we distinguish between $r\in (\Rin,M_0)$ and $r\in (M_0,R_0)$.
 If $r\in (\Rin, M_0)$, integrating by parts  and using that $\dot u^\star(t)=C/t$ with $C=\Tin\Rin $ we get
\begin{equation}\label{const3}
	\begin{split}
- \frac{m_\eps}{2\eps}\int_{\lambda}^{M_0}&u^\star(t)t e^{-\frac{|r-\ell_\eps(t)|}{\eps}}\dt 	- \frac1{2\eps}\int_{M_0}^{R_0}u^\star(t)t e^{-\frac{|r-t|}{\eps}}\dt
		=\\&
	-	\frac{m_\eps}{2\eps}\int_{\lambda}^{\ell_\eps^{-1}(r)}u^\star(t)te^{\frac{\ell_\eps(t)-r}{\eps}}\dt
		-\frac{m_\eps}{2\eps}\int_{\ell_\eps^{-1}(r)}^{M_0}u^\star(t)te^{\frac{r-\ell_\eps(t)}{\eps}}\dt 	- \frac1{2\eps}\int_{M_0}^{R_0}u^\star(t)t e^{\frac{r-t}{\eps}}\dt
		\\
		&=-\ell_\eps^{-1}(r)u^\star(\ell_\eps^{-1}(r))+ \frac12u^\star(\lambda)(\lambda-m_\eps^{-1}\eps)e^{\frac{r_\eps-r}{\eps}}-\frac12u^\star(M_0)(1-m_\eps^{-1})\eps e^{\frac{r-M_0}{\eps}}
		\\
		&+ \frac12\int_{\lambda}^{\ell_\eps^{-1}(r)}\dot u^\star(t)(t-m_\eps^{-1}\eps)e^{\frac{\ell_\eps(t)-r}{\eps}}\dt
		- \frac12\int_{\ell_\eps^{-1}(r)}^{M_0}\dot u^\star(t)(t+m_\eps^{-1}\eps)e^{\frac{r-\ell_\eps(t)}{\eps}}\dt\\&-\frac12\int_{M_0}^{R_0}\dot u^\star(t)(t+\eps)e^{\frac{r-t}{\eps}}\dt
		\\
		&= -\ell_\eps^{-1}(r)u^\star(\ell_\eps^{-1}(r))+ \frac12u^\star(\lambda)(\lambda-m_\eps^{-1}\eps)e^{\frac{r_\eps-r}{\eps}}-\frac12u^\star(M_0)(1-m_\eps^{-1})\eps e^{\frac{r-M_0}{\eps}}\\
		&-\frac12 C\eps\left(m_\eps^{-1}e^{\frac{r_\eps-r}{\eps}}-e^{\frac{r-R_0}{\eps}}+(1-m_\eps^{-1})e^{\frac{r-M_0}{\eps}}
		+ \int_{\lambda}^{R_0}\frac1te^{-\frac{|g_\eps(t)-r|}{\eps}}(\dot g_\eps(t))^{-1}\dt
	\right)
		\,.
	\end{split}
\end{equation} 
\noindent
If $r\in (M_0,R_0)$ in a similar manner we get 
\begin{equation}\label{const4}
	\begin{split}
		- \frac{m_\eps}{2\eps}\int_{\bar\lambda}^{M_0}&u^\star(t)t e^{-\frac{|r-\ell_\eps(t)|}{\eps}}\dt 	- \frac1{2\eps}\int_{M_0}^{R_0}u^\star(t)t e^{-\frac{|r-t|}{\eps}}\dt
		\\
		&= -ru^\star(r)+ \frac12u^\star(\lambda)(\lambda-m_\eps^{-1}\eps)e^{\frac{r_\eps-r}{\eps}}-\frac12u^\star(M_0)(1-m_\eps^{-1})\eps e^{\frac{M_0-r}{\eps}}\\
		&-\frac12 C\eps\left(m_\eps^{-1}e^{\frac{r_\eps-r}{\eps}}-e^{\frac{r-R_0}{\eps}}+(1-m_\eps^{-1})e^{\frac{M_0-r}{\eps}}
			+ \int_{\lambda}^{R_0}\frac1te^{-\frac{|g_\eps(t)-r|}{\eps}}(\dot g_\eps(t))^{-1}\dt
		\right)
		\,.
	\end{split}
\end{equation} 
}

Therefore, by combining \eqref{const1}--\eqref{const4} we deduce \eqref{est:Abis}. As a consequence we have that
\begin{equation*}
	\mathcal{R}_\eps(r)\to0 \quad\text{ and }\quad A^L(r)\to-2ru^\star(r)\chi_{(\Rin,R_0)}\,.
\end{equation*}
{ Furthermore
\begin{equation*}
|\mathcal R_\eps(r)|\lesssim \eps\max\left\{e^{-\frac{|r-r_\eps|}{\eps}}, e^{-\frac{|r-R_0|}{\eps}},\sqrt\eps  e^{-\frac{|r-M_0|}{\eps}}
\right\}\,,
\end{equation*}
and, since $u^\star<0$ in $(\Rin,R_0)$
\begin{equation*}
	\mathcal{R}_\eps(r)\gtrsim \eps \left(
	e^{-\frac{|r-R_0|}{\eps}} -e^{-\frac{|r-r_\eps|}{\eps}}
	-\sqrt\eps   e^{-\frac{|r-M_0|}{\eps}}
	\right)\,.
\end{equation*}
}
From this we deduce
\begin{equation*}
\begin{split}
	A^L(r)\le |A^L(r)|&\le -ru^\star(r)\chi_{(M_0,R_0)}-\ell_\eps^{-1}(r)u^\star(\ell_\eps^{-1}(r))\chi_{(\Rin,M_0)}+|\mathcal R_\eps(r)|\\
	&\lesssim (R_0-r)\chi_{(M_0,R_0)} + (R_0-\ell_\eps^{-1}(r))\chi_{(\Rin,M_0)}+\eps\\&
\lesssim \max\{(R_0-r),\eps\}\,.
\end{split}
\end{equation*}
For the lower bound in \eqref{eq:AL} we have
\begin{equation*}
	\begin{split}
		A^L(r)&\gtrsim (R_0-r)\chi_{(M_0, R_0)} + (R_0-\ell_\eps^{-1}(r))\chi_{(\Rin,M_0)}+
		\mathcal R_\eps(r) \\&\gtrsim (R_0-r)+\mathcal{R}_\eps(r)\,.
	\end{split}
\end{equation*}
Next we distinguish the following cases:
 If $r\in (R_0-\eps,R_0)$
\begin{equation*}
	\begin{split}
	A^L(r) \gtrsim (R_0-r)+ \eps \left(e^{-1}-e^\frac{r_\eps-R_0+\eps}{\eps}-\sqrt\eps
		\right)\gtrsim \eps\gtrsim \max\{(R_0-r),\eps\}\,;
	\end{split}
\end{equation*}
 If $r\in (\Rin, R_0-\eps)$
\begin{equation*}
	\begin{split}
	A^L(r)& \gtrsim (R_0-r)+  \eps \left(e^{-\frac{\Rin-R_0}{\eps}}-e^{-\frac{1}{\sqrt\eps}}-\sqrt\eps 
		\right)
		\\&\gtrsim (R_0-r)+O(\eps\sqrt\eps)\gtrsim (R_0-r)
\gtrsim \max\{(R_0-r),\eps\}\,.
	\end{split}
\end{equation*}
 {Then \eqref{eq:AL} and \eqref{eq:f(x)0} readily follow.}
{Now we prove \eqref{eq:f(x)}. 
In $(\Rin,M_0)$ a straightforward computation gives
	\begin{equation*}
		(f^L(r))^2=1+\frac{-ru^\star(r)+\ell_\eps^{-1}(r)u^\star(\ell_\eps^{-1}(r))-\mathcal{R}_\eps(r)}{A^L(r)}
		=1+o_L(1)\,,
	\end{equation*}
{where the last equality follows from $|-ru^\star(r)+\ell_\eps^{-1}(r)u^\star(\ell_\eps^{-1}(r))|+|\mathcal R_\eps(r)|\lesssim \sqrt\eps$ and \eqref{eq:AL}.}
If $r\in (M_0 ,R_0-N\eps)$ 	for $N\in \mathbb N$ fixed it holds 
	\begin{equation*}
		\begin{split}
			(f^L(r))^2&=\frac{-ru^\star(r)}{A^L(r)}
			=1- \frac{\mathcal R_\eps(r)}{A^L(r)}=1+o_N(1)
			\,,
		\end{split}
	\end{equation*}
where the last equality follows from
	\begin{equation*}
		\left|\frac{\mathcal R_\eps(r)}{A^L(r)}\right|= 
		\frac{\left|\mathcal R_\eps(r)\right|}{A^L(r)}\lesssim \frac{\eps
			\max\left\{e^{-\frac{|r-r_\eps|}{\eps}}, e^{-\frac{|r-R_0|}{\eps}},{\sqrt\eps}
			\right\}
		}{\max\{(R_0-r),\eps\}}\lesssim \frac{\eps e^{-N}}{N\eps}\lesssim\frac{e^{-N}}{N}=o_N(1)\,.
	\end{equation*}
Finally,	if $r\in (R_0-N\eps,R_0)$ we have 
\begin{equation*}
	\mathcal R_\eps(r)\gtrsim \eps\left(
	e^{-N}-e^{\frac{r_\eps-R_0}{\eps}} -\sqrt\eps	\right)
	\gtrsim \eps e^{-N}\ge0\,,
\end{equation*}
which implies
\begin{equation*}
	(f^L(r))^2=\frac{-ru^\star(r)}{-ru^\star(r)+\mathcal R_\eps(r)
	}\le 1
	\,.
\end{equation*}
}

%
%

\noindent
\textit{Step 2:} in this step we show \eqref{eq:f'(x)-f''(x)}. A direct computation shows that 
\begin{equation*}
\begin{split}
(\dot f^L(r))^2&=\frac{\left[ \frac{\di}{\di r}{(-ru^\star(r))} A^L(r)+ru^\star(r)\dot A^L(r)\right]^2}{-4ru^\star(r)(A^L(r))^3}
\\ 
&\lesssim
\frac{\left[ \frac{\di}{\di r}{(-ru^\star(r))} \ell_\eps^{-1}(r)u^\star(\ell_\eps^{-1}(r))+ru^\star(r) \frac{\di}{\di r}{(\ell_\eps^{-1}(r)u^\star(\ell_\eps^{-1}(r)))}
	\right]^2}{-4ru^\star(r)(A^L(r))^3}\chi_{(\Rin,M_0)}\\&
+
\frac{\left( \frac{\di}{\di r}{(-ru^\star(r))} \mathcal R_\eps(r)+ru^\star(r)\dot{ \mathcal R}_\eps(r)\right)^2}{-4ru^\star(r)(A^L(r))^3}
\\
&\lesssim { \frac{\eps}{(R_0-r)^4}\chi_{(\Rin,M_0)}  }+
\frac{(\max\{R_0-r,\eps\})^2 \max\{e^{\frac {r-R_0}\eps}, e^{\frac{r_\eps-r}{\eps}},\eps {e^{-\frac{|M_0-r|}{\eps}} }\}}{(R_0-r)(\max\{R_0-r,\eps\})^3} 
\\&\lesssim \frac{\max\{e^{\frac {r-R_0}\eps}, e^{-\frac{1}{\sqrt\eps}},\eps { e^{-\frac{|M_0-r|}{\eps}} } \} }{(R_0-r)\eps}
\,.
\end{split}
\end{equation*}
%
Furthermore, by Young's inequality and using \eqref{linear-growth-constr}, \eqref{eq:AL} and that
\begin{equation*}
	\Big|\frac{\di}{\di r}(ru^\star(r))\Big|
	+\Big| \frac{\di^2}{\di r^2}(ru^\star(r))\Big|\lesssim 1\,,
	\end{equation*}
\begin{equation*}
	  |\dot A^L(r)|\lesssim 1\,,\quad |\ddot A^L(r)|\lesssim \eps^{-1}\max\left\{e^{-\frac{r_\eps-r}{\eps}}, e^{-\frac{r-R_0}{\eps}},\sqrt\eps
	\right\}\,,
\end{equation*}
 we have
\begin{equation*}
\begin{split}
(\ddot f^L(r))^2&\lesssim 
\frac{1}{(R_0-r)A^L(r)}+
\frac{(R_0-r)(\ddot A^L(r))^2}{(A^L(r))^3}\\&+ \frac{1}{(R_0-r)^3A^L(r)}+ \frac{(\dot A^L(r))^2}{(R_0-r)(A^L(r))^3}+ \frac{(R_0-r)(\dot A^L(r))^4}{(A^L(r))^5}\\
&\lesssim\frac{1}{(R_0-r)\max\{R_0-r,\eps\}}+
 \frac{(R_0-r)\eps^{-2}\max\left\{e^{-\frac{r_\eps-r}{\eps}}, e^{-\frac{r-R_0}{\eps}},\eps
 	\right\}}{(\max\{R_0-r,\eps\})^3}
 \\
 &+
\frac1{(R_0-r)^3\max\{R_0-r,\eps\}}+ \frac1{(R_0-r)(\max\{R_0-r,\eps\})^3}+ \frac{R_0-r}{(\max\{R_0-r,\eps\})^5}\,.
\end{split}
\end{equation*}
Hence 
\begin{equation*}
(\ddot f^L(r))^2\lesssim \frac{1}{(R_0-r)\eps}+ \frac{1}{\eps^4}+ \frac1{(R_0-r)^3\eps} + \frac{1}{(R_0-r)\eps^3}\lesssim \frac1{(R_0-r)^3\eps}\quad \text{ if }r\in(R_0-\eps,R_0)\,,
\end{equation*}
and 
\begin{equation*}
(\ddot f^L(r))^2\lesssim \frac{1}{(R_0-r)^2}+
 \frac{\eps}{(R_0-r)^2\eps^2}+ \frac1{(R_0-r)^4}\lesssim \frac1{(R_0-r)^3\eps}\quad \text{ if }r\in(\Rin,R_0-\eps)\,.
\end{equation*}

 \noindent 
\textit{Step 3:} in this step we show \eqref{eq:est-u}. By \eqref{eq:plancherel-u} it holds
\begin{equation}\label{eq:u^2-1}
	\begin{split}
		\fint_{-\pi L}^{\pi L}(\hat \xi^L(r,\cdot))^2\di r=
		\fint_{-\pi L_0}^{\pi L_0}(\hat \xi^L(r,\cdot))^2\di r&=\sum_ {k\in \frac{\mathbb{Z}}{L_0},k\ne 0 }  \frac{{ a^{L}(r,k)}}{k^2} \\&= 
		\bigg(\sum_ {k\in \frac{ \mathbb{Z}}{L_0},k\ne 0 }  \frac{{ b^{L_0}(\cdot,k)}}{k^2} \bigg) *\rho_\eps(r) \\
		&
		=\int_{-\infty}^{R_0} \bigg(\sum_ {k\in \frac{ \mathbb{Z}}{L_0},k\ne 0 }  \frac{{ b^{L_0}(z,k)}}{k^2} \bigg)\rho_\eps(r-z)\dz
		\,.
	\end{split}
\end{equation}
Combining the fundamental theorem of calculus and Hölder's inequality it holds
\begin{equation}\label{eq:u^2-2}
	b^{L_0}(z,k)=
	\left(\sqrt{	b^{L_0}(z,k) }\right)^2=
	\left(\int_z^{R_0}\frac{\partial_rb^{L_0}(\hat z,k)}{2\sqrt{b^{L_0}(\hat z,k)}}\,{\rm d}\hat z\right)^2
	\le (R_0-z)  \int_z^{R_0} \frac{\partial_r(b^{L_0}(\hat z,k))^2}{4{b^{L_0}(\hat z,k)}}\,{\rm d}\hat z
	\,.
\end{equation}
From \eqref{eq:u^2-1} and \eqref{eq:u^2-2} it follows
\begin{equation*}
	\begin{split}
		\fint_{-\pi L}^{\pi L}(\hat \xi^L(r,\cdot))^2\di \theta&\le  \int_{-\infty}^{R_0}
		\bigg(
		\sum_ {k\in \frac{ \mathbb{Z}}{L_0},k\ne 0 } 
		\int_z^{R_0} \frac{(\partial_rb^{L_0}(\hat z,k))^2}{4k^2{b^{L_0}(\hat z,k)}}\,{\rm d}\hat z
		\bigg)(R_0-z)\rho_\eps(r-z)
		\dz\,.
	\end{split}
\end{equation*}
Now we recall the definition of $b^{L_0}$ \eqref{def:b}
\begin{equation*}\label{eq:u^2-3}
	\begin{split}
		\sum_ {k\in \frac{ \mathbb{Z}}{L_0},k\ne 0 } 	\int_z^{R_0} \frac{(\partial_rb^{L_0}(\hat z,k))^2}{4k^2{b^{L_0}(\hat z,k)}}\,{\rm d}\hat z &= 	\sum_ {k\in \frac{ \mathbb{Z}}{L_0},k\ne 0 } 
		\int_{z\vee r_\eps}^{R_0} \frac{(\partial_r\bar b^{L_0}(g_\eps^{-1}(\hat z),k))^2}{4k^2{\bar b^{L_0}(g_\eps^{-1}(\hat z),k)}}\,{\rm d}\hat z\\
		&\le m_\eps	\sum_ {k\in \frac{ \mathbb{Z}}{L_0},k\ne 0 } 
		\int_{g_\eps^{-1}(z)\vee\lambda}^{R_0} \frac{(\partial_r\bar b^{L_0}(t,k))^2}{4k^2{\bar b^{L_0}(t,k)}}\dt\\
		&
		\le m_\eps
		\int_{(g_\eps^{-1}(z)\vee\lambda,R_0)\times\R}\frac{1}{4  k^2}		\bigg( \frac{\di\mu^{L_0}_{,r}}{\di\mu^{L_0}}\bigg)^2\di\mu^{L_0}=:\omega(R_0-z)
		\,,
	\end{split}
\end{equation*}
where the last inequality follows by adapting the argument in \eqref{eq:lim_second_term_1}.  
Therefore we deduce that 
\begin{equation*}
	\fint_{-\pi L_0}^{\pi L_0}(\hat \xi^L(r,\cdot))^2\di \theta
	\le 
	\int_{-\infty}^{R_0}\omega(R_0-z) (R_0-z)\rho_\eps(r-z)\dz\,.
\end{equation*}
Note that  $\omega(t)\to 0$ as $t\to 0$, $\omega(R_0-z)\le \omega(R_0-r_\eps)\lesssim\mathcal{F}_\infty(\mu)\le C$ and $\omega$ non-increasing function.
Let $N\ge 2$ be a natural  number. {Let $a_\eps\ge N\eps$, $a_\eps\to0$.}  Assume $r\in [R_0-a_\eps, R_0)$. Since $\omega$ is non increasing we have
\begin{equation*}
	\begin{split}
		\int_{-\infty}^{R_0}\omega(R_0-z) (R_0-z) \rho_\eps(r-z)\dz &\le	\omega(2a_\eps)\int_{R_0-a_\eps}^{R_0} (R_0-z)\rho_\eps(r-z)\dz\\&+ \omega(R_0-r_\eps)	\int_{-\infty}^{R_0-a_\eps}(R_0-z)\rho_\eps(r-z)\dz\\
		&\lesssim \omega(2a_\eps)\max\{R_0-r,\eps\}+a_\eps e^{-\frac{a_\eps}{\eps}}\\
		&\lesssim\max\{R_0-r,\eps\}(\omega(a_\eps)+\frac{a_\eps}{\eps}e^{-\frac{a_\eps}{\eps}})
		\,.
	\end{split}
\end{equation*}
If instead $r\in (\Rin,R_0-N\eps)$, we get 
\begin{equation*}
	\begin{split}
		\int_{-\infty}^{R_0}\omega(R_0-z) (R_0-z)\rho_\eps(r-z)\dz&\le	
		\omega(R_0-r_\eps)	\int_{-\infty}^{R_0} (R_0-z)\rho_\eps(r-z)\dz
		\lesssim R_0-r.
	\end{split}
\end{equation*}

\noindent
\textit{Step 4:} in this step we show \eqref{eq:est-u^4}.  Since $\xi^L(r,\cdot)$ is $2\pi L_0$-periodic, for fixed $r$, we can find $\theta_0=\theta_0(r)\in[-\pi L_0,\pi L_0]$ such that
\begin{equation*}
	\hat	\xi^L(r,\theta_0)= \fint_{-\pi L_0}^{\pi L_0}\hat \xi^L(r,\hat \theta)\di\hat \theta=0\,,
\end{equation*}
By the fundamental theorem of calculus, Hölder's inequality and Plancherel 
\begin{equation*}\label{eq:abs-u}
	\begin{split}
		|\hat \xi^L(r,\theta)|=\left|\int_{\theta_0}^\theta \partial_\theta\hat \xi^L(r,\hat\theta)\di\theta\right| &\le \sqrt{2\pi L_0}\left(\int_{-\pi L_0}^{\pi L_0}(\partial_\theta\hat \xi^L)^2\di\hat\theta\right)^{\frac12}=2\sqrt2\pi L_0\sqrt{A^L(r)}\,.
	\end{split}
\end{equation*}
This  together with steps 1 and 2 yield
\begin{equation*}	\begin{split}
	\strokedint_{-\pi L}^{\pi L} 
{(\hat \xi^L)^4}\di\theta&=
\strokedint_{-\pi L_0}^{\pi L_0} 
{(\hat \xi^L)^4}\di\theta
\lesssim L_0^2|A^L(r) |\strokedint_{-\pi L_0}^{\pi L_0} 
{(\hat \xi^L)^2} \di \theta\lesssim L_0^2\max\{R_0-r,\eps\}\,.
	\end{split}
\end{equation*}

\noindent
\textit{Step 5:} in this step we show \eqref{eq:uyy+ux}. By \eqref{plancherel-derivatives} the definition of $a^L$ and \eqref{ub:1} 
\begin{equation}\label{est:uyy}
	\begin{split}
			\fint_{-\pi L}^{\pi L}&(\partial_{\theta\theta}\hat \xi^L)^2\di\theta=
		\fint_{-\pi L_0}^{\pi L_0}(\partial_{\theta\theta}\hat \xi^L)^2\di\theta= \sum_{k\in\frac{\mathbb{Z}}{L_0}}k^2a^L(r,k)= \sum_{k\in\frac{\mathbb{Z}}{L_0}}k^2b^{L_0}(\cdot,k)*\rho_\eps(r)\\
		& = \sum_ {k\in \frac{ \mathbb{Z}}{L_0}}\bar b^{L_0}(\lambda,k)k^2\int_{-\infty}^{r_\eps}\rho_\eps(r-z)\dz+\sum_ {k\in \frac{ \mathbb{Z}}{L_0}}\biggl(k^2\int_{r_\eps}^{R_0}\bar b^{L_0}(g_\eps^{-1}(z),k)\rho_\eps(r-z)\dz\biggr)\\
		& \lesssim  e^{\frac{r_\eps-r}{\eps}}{ \lambda}^3 \fint_{\Rin}^{\Rin+\frac{\sqrt\eps}{2}}\sum_{k\in\frac{\mathbb Z}{L_0}}\frac{k^2}{r^3}\bar b^{L_0}(r,k)\di r
		+
		\|\rho_\eps\|_\infty R_0^3
		\int_{{\lambda}}^{R_0} \sum_{k\in\frac{\mathbb{Z}}{L_0}}\frac{k^2}{r^3}
		\overline b^{L_0}(r,k)\di r
		\\
		&\lesssim \left( \frac{e^{\frac{r_\eps-r}{\eps}}}{\sqrt\eps}+\frac1\eps\right)\int_{\Rin}^{{R_0}} \sum_{k\in\frac{\mathbb{Z}}{L_0}}\frac{k^2}{r^3}
		\overline b^{L_0}(r,k)\di r\\
		&
		\lesssim
		\left(\frac1{\sqrt \eps}+
		\frac1\eps \right)\F_\infty(\mu^{L_0})\lesssim \frac{1}\eps \,.
	\end{split}
\end{equation}
Since the function $(z_1,z_2)\mapsto{z_1^2}/{z_2}$ is convex we can apply Jensen's inequality and get
\begin{equation*}\label{jensen}
\frac{(\partial_ra^L(r,k))^2}{a^L(r,k)}= \frac{\big(\partial_rb^{L_0}(\cdot,k)*\rho_\eps(r)\big)^2}{b^{L_0}(\cdot,k)*\rho_\eps(r)}
\le 
\frac{(\partial_rb^{L_0}(\cdot,k))^2}{b^{L_0}(\cdot,k)}*\rho_\eps(r)\,.
\end{equation*}
This together with \eqref{plancherel-derivatives} and \eqref{derivatives-a} imply
\begin{equation}\label{est:ux}
\begin{split}
		\fint_{-\pi L}^{\pi L}(\partial_r\hat\xi^L)^2\di\theta=
	\fint_{-\pi L_0}^{\pi L_0}(\partial_r\hat\xi^L)^2\di\theta&= \sum_{k\in\frac{\mathbb{Z}}{L_0},k\ne0 }\frac{1}{4k^2}\frac{(\partial_r a^L(r,k))^2}{a^L(r,k)}
	\\& \le 
	 \sum_{k\in\frac{\mathbb{Z}}{L_0},k\ne0 }
	 \frac{1}{4k^2}\frac{(\partial_rb^{L_0}(\cdot,k))^2}{b^{L_0}(\cdot,k)}*\rho_\eps(r)\\
	 &= 
	 	 \sum_{k\in\frac{\mathbb{Z}}{L_0},k\ne0 }
	\bigg( \frac{1}{4k^2}\int_{r_\eps}^{R_0}
	 \frac{(\partial_r\bar b^{L_0}(g_\eps^{-1}(z),k))^2}{\bar b^{L_0}(g_\eps^{-1}(z),k)}\rho_\eps(r-z)\dz\bigg)\\
	 &\lesssim \|\rho_\eps\|_\infty \int_{\lambda}^{R_0} \sum_{k\in\frac{\mathbb{Z}}{L_0},k\ne0 }
	 \frac{\dot u^\star(r)}{4k^2}
	 \frac{(\partial_r\bar b^{L_0}(r,k))^2}{\bar b^{L_0}(r,k)}r\di r\\&
	 \lesssim \frac1\eps \F_\infty(\mu^{L_0})\lesssim \frac{1}\eps \,.
\end{split}
\end{equation}
Gathering \eqref{est:uyy} and \eqref{est:ux} we deduce \eqref{eq:uyy+ux}.\\

\noindent
\textit{Step 6:} in this step we show \eqref{eq:uxx-uxy-uxyy}. By \eqref{plancherel-derivatives}, \eqref{derivatives-a}
\begin{equation*}\label{eq:uxy}
	\fint_{-\pi L}^{\pi L}\int_{-1}^1( \partial_{r\theta}\hat\xi^L)^2r\di r\di\theta= 
	\fint_{-\pi L_0}^{\pi L_0}\int_{\Rin}^{\Rout}(\partial_{r\theta}\hat \xi^L)^2r\di r\di\theta= 
	\int_{\Rin}^{\Rout}	\sum_ {k\in \frac{ \mathbb{Z}}{L_0},k\ne 0 } \frac14 \frac{(\partial_ra^L(r,k))^2}{a^L(r,k)}r\di r\,.
\end{equation*}
From \eqref{eq:prop-moll} we deduce
\begin{equation}\label{eq:ax}
	{(\partial_ra^L(r,k))^2}={(b^L(\cdot,k)*\dot\rho_\eps(r))^2}
	\le	 {(b^L(\cdot,k)*|\dot\rho_\eps|(r))^2} = \frac1{\eps^2}{(b^L(\cdot,k)*\rho_\eps(r))^2}
	\le \frac1{\eps^2}(a^L(r,k))^2\,,
\end{equation}
so that recalling \eqref{est:A} we obtain
\begin{equation}\label{eq:uxy-bis}
	\fint_{-\pi L}^{\pi L}\int_{\Rin}^{\Rout}(\partial_{r\theta}\hat \xi^L)^2r\di r\di\theta\le\frac1{4\eps^2} \int_{\Rin}^{\Rout}\sum_ {k\in \frac{ \mathbb{Z}}{L_0} } 	a^L(r,k)r\di r\lesssim \frac{1}{\eps^2}\,.
\end{equation}
In a similar way \eqref{plancherel-derivatives} and \eqref{derivatives-a} give  
\begin{equation}\label{eq:uxx}
	\begin{split}
		\fint_{-\pi L}^{\pi L}
		\int_{\Rin}^{\Rout}(\partial_{rr}\hat \xi^L)^2r \di r\di \theta&=
		\fint_{-\pi L_0}^{\pi L_0}
		\int_{\Rin}^{\Rout}(\partial_{rr}\hat \xi^L)^2r \di r\di\theta\\&
		=
		\int_{\Rin}^{\Rout}\sum_{k\in\frac{\mathbb{Z}}{L_0},k\ne0}\frac1{k^2}\Big[\partial_{rr}\Big(\sqrt{a^L(r,k)}\Big)\Big]^2r\di r
		\\
		&\lesssim \int_{\Rin}^{\Rout}\sum_{k\in\frac{\mathbb{Z}}{L_0},k\ne0}\frac1{k^2}\frac{(\partial_{rr}a^L(r,k))^2}{a^L(r,k)}r\di r+
		\int_{\Rin}^{\Rout}\sum_{k\in\frac{\mathbb{Z}}{L_0},k\ne0}\frac1{k^2}\frac{(\partial_ra^L(r,k))^4}{(a^L(r,k))^3}r\di r\\
		&\lesssim \frac{1}{\eps^2}\int_{\Rin}^{\Rout}\sum_{k\in\frac{\mathbb{Z}}{L_0},k\ne0}\frac1{4k^2}\frac{(\partial_ra^L(r,k))^2}{a^L(r,k)}r\di r\\
		&\lesssim \frac{1}{\eps^2} \int_{\Rin}^{\Rout}\fint_{-\pi L_0}^{\pi L_0}(\partial_r\hat \xi^L)^2r\di\theta\di r
		\lesssim\frac{1}{\eps^2}\mathcal F_\infty (\mu^{L_0})	\lesssim\frac{1}{\eps^2}
		\,,
	\end{split}
\end{equation}
where the second inequality follows from 
\begin{equation*}
	{(\partial_ra^L(r,k))^4}= (\partial_ra^L(r,k))^2{(b^L(\cdot,k)*\dot\rho_\eps(r))^2}
	\le \frac1{\eps^2} \partial_r(a^L(r,k))^2 (a^L(r,k))^2\,,
\end{equation*}
and 
\begin{equation*}
	\begin{split}
		{(\partial_{rr}a^L(r,k))^2}= {(\partial_rb^L(\cdot,k)*\dot\rho_\eps(r))^2}\le 
		\frac1{\eps^2}(\partial_ra^L(r,k))^2\,.
	\end{split}
\end{equation*}
Moreover appealing again to \eqref{eq:ax} we find
\begin{equation}\label{eq:uyyx}
	\begin{split}
		\fint_{-\pi L}^{\pi L}	\int_{\Rin}^{\Rout}(\partial_{\theta\theta r}\hat \xi^L)^2r\di r\di\theta&=
		\fint_{-\pi L_0}^{\pi L_0}	\int_{\Rin}^{\Rout}(\partial_{\theta\theta r}\hat \xi^L)^2r\di r\di\theta\\&
		=
		\int_{\Rin}^{\Rout}\sum_{k\in\frac{\mathbb{Z}}{L_0},k\ne0}{k^2}\Big[\partial_r\Big(\sqrt{a^L(r,k)}\Big)\Big]^2r\di r\\
		&= \int_{\Rin}^{\Rout}\sum_{k\in\frac{\mathbb{Z}}{L_0},k\ne0 }\frac{k^2}4\frac{(\partial_ra^L(r,k))^2}{a^L(r,k)}\di r\\& \lesssim\frac1{\eps^2}\int_{\Rin}^{\Rout}\sum_{k\in\frac{\mathbb{Z}}{L_0},k\ne0 }{k^2}{a^L(r,k)}r\di r\\
		&
		\lesssim \frac1{\eps^2}\int_{\Rin}^{\Rout}\fint_{-\pi L_0}^{\pi L_0}(\partial_{\theta\theta}\hat \xi^L)^2r\di\theta\di r
		\lesssim\frac{1}{\eps^2}\mathcal F_\infty (\mu^{L_0})	\lesssim\frac{1}{\eps^2}\,.
	\end{split}
\end{equation}
Eventually gathering together \eqref{eq:uxy-bis}--\eqref{eq:uyyx} we deduce \eqref{eq:uxx-uxy-uxyy}.\\

\noindent 
\textit{Step 7:} in this step we show \eqref{eq:ux}.  As in step 4 we can find $\theta_0\in[-\pi L_0,\pi L_0]$ such that
\begin{equation*}
	\partial_r\hat	\xi^L(r,\theta_0)= \fint_{-\pi L_0}^{\pi L_0}\partial_r\hat \xi^L(r,\hat \theta)\di\hat \theta=0\,,
\end{equation*}
so that by the fundamental theorem of calculus, Hölder's inequality it holds
\begin{equation}\label{eq:abs-u_x}
	\begin{split}
		|\partial_r\hat \xi^L(r,\theta)|=\left|\int_{\theta_0}^\theta\partial_{r\theta}\hat \xi(r,\theta')\di\theta'\right| &\le \sqrt{2\pi L_0}\left(\int_{-\pi L_0}^{\pi L_0}(\partial_{r\theta}\hat \xi^L)^2r\di\theta'\right)^{\frac12}\\
		&={\sqrt2\pi }L_0\Biggl(	\sum_ {k\in \frac{ \mathbb{Z}}{L},k\ne0}\frac{1}{4}
		\frac{ (\partial_ra^L(r,k))^2}{a^L(r,k)}\Biggr)^{\frac12}\\
		& \lesssim L_0\Bigg(\frac1{4\eps^2} \sum_ {k\in \frac{\mathbb{Z}}{L},k\ne0} a^L(r,k)\Bigg)^{\frac12}\lesssim \frac{L_0}{\eps}
		\,,
	\end{split}
\end{equation}
where we used \eqref{eq:ax} and \eqref{est:A}. 
Therefore \eqref{eq:abs-u_x} implies
\begin{equation*}\label{ub:est4}
	\begin{split}
		\strokedint_{-\pi L_0}^{\pi L_0} \int_{\Rin}^{\Rout}
		{(\partial_r\hat \xi^L)^4} r\di r\di\theta&\lesssim  \frac{ L_0^2}{\eps^2 }
		\strokedint_{-\pi L_0}^{\pi L_0} \int_{\Rin}^{\Rout}
		{(\partial_r\hat \xi^L)^2}r \di r\di\theta
		\lesssim \frac{ L_0^2}{\eps^2}\,.
	\end{split}
\end{equation*}

\noindent
\textit{Step 8:} in this step we show \eqref{eq:weak-conv}. By recalling Definition \ref{def:muL} we have 
\begin{equation*}
\hat\mu^L:=	\mu^L(\hat \xi^L)= \sum_ {k\in \frac{ \mathbb{Z}}{L_0} } a^L(r,k) \mathcal{L}^1\res(\Rin,\Rout)\times\delta_k\,,
\end{equation*}
and 
\begin{equation*}
\hat\mu^L_{,r}=	\mu^L_{,r}(\hat \xi^L)= \sum_ {k\in \frac{ \mathbb{Z}}{L_0} }\partial_ra^L(r,k) \mathcal{L}^1\res(\Rin,\Rout)\times\delta_k\,,
\end{equation*}
Let $\varphi \in C^\infty_c((\Rin,\Rout)\times\R)$. Then 
\begin{equation*}
	\begin{split}
		\int_{(\Rin,\Rout)\times\R}\varphi\di\hat \mu^L&= 
		\int_{(\Rin,\Rout)\times\R}\varphi\di\hat\mu^L- 	\int_{(\Rin,R_0)\times\R}\varphi\di\mu^{L_0}
		+ \int_{(\Rin,R_0)\times\R}\varphi\di\mu^{L_0}\,.
	\end{split}
\end{equation*}
By Lemma \ref{lem:discretisation} we have that 
\begin{equation*}
	\lim_{L\to\infty}\int_{(\Rin,R_0)\times\R}\varphi\di\mu^{L_0}=\int_{(\Rin,R_0)\times\R}\varphi\di\mu\,,
\end{equation*}
hence it suffices to show that 
\begin{equation*}
\lim_{L\to +\infty}	\biggl(\int_{(\Rin,\Rout)\times\R}\varphi\di\hat\mu^L- 	\int_{(\Rin,R_0)\times\R}\varphi\di\mu^{L_0} \biggr)=0\,.
\end{equation*}
Indeed by \eqref{est:A} and \eqref{eq:constraint} we have 
\begin{equation*}
	\begin{split}
		\biggl|\int_{(\Rin,\Rout)\times\R}\varphi\di\hat\mu^L- &	\int_{(\Rin,R_0)\times\R}\varphi\di\mu^{L_0} \biggr|\le
\|\varphi\|_\infty \left|
	\int_{(\Rin,\Rout)\times\R}\di\hat\mu^L- 	\int_{(\Rin,R_0)\times\R}\di\mu^{L_0}\right|\\&=\|\varphi\|_\infty
	\Bigg| \int_{\Rin}^{\Rout} \sum_ {k\in \frac{ \mathbb{Z}}{L_0} } a^L(r,k)\di r- \int_{\Rin}^{R_0} \sum_ {k\in \frac{ \mathbb{Z}}{L_0} }\bar b^L(r,k)\di r\Bigg|\\
	&\lesssim \|\varphi\|_\infty   \left|\int_{\Rin}^{\Rout}\mathcal R_\eps(r)\di r\right|+\sqrt\eps\lesssim \sqrt\eps\to0\quad \text{as }L\to+\infty\,.
	\end{split}
\end{equation*}
Moreover we have 
\begin{equation*}
		\int_{(\Rin,\Rout)\times\R}\varphi\di\hat \mu^L_{,r}= -	\int_{(\Rin,\Rout)\times\R}\partial_r\varphi\di\hat \mu^L\to 
	-	\int_{(\Rin,\Rout)\times\R}\partial_r\varphi\di \mu= 
			\int_{(\Rin,\Rout)\times\R}\varphi\di \mu_{,r}\,.
\end{equation*}

\noindent
\textit{Step 9:} in this step we show \eqref{eq:limsup}.
By  \eqref{plancherel-derivatives}, \eqref{def:b} we have 
\begin{equation}\label{ub:est-lt1}
	\begin{split}
		&	\strokedint_{-\pi L_0}^{\pi L_0} \int_{\Rin}^{\Rout} \frac{(\partial_{\theta\theta}\hat \xi^L)^2}{r^3}
		\di r \di \theta= \int_{\Rin}^{\Rout}\sum_ {k\in \frac{ \mathbb{Z}}{L_0}}\frac{a^L(r,k)k^2}{r^3} \di r=
		\int_{\Rin}^{\Rout}\sum_ {k\in \frac{ \mathbb{Z}}{L_0}}\frac{b^{L_0}(\cdot,k)*\rho_\eps(r)k^2}{r^3} \di r\\
		& =
		\int_{\Rin}^{\Rout}\sum_ {k\in \frac{ \mathbb{Z}}{L_0}}\biggl(\frac{k^2}{r^3}\int_{r_\eps}^{R_0}\bar b^{L_0}(g_\eps^{-1}(z),k)\rho_\eps(r-z)\dz\biggr)\di r+\int_{\Rin}^{\Rout} \sum_ {k\in \frac{\mathbb{Z}}{L_0}}
	\frac{	\bar b^{L_0}(\lambda,k)k^2}{r^3} 
	\int_{-\infty}^{r_\eps}\rho_\eps(r-z)\dz\di r
	\,.
	\end{split}
\end{equation}
By Fubini's theorem, the change of variable $t=\frac{r-z}{\eps}$ and the fact that 
\begin{equation*}
	\frac1{(z+t\eps)^3}\le \frac1{z^3}+C\eps|t|\le \frac{1}{z^3}(1+C\eps|t|)\quad \text{ for }z\in (r_\eps,R_0)\,,
\end{equation*}
 we estimate
\begin{equation}\label{ub:est-lt2}
\begin{split}
	\int_{\Rin}^{\Rout}\sum_ {k\in \frac{ \mathbb{Z}}{L_0}}&\biggl(\frac{k^2}{r^3}\int_{r_\eps}^{R_0}\bar b^{L_0}(g_\eps^{-1}(z),k)\rho_\eps(r-z)\dz\biggr)\di r\\&= 	\int_{r_\eps}^{R_0}	\int_{\Rin}^{\Rout}\sum_ {k\in \frac{ \mathbb{Z}}{L_0}}\biggl(\frac{k^2}{r^3}
	\bar b^{L_0}(g_\eps^{-1}(z),k)\rho_\eps(r-z)\biggr)\di r\dz\\
	&=	\int_{r_\eps}^{R_0}	\int_{\frac{\Rin-z}\eps}^{\frac{\Rout-z}\eps}\sum_ {k\in \frac{ \mathbb{Z}}{L_0}}\biggl(\frac{k^2}{(z+t\eps)^3}
	\bar b^{L_0}(g_\eps^{-1}(z),k)\rho(t)\biggr)\di t\dz\\
	&\le 
		\int_{r_\eps}^{R_0}	\int_{\frac{\Rin-z}\eps}^{\frac{\Rout-z}\eps}\sum_ {k\in \frac{ \mathbb{Z}}{L_0}}\biggl( \frac{k^2}{z^3}
	\bar b^{L_0}(g_\eps^{-1}(z),k)(1+C\eps|t|)\rho(t)\biggr)\di t\dz\\
	&\le m_\eps(1+C\eps)	\int_{ \lambda}^{R_0}\sum_ {k\in \frac{ \mathbb{Z}}{L_0}} \frac{k^2}{z^3}
	\bar b^{L_0}(z,k)\dz
	\,,
\end{split}
\end{equation}
while from \eqref{ub:1} we deduce 
\begin{equation}\label{ub:est-lt3}
	\begin{split}
			\int_{\Rin}^{\Rout}	\sum_ {k\in \frac{ \mathbb{Z}}{L_0}} \frac{\bar b^{L_0}(\lambda,k)k^2}{r^3}
	&	\int^{r_\eps}_{-\infty}\rho_\eps(r-z)\dz\di r\\
		&\lesssim  		\int_{\Rin}^{\Rout}\int_{-\infty}^{r_\eps}\rho_\eps(r-z)\dz\di r
	 \sum_ {k\in \frac{ \mathbb{Z}}{L_0}} \frac{\bar b^{L_0}(\lambda,k)k^2}{\lambda^3}
		\\
		&
		\lesssim
	 \frac\eps2\Big(e^{\frac{r_\eps-\Rin}{\eps}}-
	e^{\frac{r_\eps-\Rout}{\eps}}
	\Big)
		\fint_{\Rin}^{\Rin+\frac{\sqrt\eps}{2}} \sum_{
			k\in \frac{ \mathbb{Z}}{L_0}}\frac{ \overline b^{L_0}(r,k)k^2}{r^3}\di r 
		\,.
	\end{split}
\end{equation}
Analogously, from \eqref{plancherel-derivatives}, \eqref{derivatives-a} and Jensen's inequality {and using that 
\begin{equation*}
(z+t\eps)\dot u^\star(z+t\eps)\le z\dot u^\star(z)(1+C\eps |t|)\quad \text{ for }z\in (r_\eps,R_0)\,,
\end{equation*}
} it holds
\begin{equation}\label{ub:est-lt4}
	\begin{split}
		\strokedint_{-\pi L_0}^{\pi L_0} \int_{\Rin}^{\Rout}  \dot u^\star (\partial_r\hat \xi^L)^2r\di r \di\theta&= \int_{\Rin}^{\Rout}\sum_ {k\in \frac{ \mathbb{Z}}{L_0},k\ne0}\frac{\dot u^\star}{4k^2}
		\frac{(\partial_ra^L(r,k))^2}{ a^L(r,k)} r
		\di r\\
		&\le 
		\int_{\Rin}^{\Rout}\sum_ {k\in \frac{ \mathbb{Z}}{L_0},k\ne0}\frac{\dot u^\star}{4k^2}
		\frac{(\partial_rb^{L_0}(\cdot,k))^2}{ b^{L_0}(\cdot,k)} *\rho_\eps(r)\di r\\
		&\le m_\eps(1+C\eps)\int_{\bar\lambda}^{R_0}
		\sum_ {k\in \frac{ \mathbb{Z}}{L_0},k\ne0}\frac{\dot u^\star}{4k^2}
		\frac{(\partial_r\bar b^{L_0}(r,k))^2}{\bar b^{L_0}(r,k)} \di r\,.
	\end{split}
\end{equation}
Gathering together \eqref{ub:est-lt1}--\eqref{ub:est-lt4} we obtain 
\begin{equation*}\label{eq:step1}	\fint_{-\pi L}^{\pi L}\int_{\Rin}^{\Rout}\Big(\dot u^\star(\partial_r\hat \xi^L)^2+ \frac{(\partial_{\theta\theta}\hat \xi^L)^2}{r^4}
	\Big)r \di r\di\theta
\le m_\eps \biggl(1+C\eps+ \frac{C\eps }{\sqrt\eps}\biggr)\mathcal F_\infty (\mu^{L_0}) \,,
\end{equation*}
and hence by letting $L\to+\infty$ and recalling Lemma \ref{lem:discretisation} \ref{(iii)discr} we deduce \eqref{eq:limsup} and the proof is concluded.

\end{proof}
\noindent \textbf{Construction of the in-plane displacement.}
We conclude by constructing the in-plane displacement and proving Proposition~\ref{prop:upb}. The displacement is chosen so that the first, third, and fourth terms in \eqref{def:F_L} vanish in the limit. More precisely, the radial component $u_r$ is defined as the relaxed solution $u^\star$ plus a suitable correction term, chosen in such a way that the fourth term vanishes identically. The tangential component $u_\theta$ is then constructed so that the third term becomes negligible, while remaining compatible with the smallness of the first term.

We emphasize that, unlike in the toy model, the radial geometry introduces a strong coupling between $u_r$ and $u_\theta$, since the energy depends explicitly on both quantities and not only on their derivatives. As a consequence, the construction of the in-plane displacement is considerably more delicate than in the flat setting considered in \cite{BeMa25}.
\begin{proof}[Proof of Proposition \ref{prop:upb}]
Let $\mu\in\mathcal{M}_\infty$ be as in the statement.  We introduce some parameters depending on $L$, which will be chosen later.
Let $\eps=\eps(L)>0$ and $n=n(L)\in\mathbb{N}$ be such that
\begin{equation}\label{cond_parameters1}
	\lim_{L\to +\infty}\eps(L)=0\,,\quad
	\lim_{L\to +\infty}n(L)=\lim_{L\to +\infty}\frac L{n(L)}=+\infty\,.
\end{equation}
Let $ M=M(L)\in\mathbb{N}$, $M\ge 2$  be such that
\begin{equation}\label{cond_parameters2}
	\lim_{L\to +\infty}M(L)=+\infty\,,\quad \text{and}\quad \lim_{L\to +\infty}M(L){\eps(L)}=0\,,
\end{equation}
Define also $\delta=\delta(L):=\frac{\eps(L)}{M(L)}<\eps(L)$  and $L_0:=L/n(L)$. 
Consider the function  $\hat \xi^L\in \mathcal{A}_L^{\rm out}\cap \mathcal{A}_{L_0}^{\rm out}$     given by Lemma \ref{lem:moll}. Recall that
$$A^L(r):=\fint_{-\pi L}^{\pi L}\frac{(\partial_\theta\hat{\xi}^L(r,\cdot))^2}{2}\di \theta \quad \text{ and }
\quad	f^L(r):=\sqrt{\frac{-u^\star(r)r}{A^L(r)}}\quad \text{for }r\in (\Rin,R_0)\,.$$
 %
We let $\psi_\delta\in C^\infty(\R)$ be a cut-off function with 
\begin{equation}\label{psi-delta}
\psi_\delta\equiv0\ \text{ in } [R_0-\delta,+\infty)\,, \quad \psi_\delta\equiv1\ \text{ in } (-\infty,R_0-2\delta]\,,\quad |\dot\psi_\delta(r)|\le C\delta^{-1}\,,\quad |\ddot\psi_\delta(r)|\le C\delta^{-2}\,.
\end{equation}
Then clearly $\dot\psi_\delta=\ddot\psi_\delta=0$ in $(R_0-2\delta,R_0-\delta)^c$.
We next define  $\big(u^{L}_r,u^{L}_\theta,\xi^L\big)$ as follows:
\begin{equation}\label{def:recovery}
	\begin{split}
	&	{\xi^L(r,\theta):= \psi_\delta(r)f^L(r)\hat \xi^L(r,\theta) \,,}\\[1em]
	& { u_\theta^{L}(r,\theta):= u_\theta^{0,L}(r,\theta)+u_\theta^{1,L}(r)
		\,,}\\[1em]
	& u^{L}_r(r,\theta):=u^\star(r)+\frac{1}{L^2}u_r^{0,L}(r,\theta)
	\,,
	\end{split}
\end{equation}
where 
\begin{equation}\label{def:recovery2}
	\begin{split}
&u_\theta^{0,L}(r,\theta):=-\psi_\delta^2(r)u^\star(r)\theta-
\int_{0}^{\theta }\frac{(\partial_\theta{\xi}^L)^2}{2r}\di \hat\theta\,,\\[1em]
&u_\theta^{1,L}(r):= r\int_r^{R_0}\frac{1}{\hat r^2} \fint_{-\pi L_0}^{\pi L_0} \left(\hat r\partial_{r}u_\theta^{0,L}
-u_\theta^{0,L}+\partial_r\xi^L\partial_\theta\xi^L
\right)\di\hat\theta\di\hat r\,,\\[1em]
& u_r^{0,L}(r,\theta):=-\int_0^\theta(r\partial_r u^L_\theta - u_\theta^L+\partial_r\xi^L\partial_\theta\xi^L)\di\hat\theta\,.
	\end{split}
\end{equation}
{Notice that $u^{1,L}$ satisfies 
\begin{equation}\label{eq:u-theta-1}
	\begin{cases}
	r	\partial_ru_\theta^{1,L}- u_\theta^{1,L}= -\displaystyle\fint_{-\pi L_0}^{\pi L_0} \left( r\partial_{r}u_\theta^{0,L}
	-u_\theta^{0,L}+\partial_r\xi^L\partial_\theta\xi^L
	\right)\di\hat\theta&\text{in }(\Rin,R_0)\\[1em]
	u_\theta^{1,L}\equiv0&\text{in }[R_0,\Rout)
	\end{cases}\,.
\end{equation}
}
For the readers convenience we divide the rest of the proof into a number of steps.\\

\noindent \textit{Step 1:} we show that $(u_r^L,u_\theta^L,\xi^L)\in  \mathcal{A}_L^{\rm in}\times  \mathcal{A}_L^{\rm out} \cap \mathcal{A}_{L_0}^{\rm in}\times  \mathcal{A}_{L_0}^{\rm out}$.

Clearly $\xi^L\in \mathcal{A}_L^{\rm out}\cap \mathcal{A}_{L_0}^{\rm out}$.  To see that $(u_r^L(r,\cdot),u_\theta^L(r,\cdot))$ is $2\pi L_0$ periodic we use the following fact: 

\smallskip
\textit{A differentiable function $h$ is $T$-periodic if $h'$ is $T$-periodic and $h(t)=h(t+T)$ for some $t$. }

\smallskip

\noindent
The function $\partial_\theta u^L_{\theta}(x,\cdot)$ is $2\pi L_0$-periodic, since $\partial_\theta\xi^L(r,\cdot)$ is, and from \eqref{plancherel-derivatives} satisfies
\begin{equation*}
	\begin{split}
		u_\theta^L(r,\pi L_0)-u_\theta^L(r,-\pi L_0)&=-2\pi L_0\psi_\delta^2(r)u^\star(r)-	\int_{\pi L_0}^{\pi L_0 }\frac{(\partial_\theta{\xi}^L(r,\cdot))^2}{2r}\di \theta'\\
	&=-
	2\pi L_0\frac{\psi_\delta^2(r)}r(u^\star(r)r-(f^L(r))^2A^L(r))
	=
	0\,,
	\end{split}
\end{equation*}
from which we deduce $u^L_\theta(r,\cdot)$ is $2\pi L_0$-periodic. Using this periodicity, and in particular also of $\partial_ru^L_{\theta}(r,\cdot)$, we see that $\partial_\theta u^L_{r}(r,\cdot)$ is $2\pi L_0$-periodic. Moreover  from \eqref{eq:u-theta-1} we have
\begin{equation*}
	\begin{split}
	u_r^L(r,\pi L_0)&-u_r^L(r,-\pi L_0)=-\frac{1}{L^2}\int_{-\pi L_0}^{\pi L_0}(r\partial_r u^L_\theta - u_\theta^L+\partial_r\xi^L\partial_\theta\xi^L)\di\hat\theta\\
	&=-\frac{2\pi L_0}{L^2}\left(
	r\partial_ru_\theta^{1,L}-u_\theta^{1,L}
	+\fint_{-\pi L_0}^{\pi L_0}
	(r\partial_r u^{0,L}_\theta - u_\theta^{0,L}+\partial_r\xi^L\partial_\theta\xi^L)\di\hat\theta\
	\right)=0
	\,.
	\end{split}
\end{equation*}
Thus we deduce that $ u^L_{r}$ is $2\pi L_0$-periodic.  For the reader convenience we divide the rest of the proof into several  steps. 
We will repeatedly use that the averaged integral over $(-\pi L,\pi L)$ of a $2\pi L_0$-periodic function is equal to  the averaged integral over $(-\pi L_0,\pi L_0)$ of the same function.\medskip

\noindent
\textit{Step 2:} we show that $(u_r^L,u_\theta^L,\xi^L)$ converges to $\mu$ in the sense of Definition \ref{def:convergence}. 
We have that 
	\begin{equation*}
	\hat	\xi^{L}(r,\theta)= \sum_ {k\in \frac{\mathbb{Z}}{L_0}, k>0 }  a^L_k(r)\sqrt2 \sin(k\theta) + \sum_ {k\in \frac{ \mathbb{Z}}{L_0}, k < 0 }  a^L_k(r)\sqrt2\cos(k\theta)\,,
\end{equation*}
so that 
\begin{equation*}
	\xi^L(r,\theta)= \sum_ {k\in \frac{ \mathbb{Z}}{L_0}, k>0 }\psi_\delta(r)f^L(r)  a^L_k(r)\sqrt2 \sin(k\theta) + \sum_ {k\in \frac{ \mathbb{Z}}{L_0}, k < 0 } \psi_\delta(r)f^L(r) a^L_k(r)\sqrt2\cos(k\theta)\,.
\end{equation*}
Therefore we get
\begin{equation*}
	\begin{split}
		\mu^L:=	\mu^L(\xi^L)&=\sum_ {k\in \frac{\mathbb{Z}}{L_0} }\psi^2_\delta(r)(f^L(r) )^2 (a^L_k(r))^2k^2 \mathcal{L}^1\res(\Rin,\Rout)\times  \delta_k\\
		&= \psi_\delta^2(r)(f^L(r))^2 \Big(\sum_ {k\in \frac{ \mathbb{Z}}{L_0} } (a^L_k(r))^2k^2  \mathcal{L}^1\res(\Rin,R_0-\delta)\times \delta_k\Big) \,.
	\end{split}
\end{equation*}
We show $\mu^L\stackrel{*}{\rightharpoonup}\mu$. We fix $\varphi\in C_c^\infty((\Rin,\Rout)\times\R)$ and we write
\begin{equation*}
	\begin{split}
	\int_{(\Rin,\Rout)\times\R}\varphi\di\mu^L&= 	\int_{(\Rin,R_0-\delta)\times\R}\varphi\di\mu^L\\&=	\left(\int_{(\Rin,R_0-\delta)\times\R}\varphi\di\mu^L- 	\int_{(\Rin,R_0-\delta)\times\R}\varphi\di\mu^L(\hat \xi^L)\right)\\&+\int_{(\Rin,R_0-\delta)\times\R}\varphi\di\mu^L(\hat \xi^L)\,.
	\end{split}
\end{equation*}
By Lemma \ref{lem:moll} and the fact that $\mu([R_0,\Rout)\times\R)=0$ it holds 
\begin{equation*}
	\lim_{L\to+\infty}\int_{(\Rin,R_0-\delta)\times\R}\varphi\di\mu^L(\hat \xi^L)=\int_{(\Rin,R_0)\times\R}\varphi\di\mu\,.
\end{equation*}
Therefore it is sufficient to show that 
\begin{equation*}
	\lim_{L\to+\infty}\left(\int_{(\Rin,R_0-\delta)\times\R}\varphi\di\mu^L- 	\int_{(\Rin,R_0-\delta)\times\R}\varphi\di\mu^L(\hat \xi^L)\right)=0\,.
\end{equation*}
Recalling that $\psi_\delta=0$ in $[R_0-\delta,+\infty)$, $\psi_\delta=1$ in $(-\infty,R_0-2\delta]$ and $M\eps=M^2\delta\ge 2\delta$, we have 
\begin{equation*}
	\begin{split}
	 \bigg|
		\int_{(\Rin,R_0-\delta)\times\R}\varphi\di\mu^L- 	\int_{(\Rin,R_0-\delta)\times\R}\varphi\di\mu^{L}(\hat \xi^L)\bigg|
	&	=
		\bigg| \int_{\Rin}^{R_0-\delta} \sum_ {k\in \frac{\mathbb{Z}}{L_0} }\varphi(r,k) (a^L_k(r))^2k^2 \Big(\psi^2_\delta(r)(f^L(r))^2-1\Big)\di r
		\bigg|\\
		&
		\le  \|\varphi\|_{\infty}
		\bigg|  \int^{R_0-M\eps}_{\Rin}\Big((f^L(r))^2-1\Big) \sum_ {k\in \frac{ \mathbb{Z}}{L_0} } (a^L_k(r))^2k^2\di r \bigg|\\&
	+	\|\varphi\|_{\infty}
		 \bigg|  \int_{R_0-M\eps}^{R_0-\delta} \Big(\psi^2_\delta(r)(f^L(r))^2-1\Big)\sum_ {k\in \frac{ \mathbb{Z}}{L_0} } (a^L_k(r))^2k^2 \di r \bigg|
		 \,.
	\end{split}
\end{equation*}
Since 
$$\sum_ {k\in \frac{ \mathbb{Z}}{L_0} } (a^L_k(r))^2k^2= A^L(r)=\frac12\fint_{-\pi L}^{\pi L}(\partial_\theta\hat{\xi}(r,\cdot))^2\di\theta\,,$$ by \eqref{eq:AL} and \eqref{eq:f(x)} we have {for $r\in (R_0-M\eps, R_0-\delta)$}
\begin{equation*}
 \Big(\psi^2_\delta(r)(f^L(r))^2-1\Big)\sum_ {k\in \frac{ \mathbb{Z}}{L_0} }
  (a^L_k(r))^2k^2\lesssim\Big(\psi^2_\delta(r)(1+o_L(1))-1\Big)M\eps
\end{equation*}
which together with \eqref{cond_parameters2} imply
\begin{equation*}
\|\varphi\|_{\infty}
\bigg|  \int_{R_0-M\eps}^{R_0-\delta} \sum_ {k\in \frac{ \mathbb{Z}}{L_0} } (a^L_k(r))^2k^2(\psi^2_\delta(r)(f^L(r))^2-1)\di r \bigg|\lesssim  M\eps \lesssim M^2\delta\to0\quad \text{ as } L\to+\infty\,.
\end{equation*}
Whereas \eqref{eq:f(x)}  with $N=M$ and the fact that $M=M(L)\to+\infty$ imply 
{for $r\in (\Rin,R_0-M\eps)$ 
	\begin{equation*}
		\Big((f^L(r))^2-1\Big)\sum_ {k\in \frac{ \mathbb{Z}}{L_0} }
		(a^L_k(r))^2k^2\le o_L(1)\,,
\end{equation*}}
so that 
\begin{equation*}
  \|\varphi\|_{\infty}
\bigg|  \int_{\Rin}^{R_0-M\eps} \sum_ {k\in \frac{ \mathbb{Z}}{L_0} } a^L_k(r)k^2(f^L(r)-1)\di r \bigg|\le o_L(1)\to0\quad \text{ as } L\to+\infty\,.
\end{equation*}
Eventually by duality we have 
\begin{equation*}
	\int_{\R\times(\Rin,\Rout)}\varphi\di \mu^L_{,r}= -	\int_{\R\times(\Rin,\Rout)}\varphi_{,r}\di \mu^L\to 
	-\int_{\R\times(\Rin,\Rout)}\varphi_{,r}\di \mu= 
	\int_{\R\times(\Rin,\Rout)}\varphi\di\mu_{,r}\,,
\end{equation*}
which in turn implies  $\mu^L_{,r}\stackrel{*}{\rightharpoonup}\mu_{,r}$.
\medskip

 \noindent
\textit{Step 3:} we show that 
\begin{equation*}\label{step3}
	\limsup_{L\to\infty}\mathcal{F}_L(u_r^L,u_\theta^L,\xi^L)\le \mathcal{F}_\infty(\mu)\,.
\end{equation*}
To prove this, we show that the fifth term in \eqref{def:F_L} is the leading-order contribution and converges to $\mathcal{F}_\infty(\mu)$, while all remaining terms vanish in the limit. Since all variables are $2\pi L_0$ periodic, we will replace the average in $[-\pi L,\pi L]$ with the average in $[-\pi L_0,\pi L_0]$. \medskip

\noindent
\textit{Step 3-(i): Estimate of second and fourth term.}
Since $u^\star$ is negative in $(\Rin,R_0)$ we have $(u^\star)_+=0$ in the same interval. In addition $\xi^L=0$ in $(R_0-\delta, \Rout)$ thus we have
\begin{equation}\label{ub:step1.0}
	\fint_{-\pi L_0}^{\pi L_0}\int_{\Rin}^{\Rout}	L^2\frac{(u^\star)_+}{r}\frac{(\partial_\theta\xi^L)^2}{r^2}r\di r\di\theta=0 \,.
\end{equation}
Furthermore noticing that
\begin{equation*}
	\partial_\theta u_r^L=-\frac1{L^2}(r\partial_ru_\theta^L-u_\theta^L+\partial_r\xi^L\partial_\theta\xi^L)\,,
\end{equation*}
we also get
\begin{equation}\label{ub:step1.1}
	\strokedint_{-\pi L_0}^{\pi L_0} \int_{\Rin}^{\Rout}  \Big( L^2 \frac{	\partial_\theta u_r^L}{2r}
	 +\frac{\partial_ru_\theta^L}2-\frac{u_\theta^L}{2r}+\frac{\partial_r\xi^L\partial_\theta\xi^L}{2r}
	  \Big)^2 r \di r\di\theta=0\,.
\end{equation}
\medskip

\noindent
\textit{Step 3-(ii): Estimate of the fifth/leading term.} We show that 
\begin{equation}\label{ub:step2}
	\begin{split}
		\limsup_{L\to+\infty}\pi
	\fint_{-\pi L_0}^{\pi L_0}\int_{\Rin}^{\Rout}\Big(\dot u^\star(\partial_r\xi^L)^2 & + \frac{(\partial_{\theta\theta}\xi^L)^2}{r^4} \Big)r\di r\di\theta\\
&	\le \mathcal{F}_\infty(\mu)+C\lim_{L\to +\infty} \omega(2M^2\delta)\log M\,.
	\end{split}
\end{equation}
We start by computing
\begin{equation}\label{der:ux}
\partial_r\xi^L(r,\theta)= \psi_\delta(r)f^L(r)\partial_r\hat \xi^L(r,\theta)+ \dot\psi_\delta(r)f^L(r)\hat \xi^L(r,\theta)+
	\psi_\delta(r)\dot f^L(r)\hat \xi^L(r,\theta)\,.
\end{equation}
Therefore by Young's inequality
\begin{equation}\label{ub:step2-0}
\begin{split}
		\fint_{-\pi L_0}^{\pi L_0} \int_{\Rin}^{\Rout}\dot u^\star(\partial_r\xi^L)^2 r  \di r\di\theta
	&	\le (1+\alpha) 
			\fint_{-\pi L_0}^{\pi L_0}\int_{\Rin}^{R_0-\delta}\dot u^\star(f^L)^2(\partial_r\hat \xi^L)^2  r  \di r\di\theta\\
			&+ 2(1+\alpha^{-1}) \frac1{\delta^2}	\fint_{-\pi L_0}^{\pi L_0}\int_{R_0-2\delta}^{R_0-\delta}\dot u^\star(f^L)^2(\hat \xi^L)^2  r  \di r\di\theta\\
				&+ 2(1+\alpha^{-1})	\fint_{-\pi L_0}^{\pi L_0}\int_{\Rin}^{R_0-\delta}\dot u^\star(\dot f^L)^2(\hat \xi^L)^2  r  \di r\di\theta\,,
\end{split}
\end{equation}
for any $\alpha>0$. We now estimate each term on the right hand side of \eqref{ub:step2-0}.
In particular,
 \eqref{eq:f(x)},  \eqref{eq:est-u} and \eqref{eq:f'(x)-f''(x)}  yield
\begin{equation}\label{ub:step2-1}
\fint_{-\pi L_0}^{\pi L_0}\int_{\Rin}^{R_0-\delta}\dot u^\star(f^L)^2(\partial_r\hat \xi^L)^2  r  \di r\di\theta
\le 
	 (1+o_L(1))	\fint_{-\pi L_0}^{\pi L_0}\int_{\Rin}^{R_0-\delta}\dot u^\star(\partial_r\hat \xi^L)^2  r  \di r\di\theta
	 \,,
\end{equation}
\begin{equation}\label{ub:step2-2}
	\begin{split}
	\frac1{\delta^2}	\fint_{-\pi L_0}^{\pi L_0}\int_{R_0-2\delta}^{R_0-\delta}\dot u^\star(f^L)^2(\hat \xi^L)^2  r\di r\di\theta  &\lesssim \frac1{\delta^2}{\delta}(\omega(2M\eps)+Me^{-M})\int_{R_0-2\delta}^{R_0-\delta}r\di r
\\&\lesssim (\omega(2M\eps)+Me^{-M})\,,
	\end{split}
\end{equation}
and
\begin{equation}\label{ub:step2-3}
	\begin{split}
\fint_{-\pi L_0}^{\pi L_0}\int_{\Rin}^{R_0-\delta}\dot u^\star(\dot f^L)^2(\hat \xi^L)^2 r \di r\di\theta&\lesssim  
\int^{R_0-\delta}_{R_0-\eps} \frac1{(R_0-r)\eps}\eps (\omega(2M\eps)+Me^{-M})\di r \\
&+ \int^{R_0-\eps}_{R_0-\sqrt{\eps}} \frac{e^{\frac {r-R_0}\eps}}{(R_0-r)\eps}(R_0-r)(\omega(2\sqrt\eps)+{(\sqrt\eps)^{-1}}e^{-\frac1{\sqrt\eps}})\di r\\&
+ \int_{(\Rin, M_0-\sqrt\eps)\cup(M_0+\sqrt{\eps},R_0-\sqrt\eps)}\frac{e^{-\frac1{\sqrt\eps}}}{(R_0-r)\eps}(R_0-r)\di r\\
&+\int_{M_0-\sqrt\eps}^{M_0+\sqrt\eps}
\frac{\eps}{(R_0-r)\eps}(R_0-r)\di r
\\
& \lesssim \log\frac{\eps}{\delta}(\omega(2M\eps)+Me^{-M})+o_L(1)\,. 
	\end{split}
\end{equation} 
Gathering together \eqref{ub:step2-0}--\eqref{ub:step2-3} and recalling that $\delta=\eps/M$ we infer 
\begin{equation}\label{lead-1}
	\begin{split}
		\fint_{-\pi L_0}^{\pi L_0}\int_{\Rin}^{\Rout}\dot u^\star(\partial_r\xi^L)^2 r  \di r\di\theta &\le (1+\bar \alpha) 	\fint_{-\pi L_0}^{\pi L_0}\int_{\Rin}^{\Rout}\dot u^\star(\partial_r\hat\xi^L)^2 r  \di r\di\theta\\& + 
	C\log M(\omega(2M^2\delta)+Me^{-M})+o_L(1)\,,
	\end{split}
\end{equation}
with $\bar{\alpha}:=\alpha+\alpha o_L(1)+o_L(1)$.
For the second term, since
 $\partial_{\theta\theta}\xi^L=\psi_\delta(r)f^L(r)\partial_{\theta\theta}\hat \xi^L$,  from \eqref{eq:f(x)} it follows
\begin{equation}
	\label{lead-2}\begin{split}
	\fint_{-\pi L_0}^{\pi L_0}\int_{\Rin}^{\Rout}
	\frac{(\partial_{\theta\theta}\xi^L)^2}{r^3}
	 \di r\di\theta&\le (1+o_L(1))	\fint_{-\pi L_0}^{\pi L_0}\int_{\Rin}^{R_0-\delta}	\frac{(\partial_{\theta\theta}\hat\xi^L)^2}{r^3}  \di r\di\theta\,.
	\end{split}
\end{equation}
By \eqref{lead-1}, \eqref{lead-2}, \eqref{eq:limsup} and the fact that $M\to+\infty$ 
we finally deduce
\begin{equation*}
	\limsup_{L\to+\infty}\pi
\fint_{-\pi L_0}^{\pi L_0}\int_{\Rin}^{\Rout}\Big(\dot u^\star(\partial_r\xi^L)^2  + \frac{(\partial_{\theta\theta}\xi^L)^2}{r^4} \Big)r\di r\di\theta
\le (1+\alpha)\mathcal{F}_\infty(\mu)+ C\lim_{L\to +\infty}\omega(2M^2\delta)\log M \,.
\end{equation*}
Eventually by the arbitrariness of $\alpha$ we infer the desired estimate.
\medskip

\noindent
\textit{Step 3-(iii): Estimate of the first term.} We show that
\begin{equation}\label{ub:step3}
L^2\strokedint_{-\pi L_0}^{\pi L_0} \int_{\Rin}^{\Rout}
\biggl( \partial_ru_r^L+\frac{(\partial_r\xi^L)^2}{2L^2}-\dot u^\star \biggr)^2 r\di r\di\theta \lesssim 
\frac {L_0^4}{L^2} \frac1{\delta^2\eps}\lesssim \frac {L_0^4}{L^2}\frac{1}{\delta^3M}\,.
\end{equation}
By Young's inequality we split in two terms
\begin{equation}\label{ub:est1}
	\begin{split}
L^2\strokedint_{-\pi L_0}^{\pi L_0} \int_{\Rin}^{\Rout}
&\biggl( \partial_ru_r^L+\frac{(\partial_r\xi^L)^2}{2L^2}-\dot u^\star \biggr)^2 r\di r\di\theta \\
&\lesssim  \fint_{-\pi L_0}^{\pi L_0} \int_{\Rin}^{\Rout}
\frac{(\partial_r\xi^L)^4}{L^2} r\di r\di\theta
+L^2
\strokedint_{-\pi L_0}^{\pi L_0}\int_{\Rin}^{\Rout}
\bigl( \partial_ru_r^L-\dot u^\star \bigr)^2r \di r\di\theta\,.
	\end{split}
\end{equation}
We start from the first term on the right hand-side of \eqref{ub:est1}. 
By \eqref{der:ux} we have
\begin{equation}\label{eq:est-term1}
	\begin{split}
	\fint_{-\pi L_0}^{\pi L_0} \int_{\Rin}^{\Rout}&
	{(\partial_r \xi^L)^4} r\di r\di\theta \lesssim 	\fint_{-\pi L_0}^{\pi L_0} \int_{\Rin}^{R_0-\delta}(f^L)^4
	{(\partial_r\hat \xi^L)^4} r\di r\di\theta\\&+ 
	\frac{1}{\delta^4}
	\fint_{-\pi L_0}^{\pi L_0} \int_{R_0-2\delta}^{R_0-\delta}(f^L)^4
	{(\hat \xi^L)^4}r \di r\di\theta
	+ 	\fint_{-\pi L_0}^{\pi L_0} \int_{\Rin}^{R_0-\delta}(\dot f^L)^4	{(\hat \xi^L)^4}r \di r\di\theta\,.
	\end{split}
\end{equation}
Therefore \eqref{eq:f(x)} and \eqref{eq:ux} give
\begin{equation}
	\fint_{-\pi L_0}^{\pi L_0} \int_{\Rin}^{R_0-\delta}(f^L)^4
{(\partial_r\hat \xi^L)^4} r\di r\di\theta\lesssim \frac{L_0^2}{\eps^2}\,,
\end{equation}
whereas from \eqref{eq:f(x)}, \eqref{eq:est-u^4}, and the fact that $r\in (R_0-2\delta,R_0-\delta)$ we get
\begin{equation}
	\begin{split}
	\frac{1}{\delta^4}
\fint_{-\pi L_0}^{\pi L_0} \int_{R_0-2\delta}^{R_0-\delta}(f^L)^4
{(\hat \xi^L)^4}r \di r\di\theta
& \lesssim \frac{1}{\delta^4} \frac{\delta^2}{\eps^2} L_0^2\int_{R_0-2\delta}^{R_0-\delta}(\max\{(R_0-r),\eps\})^2
\dx
\lesssim \frac{L_0^2}{\delta}\,.
	\end{split}
\end{equation}
Finally by  \eqref{eq:f'(x)-f''(x)}   and \eqref{eq:est-u^4} 
\begin{equation}\label{eq:est-term1-3}
\begin{split}
	\fint_{-\pi L_0}^{\pi L_0} \int_{\Rin}^{R_0-\delta}(\dot f^L)^4	{(\hat \xi^L)^4}r \di r\di\theta &\lesssim
	L_0^2 \int^{R_0-\delta}_{\Rin} \frac{(\max\{(R_0-r),\eps\})^2}{(R_0-r)^2\eps^2}
	\di r\lesssim L_0^2\left(\frac1\delta+\frac1{\eps^2}\right)
	\,.
\end{split}
\end{equation}
Thus gathering together \eqref{eq:est-term1}--\eqref{eq:est-term1-3} we infer
\begin{equation}\label{eq:est-ux^4}
	\fint_{-\pi L_0}^{\pi L_0} \int_{\Rin}^{\Rout}
\frac{(\partial_r \xi^L)^4}{L^2} r\di  r\di\theta
\lesssim \frac{L_0^2}{L^2}\left(\frac1\delta+\frac1{\eps^2}\right)\,.
\end{equation}
 We now pass to estimate the second term on the right hand side of \eqref{ub:est1}. 
 By definition we have 
 \begin{equation}
\partial_ru^L_r-\dot u^\star=\frac1{L^2} \partial_ru_r^{0,L}
=- \frac1{L^2} \int_0^\theta \partial_r(r\partial_ru_\theta^L-u_\theta^L+\partial_r\xi^L\partial_\theta\xi^L)\di\hat\theta\,.
 \end{equation}
From \eqref{def:recovery}, \eqref{def:recovery2}, \eqref{eq:u-theta-1} and integrating by parts we find
\begin{equation}\label{ub:est5}
	\begin{split}
	r\partial_ru_\theta^L-u_\theta^L+\partial_r\xi^L\partial_\theta\xi^L&=
	r\partial_ru_\theta^{0,L}-u_\theta^{0,L}+\partial_r\xi^L\partial_\theta\xi^L-\fint_{-\pi L_0}^{\pi L_0}
	\left(
	r\partial_ru_\theta^{0,L}-u_\theta^{0,L}+\partial_r\xi^L\partial_\theta\xi^L
	\right)\di\theta
	\\&= \int_0^\theta \left(\frac{(\partial_\theta\xi^L)^2}{r}-\partial_\theta\xi^L\partial_{\theta r}\xi^L\right)\di\hat\theta+ \partial_r\xi^L\partial_\theta\xi^L+ B^L(r)\theta+C^L(r)\\&=
	 \int_0^\theta \left(\frac{(\partial_\theta\xi^L)^2}{r}+\partial_{\theta\theta}\xi^L\partial_{r}\xi^L\right)\di\hat\theta+ \partial_r\xi^L\partial_\theta\xi^L|_{\theta=0}+B^L(r)\theta+C^L(r)\\
	 &=  \int_0^\theta \left(\frac{(\partial_\theta\xi^L)^2}{r}+\partial_{\theta\theta}\xi^L\partial_{r}\xi^L\right)\di\hat\theta+B^L(r)\theta+D^L(r)
\,,
		\end{split}
\end{equation}
where 
\begin{equation*}
	B^L(r):= -r(\psi_\delta^2(r)u^\star(r))'+\psi_\delta^2(r)u^\star(r)\,,
\end{equation*}
\begin{equation*}
	C^L(r):=-\fint_{-\pi L_0}^{\pi L_0}
	\left(
	r\partial_ru_\theta^{0,L}-u_\theta^{0,L}+\partial_r\xi^L\partial_\theta\xi^L
	\right)\di\theta\,,\
\end{equation*}
and 
\begin{equation*}
D^L(r):=	C^L(r)+ \partial_r\xi^L\partial_\theta\xi^L|_{\theta=0}= \fint_{-\pi L_0}^{\pi L_0} \int_0^\theta \left(\frac{(\partial_\theta\xi^L)^2}{r}+\partial_{\theta\theta}\xi^L\partial_{r}\xi^L\right)\di\hat\theta\di\theta\,.
\end{equation*}

Using \eqref{ub:est5} we get
\begin{equation*}
\begin{split}
\partial_ru^L_r-\dot u^\star& 
= - \frac1{L^2} \left[  \int_0^\theta  \int_0^{\hat\theta} \partial_r\left(\frac{(\partial_\theta\xi^L)^2}{r}+\partial_{\theta\theta}\xi^L\partial_{r}\xi^L\right)\di\tilde\theta \di\hat\theta+\dot B^L(r)\frac{\theta^2}2+\dot D^L(r)\theta\right]
\\&
= - \frac1{L^2} \biggl[  \int_0^\theta  \int_0^{\hat\theta} \left(-\frac{(\partial_\theta\xi^L)^2}{r^2}
+\frac{2\partial_\theta\xi^L\partial_{\theta r}\xi^L}{r}
+\partial_{\theta\theta r}\xi^L\partial_{r}\xi^L
+\partial_{\theta\theta}\xi^L\partial_{rr}\xi^L
\right)\di\tilde\theta \di\hat\theta\\&\qquad\qquad\qquad \qquad\qquad\qquad \qquad\qquad \qquad\qquad\qquad+\dot B^L(r)\frac{\theta^2}2+\dot D^L(r)\theta\biggr]
\,.
\end{split}
\end{equation*}
This together with Young's inequality give the estimate
\begin{equation}\label{ub:est6}
	\begin{split}
 L^2
\strokedint_{-\pi L_0}^{\pi L_0}\int_{\Rin}^{\Rout}&
\bigl( \partial_ru_r^L-\dot u^\star \bigr)^2 r \di r\di\theta
\\ &
\lesssim
\frac 1{L^2}\strokedint_{-\pi L_0}^{\pi L_0}\int_{\Rin}^{R_0-\delta}\left[
\int_0^\theta  \int_0^{\hat\theta}  \left(\frac{-(\partial_\theta\xi^L)^2}{r^2}+ 
\frac{2\partial_\theta\xi^L\partial_{\theta r}\xi^L}{r}
\right) \di\tilde\theta \di\hat\theta\right]^2
r \di r\di\theta\\
& +\frac 1{L^2}\strokedint_{-\pi L_0}^{\pi L_0}\int_{\Rin}^{R_0-\delta}\left[
\int_0^\theta  \int_0^{\hat\theta}  \left(\partial_{\theta\theta r}\xi^L\partial_{r}\xi^L+
\partial_{\theta\theta}\xi^L\partial_{rr}\xi^L
\right) \di\tilde\theta \di\hat\theta\right]^2
r \di r\di\theta
\\& +
\frac 1{L^2}\strokedint_{-\pi L_0}^{\pi L_0}\int_{\Rin}^{R_0-\delta} (\dot B^L(r))^2\theta^4
r \di r\di\theta
\\&+
\frac 1{L^2}\strokedint_{-\pi L_0}^{\pi L_0}\int_{\Rin}^{R_0-\delta} (\dot D^L(r))^2\theta^2
r \di r\di\theta
	\,.
	\end{split}
\end{equation}
We now estimate each term on the right-hand side of \eqref{ub:est6}. First, observe that $D^L(r)$ is defined as the $\theta$-average of the quantity $ \int_0^\theta \left( \frac{(\partial_\theta\xi^L)^2}{r} + \partial_{\theta\theta}\xi^L\,\partial_r\xi^L \right) \di\hat\theta$ .
Hence, by Young's inequality, the last term in \eqref{ub:est6} can be controlled by the sum of two terms of the same form as the first and second terms on the right-hand side of \eqref{ub:est6}, with the integral over $(0,\theta)$ replaced by the corresponding average and multiplied by a factor $L_0^2$. Consequently, it has the same asymptotic behaviour of the first and second term.
For the third term it holds 
\begin{equation*}\begin{split}
\dot B^L(r)
&=  -r \big(2(\dot\psi_\delta(r))^2u^\star(r)+ 2\psi_\delta(r)\ddot\psi_\delta(r)u^\star(r)+ 
4\psi_\delta(r)\dot\psi_\delta(r)\dot u^\star(r)+ \psi_\delta^2(r)\ddot u^\star(r)
\big)\,.
	\end{split}
\end{equation*}
This together with \eqref{psi-delta} and the identity $u^\star(r)=C\log\frac r{R_0}$ in $(\Rin,R_0)$ imply
\begin{equation}\label{ub:est7}
	\begin{split}
\frac 1{L^2}\strokedint_{-\pi L_0}^{\pi L_0}\int_{\Rin}^{\Rout} (\dot B^L(r))^2\theta^4
r \di r\di\theta
& \lesssim 
\frac{L_0^4}{L^2}	\int_{R_0-2\delta}^{R_0-\delta}((\dot B^L(r))^2r
\di r\\&
+ 
\frac{L_0^4}{L^2}	\int_{\Rin}^{R_0-2\delta}((\psi_\delta^2(r)\ddot u^\star(r))^2r
\di r\\
&\lesssim \frac{L_0^4}{L^2} \frac1{\delta}+ \frac{L_0^4}{L^2}\lesssim \frac{L_0^4}{L^2} \frac1{\delta}
 \,.
	\end{split}
\end{equation}
We now estimate the first and the second term on the right hand-side of \eqref{ub:est6}. 
We first observe that if $a,b,c,d$ are $2\pi L_0$-periodic then by applying in order   H\"older, Young and Jensen inequalities we have
\begin{equation}\label{stima-algebrica}
	\begin{split}
		\bigg[
		\int_0^\theta\int_0^{\hat\theta}&(ab+cd)\di\tilde \theta\di\hat\theta
		\bigg]^2
		\le \left[
		\int_0^\theta  \|a\|_{L^2(0,\hat\theta)}\|b\|_{L^2(0,\hat\theta)}
		+ \|c\|_{L^2(0,\hat\theta)}\|d\|_{L^2(0,\hat\theta)}
		\di\hat\theta
		\right]^2\\
		& \lesssim 
		\left[
		\int_0^\theta  \|a\|_{L^2(0,\hat\theta)}\|b\|_{L^2(0,\hat\theta)}
		\di\hat\theta
		\right]^2
		+ \left[
		\int_0^\theta
		 \|c\|_{L^2(0,\hat\theta)}\|d\|_{L^2(0,\hat\theta)}
		\di\hat\theta
		\right]^2\\
		& \lesssim L_0^2 \fint_{-\pi L_0}^{\pi L_0} \|a\|^2_{L^2(0,\hat\theta)}\|b\|^2_{L^2(0,\hat\theta)}
		\di\hat\theta+ 
		L_0^2 \fint_{-\pi L_0}^{\pi L_0} \|c\|^2_{L^2(0,\hat\theta)}\|d\|^2_{L^2(0,\hat\theta)}
		\di\hat\theta\\
	&\lesssim
	L_0^4\left[
	\bigg(\fint_{-\pi L_0}^{\pi L_0}a^2\di\theta\bigg)
	\bigg(\fint_{-L_0}^{L_0}b^2\di\theta\bigg)+
	\bigg(\fint_{-\pi L_0}^{\pi L_0}c^2\di \theta\bigg)
	\bigg(\fint_{-\pi L_0}^{\pi L_0}d^2\di\theta\bigg)
	\right]
		\,.	\end{split}
\end{equation}
Therefore it follows that 
\begin{equation}\label{ub:est00}
	\begin{split}
		\frac 1{L^2}\strokedint_{-\pi L_0}^{\pi L_0}&\int_{\Rin}^{R_0-\delta}\left[
		\int_0^\theta  \int_0^{\hat\theta}  \left(
	\frac{2	\partial_{\theta }\xi^L\partial_{\theta r}\xi^L}r
		-
		\frac{(\partial_{\theta}\xi^L)^2}{r^2}
		\right) 
		\di\tilde\theta \di\hat\theta\right]^2
		r \di r\di\theta
		\\
		&\lesssim  \frac {L_0^4}{L^2}
		\int_{\Rin}^{R_0-\delta}\bigg(\fint_{-\pi L_0}^{\pi L_0}(\partial_{\theta}\xi^L )^2\di\theta\bigg)\bigg(\fint_{-\pi L_0}^{\pi L_0}(\partial_{\theta r}\xi^L )^2\di\theta\bigg)
		\di r \\
		&+  \frac {L_0^4}{L^2}
		\int_{\Rin}^{R_0-\delta}
		\bigg(\fint_{-\pi L_0}^{\pi L_0}(\partial_{\theta }\xi^L )^2\di\theta\bigg)\bigg(\fint_{-\pi L_0}^{\pi L_0}(\partial_\theta\xi^L )^2\di\theta\bigg)
			\di r
		\,.
	\end{split}
\end{equation}and
\begin{equation}\label{ub:est0}
	\begin{split}
		\frac 1{L^2}\strokedint_{-\pi L_0}^{\pi L_0}&\int_{\Rin}^{R_0-\delta}\left[
		\int_0^\theta  \int_0^{\hat\theta}  \left(\partial_{\theta\theta }\xi^L\partial_{rr}\xi^L+\partial_{\theta\theta r}\xi^L\partial_{r}\xi^L
		\right) 
		\di\tilde\theta \di\hat\theta\right]^2
		r \di r\di\theta
	\\
		&\lesssim  \frac {L_0^4}{L^2}
\int_{\Rin}^{R_0-\delta}\bigg(\fint_{-\pi L_0}^{\pi L_0}(\partial_{\theta\theta}\xi^L )^2\di\theta\bigg)\bigg(\fint_{-\pi L_0}^{\pi L_0}(\partial_{rr}\xi^L )^2\di\theta\bigg)
		\di r \\
		&+  \frac {L_0^4}{L^2}
	\int_{\Rin}^{R_0-\delta}
	\bigg(\fint_{-\pi L_0}^{\pi L_0}(\partial_{\theta\theta r}\xi^L )^2\di\theta\bigg)\bigg(\fint_{-\pi L_0}^{\pi L_0}(\partial_r\xi^L )^2\di\theta\bigg)
		\di r
		 \,.
	\end{split}
\end{equation}
We  show separately the following estimates for the partial derivatives of  $\xi^L$:
\begin{equation}\label{uy+uyy+ux}
		\fint_{-\pi L_0}^{\pi L_0}(\partial_{\theta}\xi^L)^2\di\theta\lesssim 1\,,\quad
	\fint_{-\pi L_0}^{\pi L_0}(\partial_{\theta\theta}\xi^L)^2\di\theta\lesssim {\frac1\eps}\,,\quad 
	\fint_{-\pi L_0}^{\pi L_0}(\partial_r\xi^L)^2\di\theta\lesssim\frac{\max\{(R_0-r),\eps\}}{(R_0-r)\eps}\,,
\end{equation}
\begin{equation}\label{uxx+uyx}
\int_{\Rin}^{R_0-\delta}\fint_{-\pi L_0}^{\pi L_0}(\partial_{rr}\xi^L)^2\di\theta r\di r \lesssim { \frac1{\delta^2}}\,,
\quad
\int_{\Rin}^{R_0-\delta}\fint_{-\pi L_0}^{\pi L_0}(\partial_{\theta r}\xi^L)^2\di\theta r\di r\lesssim\frac1{\delta\eps} \,,
\end{equation}
\begin{equation}\label{uyyx}
\int_{\Rin}^{R_0-\delta}\fint_{-\pi L_0}^{\pi L_0}(\partial_{\theta\theta r}\xi^L)^2\di\theta r\di r\lesssim\frac1{\delta\eps}\,.
\end{equation}
 %
Since $\partial_{\theta}\xi^L=\psi_\delta(r)f^L(r)\partial_{\theta}\hat \xi^L$ and $\partial_{\theta\theta}\xi^L=\psi_\delta(r)f^L(r)\partial_{\theta\theta}\hat \xi^L$   by \eqref{eq:f(x)}, \eqref{eq:AL} and \eqref{eq:uyy+ux} we have
 \begin{equation*}
 	\label{ub:est-y-yy}
 		\fint_{-\pi L_0}^{\pi L_0}(\partial_{\theta}\xi^L)^2\di\theta
 	\lesssim 1\,,\quad 
 		\fint_{-\pi L_0}^{\pi L_0}(\partial_{\theta\theta}\xi^L)^2\di\theta
\lesssim \frac{1}{\eps}\,.
 \end{equation*}
By \eqref{der:ux}, \eqref{eq:f(x)}, \eqref{eq:uyy+ux}, \eqref{eq:est-u} and \eqref{eq:f'(x)-f''(x)} we have
\begin{equation*}\label{ub:est-x}
	\begin{split}
	\fint_{-\pi L_0}^{\pi L_0}(\partial_r\xi^L)^2\di\theta
	&	\lesssim  \fint_{-\pi L_0}^{\pi L_0}(f^L(r))^2(\partial_r\hat \xi^L)^2\di \theta
	\\&	+ \frac1{\delta^2}\fint_{-\pi L_0}^{\pi L_0}\chi_{(R_0-2\delta,R_0-\delta)}(f^L(r))^2(\hat \xi^L)^2\di\theta
		+ \fint_{-\pi L_0}^{\pi L_0}(\dot f^L(r))^2(\hat \xi^L)^2\di\theta\\
		&\lesssim \frac1\eps+ \frac1{\delta^2}\frac\delta\eps\max\{(R_0-r),\eps\}\chi_{(R_0-2\delta,R_0-\delta)}+ \frac{\max\{(R_0-r),\eps\}}{(R_0-r)\eps}\\
	&	\lesssim \frac1\eps +\frac{1}{\delta}\chi_{(R_0-2\delta,R_0-\delta)}+ \frac{\max\{(R_0-r),\eps\}}{(R_0-r)\eps}\lesssim \frac{\max\{(R_0-r),\eps\}}{(R_0-r)\eps}
		\,.
	\end{split}
\end{equation*}
From  \eqref{der:ux} we calculate the second derivative in $r$
\begin{equation*}
	\begin{split}
		\partial_{rr}\xi^L&=\psi_\delta(r)f^L(r)\partial_{rr}\hat{\xi}^L+\psi_\delta(r)\ddot f^L(r)\hat{\xi}^L+	\ddot\psi_\delta(r)f^L(r)\hat{\xi}^L \\
		&+ 2\dot\psi_\delta(r)f^L(r)\partial_r\hat{\xi}^L+ 2\dot\psi_\delta(r)\dot f^L(r)\hat{\xi}^L+ 2\psi_\delta(r)\dot f^L(r)\partial_r\hat{\xi}^L\,.
	\end{split}
\end{equation*}
Hence by \eqref{eq:f(x)}, \eqref{eq:uxx-uxy-uxyy}, \eqref{eq:f'(x)-f''(x)}, \eqref{eq:est-u} and \eqref{eq:limsup} we have
\begin{equation}\label{ub:est-xx}
	\begin{split}
\int_{\Rin}^{R_0-\delta}\fint_{-\pi L_0}^{\pi L_0}&(\partial_{rr}\xi^L)^2\di\theta r\di r
		\lesssim \int_{\Rin}^{R_0-\delta}\fint_{-\pi L_0}^{\pi L_0}(f^L(x))^2(\partial_{rr}\hat \xi^L)^2\di\theta r\di r+ \int_{\Rin}^{R_0-\delta}\fint_{-\pi L_0}^{\pi L_0}(\ddot f^L(r))^2(\hat \xi^L)^2\di\theta r\di r\\
		& 	+ \int_{\Rin}^{R_0-\delta}\fint_{-\pi L_0}^{\pi L_0}(\dot f^L(r))^2(\partial_r\hat \xi^L)^2\di\theta r\di r+
	 \frac{1}{\delta^4}
	 \int_{R_0-2\delta}^{R_0-\delta}\fint_{-\pi L_0}^{\pi L_0}(f^L(r))^2(\hat \xi^L)^2\di\theta r\di r\\
		& 
		+ \frac{1}{\delta^2}
\int_{R_0-2\delta}^{R_0-\delta}\fint_{-\pi L_0}^{\pi L_0}
		(\dot f^L(r))^2(\hat \xi^L)^2\di\theta r\di r
		+\frac{1}{\delta^2}
	\int_{R_0-2\delta}^{R_0-\delta}\fint_{-\pi L_0}^{\pi L_0}
		(f^L(r))^2(\partial_r\hat \xi^L)^2\di \theta r\di r\\
		& \lesssim \frac1{\eps^2} +  \int_{\Rin}^{R_0-\delta}\frac{\max\{(R_0-r),\eps\}}{(R_0-r)^3\eps} \di r+ \frac{1}{\delta\eps}
+ \frac1{\delta^4}\int_{R_0-2\delta}^{R_0-\delta}\frac{R_0-r}\eps\max\{(R_0-r),\eps\}\di r\\
&
+ \frac{\max\{(R_0-r),\eps\}}{\delta^2}\int_{R_0-2\delta}^{R_0-\delta}\frac1{(R_0-r)\eps}\di r+\frac1{\delta^2}\\&\lesssim 
\frac1{\eps^2}+ \frac{1}{\delta^2} + \frac1{\delta\eps}\lesssim \frac1{\delta^2}
	\,.
	\end{split}
\end{equation}
Analogously by \eqref{der:ux} it follows
\begin{equation}\label{der:uxy}
	\partial_{\theta r}	\xi^L= \psi_\delta(r)f^L(r)\partial_{\theta r}\hat \xi^L+ \dot\psi_\delta(r)f^L(r)\partial_{\theta}\hat \xi^L+
	\psi_\delta(r)\dot f^L(r)\partial_{\theta}\hat \xi^L\,,
\end{equation}
and
\begin{equation}\label{der:uxyy}
\partial_{\theta\theta r}	\xi^L= \psi_\delta(r)f^L(r)\partial_{\theta\theta r}\hat \xi^L+ \dot\psi_\delta(r)f^L(r)\partial_{\theta\theta}\hat \xi^L+
	\psi_\delta(r)\dot f^L(r)\partial_{\theta\theta}\hat \xi^L\,.
\end{equation}
Therefore, recalling  \eqref{eq:uxx-uxy-uxyy}, \eqref{eq:f(x)}, \eqref{eq:uyy+ux} and \eqref{eq:f'(x)-f''(x)} we find
\begin{equation*}\label{ub:est-yx}
	\begin{split}
		\int_{\Rin}^{R_0-\delta}\fint_{-\pi L_0}^{\pi L_0}(\partial_{\theta r}\xi^L)^2\di\theta r\di r&\lesssim
		\int_{\Rin}^{R_0-\delta}\fint_{-\pi L_0}^{\pi L_0}(f^L(r))^2(\partial_{\theta r}\hat \xi^L)^2\di\theta r\di r
		+
		\frac{1}{\delta^2}
		\int_{R_0-2\delta}^{R_0-\delta}\fint_{-\pi L_0}^{\pi L_0}(\partial_{\theta}\hat\xi^L)^2\di\theta r\di r
		\\&
		+
		\int_{\Rin}^{R_0-\delta}\fint_{-\pi L_0}^{\pi L_0}(\dot f(r))^2(\partial_{\theta}\hat\xi^L)^2\di\theta r\di r
		\lesssim\frac1{\eps^2} + \frac1{\delta^2}\frac{\delta}{\eps}+ \frac1{\delta\eps} \lesssim\frac1{\delta\eps}\,.
	\end{split} 
\end{equation*}
and in a similar way 
\begin{equation*}\label{ub:est-yyx}
	\begin{split}
\int_{\Rin}^{R_0-\delta}\fint_{-\pi L_0}^{\pi L_0}(\partial_{\theta\theta r}\xi^L)^2\di\theta r\di r
 \lesssim\frac1{\delta\eps}\,.
	\end{split} 
\end{equation*}
Thus combining \eqref{ub:est00} with \eqref{uy+uyy+ux}, \eqref{uxx+uyx} we have 
\begin{equation}\label{ub:est001}
	\begin{split}
		\frac 1{L^2}\strokedint_{-\pi L_0}^{\pi L_0}\int_{\Rin}^{R_0-\delta}\left[
		\int_0^\theta  \int_0^{\hat\theta}  \left(
		\frac{2	\partial_{\theta }\xi^L\partial_{\theta r}\xi^L}r
		-
		\frac{(\partial_{\theta}\xi^L)^2}{r^2}
		\right) 
		\di\tilde\theta \di\hat\theta\right]^2
		r \di r\di\theta
		&\lesssim \frac{L_0^4}{L^2}\frac1{\delta\eps}
		\,.
	\end{split}
\end{equation}
Analogously \eqref{ub:est0} together with \eqref{uy+uyy+ux}--\eqref{uyyx} imply
\begin{equation}\label{ub:est01}
	\begin{split}
		\frac 1{L^2}\strokedint_{-\pi L_0}^{\pi L_0}\int_{\Rin}^{R_0-\delta}\left[
		\int_0^\theta  \int_0^{\hat\theta}  \left(\partial_{\theta\theta }\xi^L\partial_{rr}\xi^L+\partial_{\theta\theta r}\xi^L\partial_{r}\xi^L
		\right) 
		\di\tilde\theta \di\hat\theta\right]^2
		r \di r\di\theta
		&
	 \lesssim \frac {L_0^4}{L^2} \frac1{\delta^2\eps}
		\,.
	\end{split}
\end{equation}
Gathering together 
\eqref{ub:est6}, \eqref{ub:est7}, \eqref{ub:est01}  and \eqref{ub:est001} we infer
\begin{equation*}
	L^2	\strokedint_{-\pi L_0}^{\pi L_0}\int_{\Rin}^{\Rout}
\bigl( \partial_ru_r^L-\dot u^\star\bigr)^2 r\di r\di \theta\lesssim  \frac {L_0^4}{L^2} \frac1{\delta^2\eps}\,,
\end{equation*}
which together with \eqref{eq:est-ux^4} and \eqref{ub:est1} implies 
\begin{equation*}
	\begin{split}
		L^2	\strokedint_{-\pi L_0}^{\pi L_0}\int_{\Rin}^{\Rout}
	\bigl( \partial_ru_r^L+\frac{(\partial_r\xi^L)^2}{2L^2}-\dot u^\star\bigr)^2 r\di r\di \theta
		\lesssim \frac{L_0^2}{L^2}\left(\frac1\delta+\frac1{\eps^2}\right)+
		 \frac {L_0^4}{L^2} \frac1{\delta^2\eps}\lesssim  \frac {L_0^4}{L^2} \frac1{\delta^2\eps}
	\,.
	\end{split}
\end{equation*}

\noindent
\textit{Step 3-(iv): Estimate of the third term.} we show that
\begin{equation}\label{ub:step4}
L^2 \strokedint_{-\pi L_0}^{\pi L_0} \int_{\Rin}^{\Rout} \Big (\frac{u_r^L}{r}+\frac{\partial_\theta u_\theta^L}{r} +\frac{(\partial_\theta\xi^L)^2}{2r^2} -\frac{(u^\star)_+}{r} \Big)^2 r\di r\di\theta \lesssim\frac{L_0^4}{L^2}\frac1{\delta\eps}
\lesssim\frac{L_0^4}{L^2}\frac1{\delta^2 M}
\,.
\end{equation}
Recalling the definition of $u_r^L$ and $u_\theta^L$ it holds
\begin{equation*}
\begin{split}
\frac{u_r^L}{r}+
\frac{	\partial_\theta u_\theta^L}{r}+\frac{(\partial_\theta\xi^L)^2}{2r^2} -\frac{(u^\star)_+}{r} &= \frac{u^\star}{r}(1-\psi_\delta^2(r))\chi_{(R_0-2\delta,R_0)}\\&-\frac{1}{L^2}\int_0^\theta(r\partial_ru_\theta^L-u_\theta^L+\partial_r\xi^L\partial_\theta\xi^L)\di\hat\theta\,.
\end{split}
\end{equation*}
Therefore 
\begin{equation*}
\begin{split}
L^2 \strokedint_{-\pi L_0}^{\pi L_0} \int_{\Rin}^{\Rout} &\Big (\frac{u_r^L}{r}+\frac{\partial_\theta u_\theta^L}{r} +\frac{(\partial_\theta\xi^L)^2}{2r^2} -\frac{(u^\star)_+}{r} \Big)^2 r\di r\di\theta\\&\le L^2 \int_{R_0-2\delta}^{R_0}
 \frac{(u^\star)^2}{r}(1-
 \psi_\delta^2(r)  )^2\di r \\&+ \frac1{L^2}\fint_{-\pi L_0}^{\pi L_0}\int_{\Rin}^{\Rout}\Big(\int_{0}^{\theta}(r\partial_ru_\theta^L-u_\theta^L+\partial_r\xi^L\partial_\theta\xi^L)\di\hat\theta
 \Big)^2r\di r\di \theta
 \,.
\end{split}
\end{equation*}
For the first term on the right hand side we simply have 
\begin{equation*}
\begin{split}
 L^2 \int_{R_0-2\delta}^{R_0}
\frac{(u^\star)^2}{r}(1-
\psi_\delta^2(r)  )^2\di r\le L^2\int_{R_0-2\delta}^{R_0}  \frac{(u^\star)^2}{r}\di r \lesssim L^2\int_{R_0-2\delta}^{R_0} (r-R_0)^2\di r\lesssim L^2 \delta^3 \,.
\end{split}
\end{equation*}
Recalling \eqref{ub:est5}  and \eqref{stima-algebrica} the second term  satisfies  
\begin{equation}
\begin{split}
\frac1{L^2}&\fint_{-\pi L_0}^{\pi L_0}\int_{\Rin}^{\Rout}\Big(\int_{0}^{\theta}(r\partial_ru_\theta^L-u_\theta^L+\partial_r\xi^L\partial_\theta\xi^L)\di\hat\theta
\Big)^2r\di r\di \theta
\\&\lesssim \frac{L_0^4}{L^2} \int_{\Rin}^{R_0-\delta} \left[
\fint_{-\pi L_0}^{\pi L_0} (\partial_\theta\xi^L)^2\di\theta \fint_{-\pi L_0}^{\pi L_0} (\partial_\theta\xi^L)^2\di\theta + \fint_{-\pi L_0}^{\pi L_0} (\partial_{\theta\theta}\xi^L)^2\di\theta \fint_{-\pi L_0}^{\pi L_0} (\partial_r\xi^L)^2\di\theta
\right]r\di r\\&
+ \frac{L_0^4}{L^2} \int_{\Rin}^{R_0-\delta}( B^L(r))^2 r\di r  + \frac{L_0^2}{L^2} \int_{\Rin}^{R_0-\delta}( D^L(r))^2 r\di r 
\\
& \lesssim \frac{L_0^4}{L^2}\left(1+\frac1\delta+\frac{1}{\eps\delta}\right)
 \lesssim \frac{L_0^4}{L^2}\frac{1}{\eps\delta}\,,
\end{split}
\end{equation}
where the second inequality follow from \eqref{uy+uyy+ux} and arguing similarly to step 3-(iii).

\medskip

\noindent
\textit{Step 3-(v): Estimate of the sixth term.} From \eqref{uxx+uyx} we have 
\begin{equation}\label{ub:step6}
\frac{1}{L^2}
\strokedint_{-\pi L_0}^{\pi L_0} \int_{\Rin}^{\Rout} 
 \Big(\frac{(\partial_{rr}\xi)^2}{L^2}+\frac{2(\partial_{r\theta}\xi)^2}{r^2}\Big)	
 r \di r\di \theta
\lesssim \frac1{L^2}\left(\frac1{\delta^2}+\frac1{\delta\eps}\right) \lesssim  \frac1{L^2}\frac1{\delta^2}
\,.
\end{equation}
\medskip

\noindent
\textit{Conclusions.}  By Step 2 we have that $(u_r^L,u_\theta^L,\xi^L)$ converges to $\mu$ in the sense of Definition \ref{def:convergence}.
Moreover  by collecting the estimates showed in Steps 3-(i)--3-(v), i.e., \eqref{ub:step1.0}, \eqref{ub:step1.1}
 \eqref{ub:step2}, \eqref{ub:step3}, \eqref{ub:step4} and \eqref{ub:step6} we find
\begin{equation}\label{the-end0}
	\begin{split}
\limsup_{L\to+\infty}L^2(\mathcal E_L(w^L,u^L)-\mathcal E_0)&\le \mathcal{F}_\infty(\mu)+ C\lim_{L\to +\infty}
\omega(2M^2\delta)\log M\\
&+C\lim_{L\to +\infty} \Big(\frac{L_0^4}{L^2}\frac1{\delta^3M}+\frac{L_0^4}{L^2}\frac1{\delta^2M} + \frac{1}{L^2\delta^2}
\Big)\,.
	\end{split}
\end{equation}
We now properly choose the parameters. We start by noticing  that for every $\overline M\in\mathbb N$ there exists $L_{\overline M}>\overline M^{4}$ such that
\begin{equation}\label{cond:M}
 \omega(2\overline M^2L^{-2/3})\le {\overline M}^{-1}\quad\forall L\ge L_{\overline M}\,.
\end{equation}
Since $L_{\overline M+1}\ge L_{\overline M}$ we set
$$ M=M(L):=\overline M\quad \text{ if } L\in[L_{\overline M},L_{\overline M+1})$$
Next we define
\begin{equation*}
	\delta:=\frac{\eps}{M}= L^{-2/3}M^{-1/8}\,\iff\, \eps= L^{-2/3}M^{7/8}\,,
\end{equation*}
and we choose 
\begin{equation*}
	L_0:=\frac Ln\in [M^{1/8},2M^{1/8})\,\iff\, n\in \Big[\frac{L}{2M^{1/8}},  \frac L{M^{1/8}} \Big)\,.
\end{equation*}
These choices ensures the validity of \eqref{cond_parameters1} and \eqref{cond_parameters2}. Indeed we have 
\begin{equation*}
	\eps= \frac1{L^{2/3}M^{-7/8}}= \frac1{L^{2/3}\overline M^{-7/8}}\le 
	\frac1{\overline M^{8/3}\overline M^{-7/8}}
	\quad\text{ if }L\in[L_{\overline M},L_{\overline M+1})\,,
\end{equation*}
where the last inequality follows from the fact that $L\ge L_{\overline M}\ge \overline M^{4}$. Hence $\eps\to0$ and $\delta\to0$ as $L\to+\infty$. In a similar way we have $M^2\delta=M\eps\to0$ and $L_0,n\to+\infty$ as $L\to+\infty$.\\
Recalling that $\omega$ is monotone we find
$$ \omega(2M^2\delta)=\omega(2M^2L^{-2/3}M^{-1/8})\le \omega(2M^2L^{-2/3})= \omega(2\overline M^2L^{-2/3}) \quad\text{ if }L\in[L_{\overline M},L_{\overline M+1})\,,$$ 
which together with \eqref{cond:M} imply 
\begin{equation}\label{the-end1}	
	\log M\omega(2M^2\delta)\le \log( \overline M)\overline M^{-1}\quad \text{ if } L\in[L_{\overline M},L_{\overline M+1})\,.
\end{equation}
Moreover  if  $L\in[L_{\overline M},L_{\overline M+1})$ it holds
\begin{equation}\label{the-end2}
L^2\delta^3= L^2L^{-2}M^{-3/8}=\overline M^{-3/8}\,,
\end{equation}
and
\begin{equation}\label{the-end4}
	 \frac{L_0^4}{L^2}\frac1{\delta^2M}  \le
 \frac{L_0^4}{L^2}\frac1{\delta^3M}  \le \frac{16M^{1/2}}{M M^{-3/8}}=16M^{-1/8}= 16\overline M^{-1/8} \,,
\end{equation}
 Eventually collecting \eqref{the-end0}--\eqref{the-end4} we infer 

\begin{equation*}
	\begin{split}
			\limsup_{L\to+\infty}\mathcal{F}_L(u_r^L,u_\theta^L,\xi^L)=
		\limsup_{L\to+\infty}L^2(\mathcal E_L(u_r^L,u_\theta^L,\xi^L)-\mathcal E_0)\le \mathcal{F}_\infty(\mu)\,.
	\end{split}
\end{equation*}
\end{proof}
\section{Existence and regularity of minimizers of $\mathcal{F}_\infty$}\label{sec:regularity}

In this section, we address the existence of minimizers for the limiting functional $\mathcal{F}_\infty$ and discuss some of their qualitative properties, including energy equipartition.

More precisely, the results of this section show, in particular, that minimizers of $\mathcal F_\infty$
admit a one-dimensional structure after disintegration in the k-variable. More precisely, almost every fiber is described by a BV density in the radial variable, and the two contributions appearing in the functional balance exactly through an equipartition identity. Moreover, condition \eqref{small-k} shows that minimizers do not carry mass on sufficiently small frequencies.

To this end, we introduce the notion of disintegration of measures with respect to the $k$-variable, which differs slightly from the disintegration with respect to the $r$-variable introduced in Section~\ref{sec:upb}.

The proofs of the results stated in this section are omitted, since they follow from minor modifications of the arguments in Section~6 of \cite{BeMa25}, together with the equivalent formulation of the functional given in \eqref{def:F_infty2}\\

\noindent
For a given interval $I\subset\R$ we denote by $L^0(I)$ the space of Lebesgue measurable functions $g\colon I\to \R$. Moreover, we let $\pi_2\colon I\times\R\to \R$ denote the canonical projection, and for any $\mu\in\mathcal{M}_b(I\times\R)$ we write $(\pi_2)_\sharp\mu\in \mathcal{M}_b^+(\R)$ for the push-forward of $\mu$ through $\pi_2$.

\begin{definition}[Disintegration of measures in the $k$-variable]\label{def:disintegration-in-k}
	Let $I\subset\R$ be an interval and let $\mu\in\mathcal{M}_b(I\times\R)$. We say that the family
	\begin{equation*}
		{\big(\lambda\,,\,(g_k)_{k\in\R}\big)}
		\quad\text{with}\quad
		\lambda\in \mathcal{M}_b(\R)
		\quad\text{and}\quad
		g_k\in L^0(I)\ \ \forall k\in \R\,,
	\end{equation*}
	is a disintegration of $\mu$ \(($in the $k$-variable$)\) if $k\mapsto g_k$ is $\lambda$-measurable, $\int_I g_k(r)\di r=1$ for $\lambda$-a.e. $k\in\R$, and
	\begin{equation}\label{eq:disint-in-k}
		\int_{I\times\R}f(r,k)\di\mu
		=
		\int_\R\int_I f(r,k)g_k(r)\di r\di\lambda(k)\,,
	\end{equation}
	for every $f\in L^1(I\times\R;|\mu|)$.
\end{definition}

With this definition in place, we can now state the main result of this section.

\begin{theorem}[Minimizers of $\mathcal{F}_\infty$]\label{thm:minimizers}
	Let $\mathcal{M}_\infty$ and $\mathcal{F}_\infty$ be as in \eqref{def:limit_measures} and \eqref{def:F_infty}, respectively. Then there exists $\hat\mu\in\mathcal{M}_\infty$ such that
	\begin{equation*}
		\mathcal{F}_\infty(\hat\mu)
		=
		\inf_{\mu\in\mathcal{M}_\infty}\mathcal{F}_\infty(\mu)\,.
	\end{equation*}
	Moreover, every minimizer $\hat\mu$ satisfies the following properties: there exist a constant $C>0$ and a $(\pi_2)_\sharp\hat\mu$-measurable map $k\mapsto g_k$, with $g_k\in BV(\Rin,R_0)$ for $(\pi_2)_\sharp\hat\mu$-a.e. $k\in\R$, such that
	\[
	{\big((\pi_2)_\sharp\hat\mu\,,\, (g_k)_{k\in\R}\big)}
	\quad \text{is a disintegration of }\hat\mu\,,
	\]
	and
	\begin{equation}\label{equipart-section}
		\int^{R_0}_{\Rin}\frac{k^2}{r^3}g_k(r)\di r
		=
		\int_{\Rin}^{R_0}
		\frac{\dot u^\star(r)r}{4k^2}
		\left(\frac{\di\hat\mu_{,r}}{\di\hat\mu}\right)^2
		g_k(r)\di r
		\quad
		\text{for } (\pi_2)_\sharp\hat\mu\text{-a.e. }k\in\R\,,
	\end{equation}
	and
	\begin{equation}\label{small-k}
		(\pi_2)_\sharp\hat\mu\Big(\big\{|k|<C\big\}\Big)=0\,.
	\end{equation}
\end{theorem}
\begin{cor}[Equipartition of the energy]
	Let $\hat\mu\in\mathcal{M}_\infty$ be a minimizer of $\mathcal{F}_\infty$. Then
	\begin{equation*}
		\int_{(\Rin,R_0)\times\R}\frac{k^2}{r^3}\di\hat\mu
		=
		\int_{(\Rin,R_0)\times\R}
		\frac{\dot u^\star(r)r}{4k^2}
		\Big(\frac{\di\hat\mu_{,r}}{\di\hat\mu}\Big)^2
		\di\hat\mu\,.
	\end{equation*}
\end{cor}

\section*{Acknowledgements}

R.~Marziani has received funding from the European Union’s Horizon research and innovation programme under the Marie Skłodowska-Curie Grant Agreement No.~101150549. 

The research leading to this paper was partially supported by the Gruppo Nazionale per l’Analisi Matematica, la Probabilità e le loro Applicazioni (GNAMPA) of the Istituto Nazionale di Alta Matematica (INdAM), through funding for a visiting period at the Courant Institute in 2025, and through the INdAM--GNAMPA projects 2025 CUP E53C24001950001, and 2026 CUP E53C25002010001.

 Views
and opinions expressed are however those of the authors only and do not necessarily reflect those
of the European Union or the European Research Executive Agency. Neither the European Union
nor the granting authority can be held responsible for them.

%

\end{document}